\documentclass[leqno,11pt]{article}

\usepackage{microtype}
\usepackage[english]{babel}
\usepackage{amsmath,amsthm,amssymb,mathrsfs} 
\usepackage{times}
\usepackage[pagebackref,hypertexnames=false]{hyperref}

\usepackage{fancybox,fancyhdr,graphics,epsfig}
\usepackage[usenames,dvipsnames]{color}
\usepackage{bbm}
\usepackage{subcaption}
\usepackage{graphicx}
\usepackage{setspace}

\usepackage{verbatim}
\usepackage{amsmath,amsthm,bbm,amssymb}
\usepackage[dvipsnames]{xcolor}

\DeclareMathOperator{\rank}{rank}

\DeclareMathOperator{\im}{Im}

\DeclareMathOperator{\vol}{Vol}
\DeclareMathOperator{\logg}{\log\log}

\DeclareMathOperator*{\argmin}{argmin}

\def\N{\mathbb{N}}
\def\Q{\mathbb{Q}}
\def\R{\mathbb{R}}
\def\F{\mathbb{F}}
\def\Z{\mathbb{Z}}
\def\P{\mathbb{P}}

\def\T{\mathbb{T}}

\def\S{\mathbb{S}}

\def\cA{\mathcal{A}}

\def\cC{\mathcal{C}}

\def\cH{\mathcal{H}}

\def\cL{\mathcal{L}}

\def\cP{\mathcal{P}}

\def\cR{\mathcal{R}}

\def\cS{\mathcal{S}}

\def\cX{\mathcal{X}}
\def\cY{\mathcal{Y}}

\def\cL{\mathcal{L}}

\newcommand{\E}{\mathbb{E}} 

\newcommand{\given}{\;|\;}
\newcommand{\mean}[1] {\E\left\{{#1}\right\}}

\newcommand{\meanx}[1] {\E\{{#1}\}}
\newcommand{\cmean}[2] {\E\left\{#1\given #2\right\}}

\newcommand{\ind}{\boldsymbol{\mathbbm{1}}} 
\newcommand{\indf}[1]{\ind\set{#1}} 

\newcommand{\var}[1]{\mathrm{Var}\param{{#1}}}

\DeclareMathOperator{\Gr}{Gr}

\newcommand{\set}[1]{\left\{#1\right\}}

\newcommand{\param}[1]{\left(#1\right)}
\newcommand{\abs}[1] {\left| {#1}\right|}
\newcommand{\floor}[1] {\left\lfloor{#1}\right\rfloor}

\newcommand{\prob}[1]{\mathbb{P}\left(#1\right)}

\newcommand{\cprob}[2]{\mathbb{P}\left(#1\given #2\right)} 

\newcommand{\eps}{\epsilon}

\newcommand{\bx}{\mathbf{x}}

\def\Cn{C^{^\bullet}}
\def\Cp{C^{^\circ}}

\newcommand{\rmax}{r_{\max}}

\def\bvphi{\boldsymbol\varphi}

\newcommand{\cB}{{\cal{B}}}

\def\bth{\boldsymbol\theta}
\providecommand{\setthms}[1]{#1}
\setthms{
\newtheorem{lem}{Lemma}[section]
\newtheorem{thm}[lem]{Theorem}
\newtheorem{prop}[lem]{Proposition}
\newtheorem{cor}[lem]{Corollary}

\newtheorem{rem}[lem]{Remark}

\theoremstyle{definition}
\newtheorem{defn}[lem]{Definition}

}
\newcommand{\cech}{\v{C}ech }

\newcommand{\erdren}{Erd\H{o}s-R\'enyi }

\newcommand{\iid}{\mathrm{i.i.d.}}
\newcommand{\whp}{{w.h.p.}}

\newcommand{\ninf}{n\to\infty}
\newcommand{\pois}[1]{\mathrm{Poisson}\param{{#1}}}

\newcommand{\limninf}{\lim_{\ninf}}

\newcommand{\dtv}[2]{d_{\mathrm{TV}}\param{{#1},{#2}}}
\newcommand{\bs}{\backslash}

\def\const {C^*}

\numberwithin{equation}{section}

\def\ems{\emptyset}

\def\bsplit#1\esplit{\begin{split} #1 \end{split} }
\def\splitb#1\splite{\begin{split} #1 \end{split} }
\def\beq#1\eeq{\begin{equation} #1 \end{equation}}
\def\eqb#1\eqe{\begin{equation} #1 \end{equation}}

\def\vcap{V_{\scap}}
\def\vdiff{V_{\diff}}
\def\vint{V_{\mathrm{int}}}
\def\kint{D_{\mathrm{int}}}
\def\vsimp{V_{\simp}}
\def\vuni{V_{\mathrm{uni}}}
\def\gcrit{g_{\mathrm{crit}}}
\def\hcrit{h_{\mathrm{crit}}}
\def\hsep{h_{\mathrm{sep}}}
\def\phmin{\phi_{\min}}
\def\cmin{c_{\min}}
\def\rhmin{\rho_{\min}}
\def\hrho{\hat\rho}

\def\Abp{{D}_{\mathrm{bp}}}

\def\hcX{\hat\cX}
\def\hcXmin{\hcX_{\min}}
\def\hphi{\hat\phi}

\def\bC{\bar C}
\def\bZ{\bar Z}
\def\bB{\bar B}

\def\Dn{D}
\def\gn{g}

\def\Cn{F^{\scriptscriptstyle{\bullet}}}
\def\Cp{F^{\scriptscriptstyle{\circ}}}
\def\Nn{N^{\scriptscriptstyle{\bullet}}}
\def\Np{N^{\scriptscriptstyle{\circ}}}
\def\Rn{\cR^{\scriptscriptstyle{\bullet}}}
\def\Rp{\cR^{\scriptscriptstyle{\circ}}}

\DeclareFontFamily{U}{FdSymbolC}{}
\DeclareFontShape{U}{FdSymbolC}{m}{n}{<-> s * FdSymbolC-Book}{}
\DeclareSymbolFont{fdarrows}{U}{FdSymbolC}{m}{n}
\DeclareMathSymbol{\origof}{\mathrel}{fdarrows}{"73}

\DeclareFontFamily{U}{FdSymbolD}{}
\DeclareFontShape{U}{FdSymbolD}{m}{n}{<-> s * FdSymbolD-Book}{}
\DeclareSymbolFont{fdnarrows}{U}{FdSymbolD}{m}{n}
\DeclareMathSymbol{\norigof}{\mathrel}{fdnarrows}{"73}

\DeclareMathOperator{\diff}{diff}
\DeclareMathOperator{\simp}{simp}
\DeclareMathOperator{\scap}{cap}

\def\Cnp{F^{\origof}}

\def\Cnnp{F^{\norigof}}

\def\Scrit{S_k^{\mathrm{crit}}}
\def\Vcrit{V_k^{\mathrm{crit}}}
\def\hbth{\hat\bth}
\usepackage{accents}
\def\thres{\log n + (k-1)\logg n}
\newlength{\dhatheight}

\title{\vspace{-50pt}
Homological Connectivity in Random \cech Complexes}
\author{Omer Bobrowski\footnote{omer@ee.technion.ac.il}\\
\small Technion - Israel Institute of Technology
\vspace{-20pt}}

\date{}

\begin{document}

\maketitle

{\abstract{

We study the homology of random \cech complexes generated by a homogeneous Poisson process. We  focus on  `homological connectivity' -- the stage where the random complex is dense enough, so that its homology ``stabilizes" and becomes isomorphic to that of the underlying topological space. 
Our results form a comprehensive high-dimensional analogue of well-known phenomena related to connectivity in the \erdren graph and random geometric graphs.
We first prove that there is a sharp phase transition describing homological connectivity.
Next, we analyze the behavior of the complex in the critical window. We show that the cycles obstructing homological connectivity have a very unique and simple shape. In addition, we prove that the process counting the last obstructions converges to a Poisson process.
We make a heavy use of Morse theory, and its adaptation to distance functions. 
In particular, our results classify the critical points of random distance functions according to their exact effect on homology.
}}

{
\tableofcontents
}

\section{Introduction}

One of the first, and  most widely-known results on random graphs is the connectivity of the so called \emph{\erdren random graph}. Let $G(n,p)$ be graph generated on $n$ labelled vertices, where each edge is added independently and with probability $p = p(n)$. The main results in \cite{erdos_random_1959} can be roughly divided into four  statements.\footnote{We note that the exact model studied in \cite{erdos_random_1959} is slightly different than the $G(n,p)$. However: (a) the $G(n,p)$ model is more commonly studied, and (b) it is has a similar limiting behavior as the fixed-size model in \cite{erdos_random_1959}.}
\begin{enumerate}
\item There is a sharp phase transition for connectivity at $p = \frac{\log n}{n}$, i.e.~
\eqb\label{eq:ER_1}
	\limninf\prob{G(n,p)\text{ is connected} }= \begin{cases} 1 & p = \frac{\log n+w(n)}{n},\\
	0 & p = \frac{\log n-w(n)}{n}  ,\end{cases}
\eqe
for any $w(n) \to \infty$, $w(n) = o(\log n)$. 
\item Denote by $N_{\mathrm{comp}}$ the number of connected components in $G(n,p)$, and by $N_{\mathrm{iso}}$ the number of isolated vertices in $G(n,p)$.
When $p$ is in the critical range i.e.~$p = \frac{\log n + \lambda}{n} $ for some $\lambda\in \R$, then with high probability (\whp) we have
\eqb\label{eq:ER_2}
	N_{\mathrm{comp}} = N_{\mathrm{iso}}+1.
\eqe
This implies that 
the random $G(n,p)$ graph consists of a giant connected component and isolated points only. In other words, the ``obstructions" to connectivity are isolated points, and once they vanish (i.e.~merge with the giant component), the graph is connected.
\item For $p= \frac{\log n+\lambda}{n}$,   studying the limiting distribution of $N_{\text{iso}}$ leads to
\eqb\label{eq:ER_3}
	(N_{\mathrm{comp}}-1) \xrightarrow{\cL} \pois{e^{-\lambda}},
\eqe
where $\xrightarrow{\cL}$ stands for convergence of the probability law.
\item Combining \eqref{eq:ER_2} and \eqref{eq:ER_3}, if $p= \frac{\log n+\lambda}{n}$ then
\eqb\label{eq:ER_4}
	\limninf\prob{G(n,p)\text{ is connected}} = e^{-e^{-\lambda}}.
\eqe
\end{enumerate}
Clearly, the existence of isolated vertices prevents connectivity. The challenging part of the proof in \cite{erdos_random_1959} is to show that the vanishing of the isolated vertices indeed implies connectivity.

There is an alternative way to phrase statements 2-4, which sometimes provides additional insight into the behavior of the random graph. Instead of considering $G(n,p)$ as a fixed graph with $n$ vertices and probability $p$, one can consider it as an increasing sequence of graphs indexed by $p$ (constructing such a sequence can be done by attaching uniform $\iid$ ``clocks" to the edges, and adding each edge when its corresponding clock goes off). 
With this sequence in mind, we can define the following random ``hitting-times",
\eqb\label{eq:st_graph}
\splitb
	T_0^{\text{iso}}&:= \inf \set{p : G\text { has no isolated vertices}},\\
	T_0 &:= \inf \set{p: G \text{ is connected}}.
\splite
\eqe
Then \eqref{eq:ER_2} implies that \whp~we have $T_0^{\text{iso}} = T_0$. In addition, if we define $T_0' = \exp(-nT_0 +\log n)$,  then \eqref{eq:ER_4} is equivalent to $T'_0\xrightarrow{\cL} \mathrm{Exponential(1)}$. Note also that if we consider $G(n,p)$ as a weighted graph (where  weights = clock values), then $T_0$ equals the largest edge-weight in the minimal spanning tree (MST), while $T_0^{\text{iso}}$ is the largest weight in the nearest neighbor graph (NNG). The above equality holds not only for the largest edges, but for the entire sequence of edges in the MST/NNG around the connectivity threshold (see e.g.~\cite{skraba_randomly_2017}). It is also shown that this sequences converges to a homogenous Poisson process. We will return to these ideas later.

It turns out that the behavior described in \eqref{eq:ER_1}-\eqref{eq:ER_4} is not unique to the $G(n,p)$ model.
In \cite{linial_homological_2006,meshulam_homological_2009} the Linial-Meshulam (LM) random $d$-complex $Y_d(n,p)$ was introduced where 
one starts with the full $(d-1)$-skeleton on $n$ vertices, and adds $d$-dimensional faces independently with probability $p$.
The observation in \cite{linial_homological_2006} was that connectivity in $G(n,p)$ is the same as having a trivial zeroth homology $H_0$.
Thus, for the complex $Y_d(n,p)$ one may  ask analogous questions about having a trivial $(d-1)$-th homology  $H_{d-1}$, a property that \cite{linial_homological_2006} referred to as `homological connectivity'. 
The phase transition proved in \cite{linial_homological_2006,meshulam_homological_2009} is a  generalization of \eqref{eq:ER_1},
\[
	\limninf \prob{H_{d-1}(Y_d(n,p),\F) = 0} = \begin{cases} 1 & p = \frac{d\log n + w(n)}{n},\\
	0 & p = \frac{d\log n - w(n)}{n}.\end{cases}
\]
where $\F$ is a finite field.
Note that taking $d=1$ recovers the $G(n,p)$ result. Somewhat surprisingly, it turns out that, similarly to the $G(n,p)$ model, the obstructions to homological connectivity are ``isolated" or ``uncovered" $(d-1)$-faces. By that we refer to $(d-1)$-faces that are not  included in any $d$-face. It is straightforward to show that in the LM model, the existence of  isolated $(d-1)$-faces implies the existence of  non-trivial $(d-1)$-cycles (or cocycles, as used in \cite{linial_homological_2006,meshulam_homological_2009}), implying that $H_{d-1} \ne 0$. As in the $G(n,p)$ model, the more challenging part of the proof is to show that the opposite also holds. These statements were also proven for homology with integer coefficients more recently \cite{luczak_integral_2018,newman_integer_2018}.

In \cite{kahle_inside_2016} the equivalent statements to \eqref{eq:ER_2}-\eqref{eq:ER_4} were generalized as well. In particular, they showed that when $p = \frac{d\log n + \lambda}{n}$, then \whp\ the Betti number $\beta_{d-1} = \rank(H_{d-1})$  is equal to the number of isolated $(d-1)$-faces, and consequently that
\[
	\beta_{d-1}(Y_d(n,p)) \xrightarrow{\cL} \pois{e^{-\lambda}/d!}, 
\]
implying that
\[
	\limninf\prob{H_{d-1}(Y_d(n,p),\Z_2) = 0} = e^{-e^{-\lambda}/d!}.
\]
Thus, the entire set of statements \eqref{eq:ER_1}-\eqref{eq:ER_4} is generalized by the $Y_d(n,p)$ complex. We can also define hitting times analogous to those in \eqref{eq:st_graph},
\eqb\label{eq:st_complex}
\splitb
T_{d-1}^{\text{iso}} &:= \inf\set{p : Y_d \text{ has no isolated $(d-1)$-faces}},\\
T_{d-1} &:= \inf\set{p : H_{d-1}(Y_d) = 0}.
\splite
\eqe
For $d=2$, the work in \cite{kahle_inside_2016}   showed that \whp~$T_{d-1}^{\text{iso}} = T_{d-1}$.
The MST interpretation above, also has a higher-dimensional analogue via the notion of  minimal spanning acycles (MSA), defined in \cite{hiraoka_minimum_2017}.  Similarly to the MST of the $G(n,p)$, in \cite{skraba_randomly_2017} it was shown that the sequence of largest $d$-faces in the MSA converges to a Poisson process (after proper normalization).

Another interesting combinatorial model is the random \emph{flag} (or \emph{clique}) complex $X(n,p)$, where one takes the random graph $G(n,p)$ and adds a $k$-simplex for every clique of $(k+1)$ vertices.
In \cite{kahle_sharp_2014} it was proved that there exists a sequence of  phase transitions, such that for any fix $\eps >0$,
\[
	\limninf \prob{H_k(X(n,p),\Q) = 0} = \begin{cases} 1 & p = \param{\frac{(\frac{k}{2}+1+\eps)\log n}{n}}^{1/(k+1)},\\
	0 & p = \param{\frac{(\frac{k}{2}+1-\eps)\log n}{n}}^{1/(k+1)}.
	\end{cases}
\]
The connection between isolated faces and homological connectivity in this case was made via  Garland's method \cite{garland_p-adic_1973}. Brielfy, Garland's method provides a sufficient spectral condition for a simplicial complex to have a trivial $k$-th homology, given that it has no isolated $k$-faces.
In \cite{kahle_sharp_2014} it was conjectured that analogous statements to \eqref{eq:ER_2}-\eqref{eq:ER_4} should exist for this model as well. In other words, the only obstructions to homological connectivity are isolated faces, and consequently a Poisson limit for the Betti numbers should exist. However, to date it remains an open problem.
Perhaps the most significant difference between the random flag complex and the LM complex is that the homology of the former is non-monotone. In $Y_d(n,p)$ (as well as in $G(n,p)$), we start with the largest $H_{d-1}$ possible (all $(d-1)$-faces, no $d$-faces), and then adding $d$-faces (by increasing $p$) can only remove existing generators from $H_{d-1}$. This is not the case for $X(n,p)$ where increasing $p$ may also introduce new cycles to homology in all degrees. This difference makes the analysis significantly more complicated for the flag complex, and is also the main challenge we will be facing in this paper. We note that for $H_1$ a similar result was proven in \cite{demarco_triangle_2012} for $\Z_2$ coefficients. In addition, similar methods to \cite{kahle_sharp_2014} were used to study a more general model \cite{costa_large_2016,costa_homological_2015,fowler_generalized_2015}, that includes the LM and flag complexes as special cases, and the connection between homological connectivity and isolated faces was observed there as well.

The phenomena we have described so far are not unique to combinatorial models.
In the \emph{random geometric graph} $G(n,r)$  \cite{gilbert_random_1961, penrose_random_2003},
 one starts with a random set of  $n$ points in a $d$-dimensional metric space (assumed to be connected), and places an edge whenever the distance between two vertices is less than $r = r(n)$. The phase transition here was proved in \cite{penrose_longest_1997}, and is of the form
\eqb\label{eq:pt_rgg}
	\limninf \prob{G(n,r)\text{ is connected}} = \begin{cases} 1 & r = C\param{\frac{\log n + w(n)}{n}}^{1/d},\\
	0 & r =C\param{\frac{\log n - w(n)}{n}}^{1/d}.
	\end{cases}
\eqe
where  $C>0$ depends on the  space  in which the points are generated  and the probability measure on it. For the case where the points are generated on the flat torus (as in this paper), similar statements to \eqref{eq:ER_2}-\eqref{eq:ER_4} were proven as well, along with the suitable hitting-times and MST statements (where the weights are the lengths of the edges).
 We note that similar results hold when the points are generated by a homogeneous Poisson process with rate $n$. 

In this paper, we study a higher-dimensional analogue of $G(n,r)$ known as the \emph{random \cech complex}.
The  study of random geometric complexes first appeared in \cite{robins_betti_2006}. The rigorous mathematical analysis of these complexes was initiated by Kahle \cite{kahle_random_2011}, and extended over the past decade in various directions (see \cite{bobrowski_topology_2018} for a survey of the field). 
In particular, aspects related to homological connectivity were studied in 
\cite{bobrowski_distance_2014,bobrowski_topology_2014,bobrowski_random_2019,bobrowski_vanishing_2017,iyer_thresholds_2018}, where the foundations for the results we present here were laid. 
In this paper we are finally ready to put all the pieces together and present phase transition results that  are a complete analogue of \eqref{eq:ER_1}-\eqref{eq:ER_4}. 

Let $M$ be a $d$-dimensional smooth compact Riemannian manifold.
Let $\cP\subset M$ be a finite subset. The \cech complex we consider here is generated with $\cP$ as its set of vertices, by fixing a radius $r>0$ and asserting that  $(k+1)$ points span a simplex if the intersection of $r$-balls  around them (in the Riemmannian metric) is non-empty, see Section \ref{sec:cech}.
The \emph{random \cech complex} $\cC_r$ we study here is when $\cP = \cP_n$ is a homogeneous spatial Poisson process on $M$ with intensity $n$ (see Section \ref{sec:poisson}).
Our goal is to study the homology of the random complex $\cC_r$ when $n\to\infty$ and $r = r(n)\to 0$.
An important quantity in our anaylsis will be
\eqb\label{eq:lambda}
    \Lambda := \omega_d nr^d, 
\eqe
where $\omega_d$ is the volume of a $d$-dimensional unit-ball in $\R^d$. Notice that $\Lambda$ represents the expected number of points of $\cP_n$ lying in a given ball of radius $r$. This quantity controls the expected degree of the faces in the complex. 

While homological connectivity was defined earlier as the point where homology becomes trivial, in our setting this definition requires an update. If $r$ is large enough, then the balls of radius $r$ around $\cP_n$ cover $M$ completely. The Nerve Lemma \ref{lem:nerve} (cf. \cite{borsuk_imbedding_1948}) then implies that $\cC_r$ is homotopy equivalent to $M$, and in particular $H_k(\cC_r)\cong H_k(M)$ for all $k$. Therefore, we should not expect the homology $H_k(\cC_r)$ to become trivial in general (unless $H_k(M)$ is itself trivial).

Note that the $1$-skeleton of $\cC_r$ (i.e.~the set of vertices and edges) is the random geometric graph $G(n,2r)$.
In this case, assuming that $M$ is connected, and has a unit-volume, then \eqref{eq:pt_rgg} holds with $C = \frac{1}{2}\omega_d^{-1/d}$.\footnote{To be precise, this was proven for the flat torus in \cite{penrose_longest_1997}, but similar analysis to \cite{bobrowski_topology_2018} can show it holds for a general $M$.}
We will therefore say that the threshold for connectivity of $\cC_r$ is when
$\Lambda = 2^{-d}\log n$.
For higher-dimensional homology, the best result known is
\eqb\label{eq:pt_hom}
	\limninf \prob{H_k(\cC_r) \cong H_k(M)} = \begin{cases} 1 & \Lambda = \log n + k\logg n + w(n),\\
	0 & \Lambda = \log n + (k-2)\logg n - w(n),
	\end{cases}
\eqe
for $1\le k \le d-1$, $w(n)\to\infty$. This result was proved first for the flat torus in \cite{bobrowski_vanishing_2017}, and later generalized to compact Riemmannian manifolds in \cite{bobrowski_topology_2018}.
While \eqref{eq:pt_hom} indicates a phase transition  where the homology of  $\cC_r$ identifies with that of the underlying manifold, it is not the full description we are seeking.
Firstly, the description in \eqref{eq:pt_hom} is not tight. It is not clear where exactly is the threshold (if one exists). Secondly, as we do not know where the exact transition occurs, we cannot analyze the obstructions as in \eqref{eq:ER_2}-\eqref{eq:ER_4}.
It turns out that the gap between \eqref{eq:pt_hom} and a full probabilisitic description for homological connectivity was much bigger than expected,  and this brings us to the present work.

A key observation that allowed us to make the  breakthrough in this paper, is the realization that the event $\set{H_k(\cC_r)\cong H_k(M)}$ does not capture well the notion of `homological connectivity'. 
The main reason is that for $k\ge 1$ the homology $H_k(\cC_r)$ does not behave in  monotone fashion, in the sense that increasing the radius $r$ might create new $k$-cycles as well as destroy existing ones. As a consequence, even if it is true that $H_k(\cC_r) \cong H_k(M)$ for a given $r$, there is no guarantee that the same holds for $r' > r$. In fact, we will show later  that for some values of $r$ such discrepancies occur with high probability.
Intuitively, we want to think of homological connectivity as the stage where homology ``stabilizes" and no more changes occur. Thus, we need to consider a slightly different event. Notice that in the LM model this is not an issue, since  homology there is  monotone.
Our definition of \emph{homological connectivity}  will thus be the occurrence of the event
\[
\cH_{k,r} := \set{H_k(\cC_s) \cong H_k(M)\ \ \forall s\ge r}.
\]
Note that $\cH_{k,r}$ is indeed monotone in $r$. Consequently, when $\cH_{k,r}$ occurs we know that homology stops changing, and thus $\cH_{k,r}$ properly captures the notion of homology stabilizing. 
Our main result will show that $\cH_{k,r}$ exhibits a sharp phase transition at $\Lambda = \thres$ (except for   $k=d-1$, for which the threshold is $\log n + (d-1)\logg n$). In addition, we will explore the behavior of the complex in the critical window, and provide detailed statements analogous to \eqref{eq:ER_2}-\eqref{eq:ER_4}.
In Section \ref{sec:results} we will describe these results in detail. In order to simplify our notation and calculations, we will focus here on the case where $M=\T^d$  i.e.~the $d$-dimensional flat torus (see Section \ref{sec:flat_torus}). However, the framework established in \cite{bobrowski_random_2019} can be used to show that similar results hold for a wide class of compact manifolds $M$.

The study here brings together stochastic geometry and algebraic topology.
The bridge between  these two fields is provided by \emph{Morse theory}, studying the connection between critical points and homology, and in particular its adaptation to distance functions \cite{ferry_when_1976,gershkovich_morse_1997}.  
A framework integrating Morse theory  into the analysis of the random \cech complexes was developed in \cite{bobrowski_distance_2014,bobrowski_topology_2018,bobrowski_vanishing_2017}.
To prove the results of this paper we will have to dig  deeper into this connection. 

\paragraph{Other related work.} Homological connectivity is one topological aspect of random complexes that can be studied. For the combinatorial models some of the other topics studied are the emergence of homology, collapsibility, the structure of the fundamental group, spectral and high-dimensional expansion properties, and more (e.g.\ \cite{aronshtam_threshold_2015,babson_fundamental_2011,costa_asphericity_2015,dotterrer_coboundary_2012,gundert_eigenvalues_2016,linial_phase_2016,kozlov_threshold_2010}).
The study of the geometric models includes questions about the Betti numbers, persistent homology, extremal behavior, topological types, and more (e.g.\ \cite{adler_crackle:_2014,auffinger_topologies_2018,bobrowski_maximally_2017,hiraoka_limit_2018,kahle_limit_2013,owada_limit_2017,yogeshwaran_random_2016}). Finally, while not addressing the \cech complex directly, the  results on ``topological learning"  of Niyogi, Smale, and Weinberger \cite{niyogi_finding_2008} certainly provided inspiration and insight for the current study.


\paragraph{Organization.} The rest of the paper is organized as follows. Section 2 provides a brief introduction to the main objects studied in this paper. In Section 3 we will provide a detailed  review on the main results in the paper. The proofs will be divided between Sections 4-8. We tried to organize it so that the more fundamental and intuitive parts of the proofs will appear earlier, and the more difficult and technical parts are saved for later. Finally, in Section 9 we provide some conclusions and open problems.

\section{Preliminaries}
In this section we briefly introduce the main objects we will be studying in this paper.

\subsection{Homology}

We provide here a brief and intuitive introduction to the topic of homology, for  readers who are unfamiliar with algebraic topology. This should be sufficient to understand most parts of this paper.
A comprehensive introduction to the topic can be found in \cite{hatcher_algebraic_2002,munkres_elements_1984}, for example.

Let $X$ be a topological space. The \textit{homology} of $X$ is a sequence of Abelian groups denoted  $\set{H_k(X)}_{k=0}^\infty$. 
In this paper, we will assume that homology is computed with field coefficients, so that the homology groups $H_k(X)$ are simply vector spaces. 
Each of these vector spaces captures different topological properties of $X$. Loosely speaking, the basis elements of $H_0(X)$ correspond the connected components of $X$, the basis of $H_1(X)$ corresponds to ``holes" or ``loops", and $H_2(X)$ corresponds to ``voids" or ``bubbles".
In higher dimensions, the basis of $H_k(X)$ represents `nontrivial $k$-cycles' in $X$, generalizing the idea of holes and voids.  A nontrivial $k$-cycle  can be thought of as a shape topologically similar (e.g.~homeomorphic) to a $k$-sphere. 
The dimension of these vector spaces (i.e. the number of linearly independent nontrivial $k$-cycles) are called the \emph{Betti numbers}, denoted $\beta_k(X):=\dim(H_k(X))$. 

Two examples that will be of use for us in this paper are the sphere and the torus.
Given coefficients in a field $\F$, the homology groups of  the $d$-dimensional sphere $\S^d$, and the $d$-dimensional torus $\T^d = \S^1\times \S^1\times\cdots\times\S^1$, are
\eqb\label{eq:hom_ex}
	H_k(\S^d) \cong \begin{cases} \F & k=0,d,\\
	0 & \text{otherwise},\end{cases} \qquad \text{and} \qquad H_k(\T^d) \cong \F^{\binom{d}{k}},\ 0\le k\le d.
\eqe
This implies that $\beta_0(\S^d) = \beta_d(\S^d) = 1$, while $\beta_k(\S^d) = 0$ for $k\ne 0,d$, and that $\beta_k(\T^d) = \binom{d}{k}$. The case $d=2$ can be visualized by drawing the $2$-dimensional torus, and observing that in addition to a single connected component ($\beta_0 = 1$) there are two essential ``holes" ($\beta_1 = 2$), and a single void ($\beta_2=1$).


\subsection{The \cech complex}\label{sec:cech}

The topological space we study in this paper is called the \emph{\cech complex}. In this section we provide some basic definitions and properties.

\begin{defn}
Let $S$ be a set. An \emph{abstract simplicial complex} $X$ is a collection of finite subsets $A\subset S$ satisfying the following property,
\[
    A \in X\text{ and } B\subset A \quad \Rightarrow \quad B\in X.
\]
\end{defn}
We refer to the sets $A\in X$ with $\abs{A}=k+1$ as $k$-dimensional simplexes, or \emph{$k$-faces}.
We will often refer to a $0$-simplex as `a vertex', $1$-simplex as `an edge', and a $2$-simplex `a triangle'. We denote by $X^k \subset X$ the set of all $k$-simplexes. 

The \cech complex we study is an abstract simplicial complex, defined in the following way.

\begin{defn}\label{def:cech_complex}
Let $\cP = \set{p_1,p_2,\ldots,p_n}$ be a set of points in a metric space, and let $B_r(p)$ be the closed ball of radius $r > 0$ around $p$. The \cech complex $\cC_r(\cP)$ is constructed as follows:
\begin{enumerate}
\item The $0$-simplexes (vertices) are the points in $\cP$.
\item A $k$-simplex $\set{p_{i_1},\ldots,p_{i_{k+1}}}$ is in $\cC_r(\cP)$ if $\bigcap_{j=1}^{k+1} {B_{r}(p_{i_j})} \ne \emptyset$.
\end{enumerate}
\end{defn}

\begin{figure}[ht]
\centering
\includegraphics[scale=0.175]{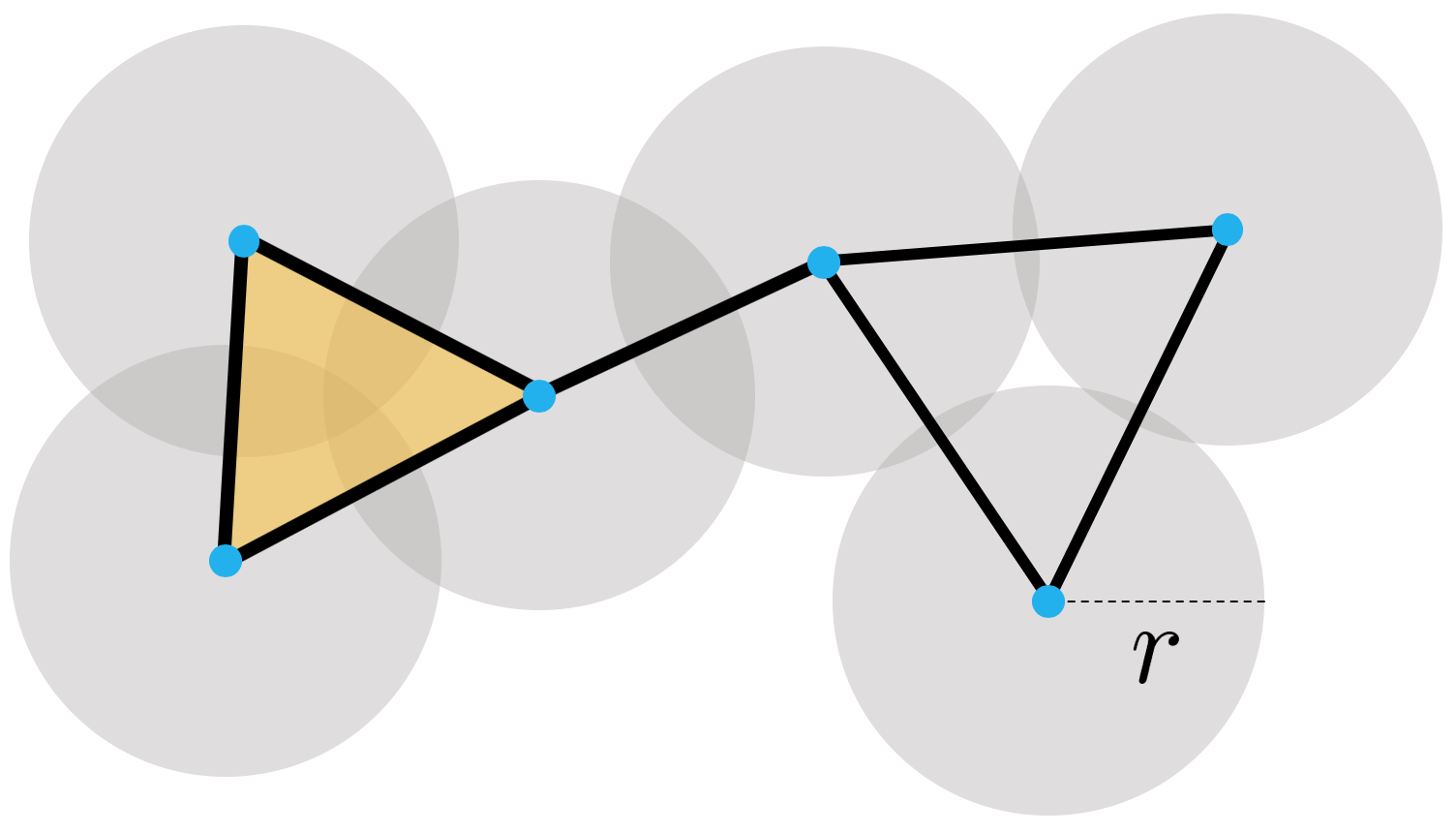}
\caption{\label{fig:cech}  A \cech complex generated by a set of points in $\R^2$. The complex has 6 vertices (0-simplexes), 7 edges (1-simplexes) and one triangle (a 2-simplex). }
\end{figure}

The following lemma states that from the point of view of topology, the abstract simplicial complex 
$\cC_r(\cP)$ is equivalent to the geometric object $B_r(\cP) = \bigcup_{p\in \cP} B_r(p)$, i.e.~the union of balls that were used to generate the complex.
This lemma is a special case of a more general topological statement originated in \cite{borsuk_imbedding_1948}, and commonly referred to as the `Nerve Lemma'.

\begin{lem}[The Nerve Lemma]\label{lem:nerve}
Let $\cC_r(\cP)$ and $B_r(\cP)$ be as defined above. If for every $p_{i_1},\ldots,p_{i_k}$ the intersection $B_r(p_{i_1})\cap \cdots\cap B_r(p_{i_k})$ is either empty or contractible, then $\cC_r(\cP)\simeq B_r(\cP)$ (i.e.~they are homotopy equivalent), and in particular,
\[
	H_k(\cC_r(\cP)) \cong H_k(B_r(\cP)),\quad \forall k\ge 0.
\]
\end{lem}

Note that in Figure \ref{fig:cech} indeed both $\cC_r(\cP)$ and $B_r(\cP)$ have a single  component and a single hole.

\subsection{Critical faces for the \cech complex in $\R^d$}\label{sec:crit_faces}

In this section we wish to provide some definitions that will allow us later to analyze the critical faces of a random \cech complex. 
For simplicity, in this section only, we present the definitions for point-sets in $\R^d$. 
The complete Morse-theoretic definitions and arguments, for either the flat torus $\T^d$ or  general compact Riemannian manifolds can be found in  \cite{bobrowski_topology_2018,bobrowski_vanishing_2017}, and will be briefly discussed in Section \ref{sec:flat_torus}.

Let $\cP$ be a fixed set of points in a metric space. 
We can think of the \cech complex $\cC_r(\cP)$ as a process evolving according the parameter $r$. In other words, we consider the \cech\!-filtration $\set{\cC_r(\cP)}_{r=0}^{\infty}$.
As  $r$ is increased from $0$ to $\infty$ faces of various dimensions are added to the complex.
As a consequence, cycles in various degrees of homology can be created or destroyed.
By \emph{critical faces} we refer to those faces whose addition to the complex facilitates such changes in homology. 
The description we are about to provide relies on \emph{Morse theory}, and more concretely on its specialized version for \emph{min-type} functions   introduced in \cite{gershkovich_morse_1997}. For more details see \cite{bobrowski_distance_2014,bobrowski_topology_2018,bobrowski_vanishing_2017}.




The main observation coming from the Morse-theoretic framework  in \cite{bobrowski_distance_2014,bobrowski_topology_2018,bobrowski_vanishing_2017} 
 is the following. Fix $r_0>0$, and suppose that we increase the value of $r$ from $r_0-\eps$ to $r_0+\eps$.
If $\eps$ is small enough,  only one of three things can happen.
Either: (a) $\cC_{r_0+\eps}(\cP)$ is \emph{homotopy equivalent} to $\cC_{r_0-\eps}(\cP)$, in which case we have
$H_k(\cC_{r_0+\eps}(\cP))\cong H_k(\cC_{r_0-\eps}(\cP))$, i.e.~the homology remains unchanged through the interval $[r_0-\eps,r_0+\eps]$, or (b) There exists a single $k$ for which a new nontrivial cycle is added to $H_k$, implying that $\beta_k(\cC_{r_0+\eps}) =\beta_k(\cC_{r_0-\eps}) +1$, or (c) There exists a single $k$ for which an existing nontrivial cycle in $H_k$ is terminated (becomes trivial), implying that $\beta_k(\cC_{r_0+\eps}) =\beta_k(\cC_{r_0-\eps}) -1$.  For each of the changes ((b) and (c)) we can associate a single face in the complex that is added exactly at $r=r_0$,  such that its addition to the complex facilitates the change in homology. We call these simplexes ``critical faces", since they correspond to critical points of the distance function, see \cite{bobrowski_distance_2014}. From Morse theory we also know that ``positive" changes in $H_k$ (cycles being formed) are associated to critical faces of dimension $k$, while  ``negative" changes  in $H_k$ (cycles being ``filled in")  are associated to critical faces of dimension $k+1$.

Critical faces are going to be highly useful for us in our analysis of the homology of random \cech complexes, providing information about global phenomena (homology) via local structures (faces). Next, we wish to provide necessary and sufficient conditions for  a $k$-dimensional face in $\cC_r(\cP)$ to be critical.

Denote by $\cC^k(\cP)$ the set of all $k$-simplexes that are generated throughout the filtration $\set{\cC_r(\cP)}_r$.
Starting with the $0$-dimensional faces, we consider all the vertices in $\cC^0(\cP)$ to be critical, since each contributes a single connected component. For higher dimensions, let $\cX\in \cC^k(\cP)$, and assume that the points in $\cX$ are in \emph{general position}. General position implies that (a) the set $\cX$ spans a (geometric) $k$-dimensional simplex in $\R^d$, and (b) there is a unique $(k-1)$-dimensional sphere containing $\cX$. To identify this sphere, we start by defining
\[
	E(\cX) := \set{ x\in \R^d : |x-x_1| = |x-x_2| = \cdots = |x-x_{k+1}|},
\]
i.e.~the set of equidistant points from $\cX$. If $\cX$ is in general position, this set is a $(d-k)$-dimensional affine plane, orthogonal to the $k$-plane containing $\cX$. In addition, there exists a unique point $c \in E(\cX)$ such that 
\eqb\label{eq:c_min}
	|c-x_1| = \inf_{x\in E(\cX)} |x-x_1|.
\eqe
We will denote this point by $c(\cX)$, and the corresponding minimum distance by $\rho(\cX)$. We will refer to them as the \emph{center} and \emph{radius} of $\cX$, respectively.

Once $c = c(\cX)$ is identified, we will consider the position of $\cX$ in a $d$-dimensional coordinate system centered at $c$, which we denote $\R^d_c$. Since we assume general position, the set $\cX$ lies on a unique $k$-dimensional linear space of $\R^d_c$, which we denote by $\Pi(\cX)$ -- an element in the Grassmannian $\Gr(d,k)$. Notice that the definition of $c$ as the center of $\cX$, implies that $\cX$ lies on a $(k-1)$-sphere of radius $\rho(\cX)$ centered at the origin of $\R^d_c$. We will denote by $\theta(\cX) \subset \S^{k-1}$ the spherical coordinates of $\cX$ on this sphere. Notice that the transformation $\cX \to (c(\cX), \rho(\cX),\Pi(\cX),\theta(\cX))$ is a bijection, a fact that we will use later when we use the Blaschke-Petkantschin formula later (see Appendix \ref{sec:bp}). To conclude, for every abstract simplex $\cX\in \cC^k(\cP)$ we defined the following,
\eqb \label{eq:crit_face_1}
\splitb
c(\cX) &= \text{The center of $\cX$}, \\
\rho(\cX) &= \text{The radius of $\cX$} , \\
\Pi(\cX) &= \text{The linear $k$-plane centered at $c(\cX)$ containing $\cX$}, \\
\theta(\cX) &= \text{The spherical coordinates of $\cX$ in $\Pi(\cX)$.}
\splite
\eqe

We will also need the following defintions,
\eqb \label{eq:crit_face_2}
\splitb
\sigma(\cX) &= \text{The \emph{open} geometric $k$-simplex (a subset of $\R^d$) spanned by $\cX$}, \\
B(\cX) &= \text{The \emph{open} ball in $\R^d$ with radius $\rho(\cX)$ centered at $c(\cX)$}.  
\splite
\eqe

The following Lemma  identifies the critical faces.

\begin{lem}[Lemma 2.4 in \cite{bobrowski_vanishing_2017}]\label{lem:crit_face}
Let $\cX \in \cC^k(\cP)$ be in general position. Then $\cX$ is a critical $k$-dimensional  face in the \cech filtration, if and only if
\[
	(1)\ \ c(\cX) \in \sigma(\cX),\quad\textrm{ and }\quad (2)\ \ B(\cX)\cap \cP= \emptyset.
\]
\end{lem}

Figure \ref{fig:crit_faces} provides examples for critical faces in dimensions $k=1,2$.

\begin{figure}[ht]
\begin{center}
\includegraphics[scale=0.25]{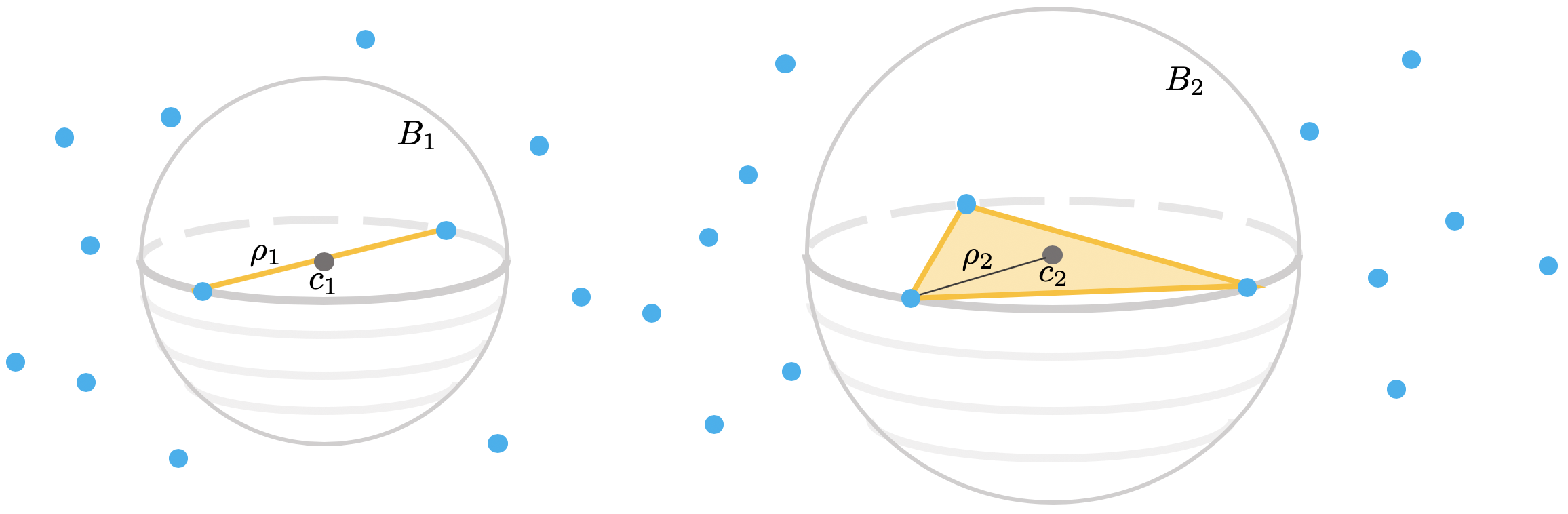}
\caption{\label{fig:crit_faces} 
Critical faces of dimension $k=1,2$ generated by points in $\R^d$. We observe two sets $\cX_1$ (left) and $\cX_2$ (right) consisting of 2 and 3 points, respectively. We denote $c_i = c(\cX_i),\rho_i = \rho(\cX_i)$, and $B_i = B(\cX_i)$.
In both cases the center $c_i$ lies in the open simplex $\sigma(\cX_i)$, satisfying condition (1) in Lemma \ref{lem:crit_face}. In addition, the open balls $B_i$ are empty, satisfying condition (2).
}\end{center}
\end{figure}

\begin{rem}\label{rem:half_sphere} 
(1) The first condition in Lemma \ref{lem:crit_face} is equivalent to stating that $\rho(\cX)$ is equal to the radius at which $\cX$ is added the complex. 
(2) The first condition is also equivalent to the requirement that there is no hemisphere of $\S^{k-1}$ that contains all the points in $\theta(\cX)$. We will use these equivalent notions in some of our arguments later.
\end{rem}

\subsection{The $d$-dimensional flat torus}\label{sec:flat_torus}

As stated in the introduction, the results in this paper can be proven for Poisson processes generated on compact smooth Riemannian manifolds. 
In order to simplify the notation and  calculations required for the proofs, we focus our attention on the special case of the $d$-dimensional flat torus. 
By `flat torus' we refer to the quotient space $\T^d  = \R^d / \Z^d$. We can think of $\T^d$ as the cube $[0,1]^d$ under the relation $0\sim 1$, and with the toroidal metric
\[
d_T(x,y) = \min_{\delta\in \Z^d} \abs{x-y+\delta},\quad x,y\in [0,1]^d.
\]
The main advantage of the flat torus is that it allows us to work in a compact Riemannian manifold setting, while keeping the metric simple in the following sense.
As we will mostly be considering infinitesimally small neighborhoods, for all practical purposes, the metric we will be using  can be assumed to be Euclidean. Compared to the cube $[0,1]^d$, the torus $\T^d$ has no boundary, which will also be desirable for us. 
 
Studying the \cech complex, and its critical faces in particular, in the flat torus case (as well as  any compact Riemmannian manifold) requires some extra caution compared to $\R^d$. Firstly, recall that the Nerve Lemma \ref{lem:nerve} requires the intersection of balls to be contractible. While this is  true  in $\R^d$, on $\T^d$ it is not always the case.
Secondly, our definition of critical faces via Lemma \ref{lem:crit_face} needs to be revisited. This lemma, as well as \eqref{eq:crit_face_1} and \eqref{eq:crit_face_2}, are not always well-defined on the torus. 

In 
\cite{bobrowski_topology_2018} we showed that there exists $\rmax >0$ such that if we limit ourselves to study the bounded filtration $\set{\cC_r(\cP)}_{r=0}^{\rmax}$ then these issues are resolved. Firstly, $\rmax$ is taken to be small enough so that all intersections are contractible. Secondly, we show in \cite{bobrowski_topology_2018} that if $E(\cX)\cap B_{\rmax}(\cX)\ne \emptyset$ then $c(\cX)$ is well defined. Further, in this case we have that $\cX \subset B_{\rmax}(c(\cX))$, and this neighborhood can be isometrically embedded  in $\R^d$ as a Euclidean ball. Thus, we can view $\cX$ as lying in $\R_c^d$ and proceed as in Section \ref{sec:crit_faces}.
Note that while there might be  faces in $\set{\cC_r(\cP)}_{r=0}^{\rmax}$ for which $E(\cX)\cap B_{\rmax}(\cX)= \emptyset$, these faces cannot be critical, and therefore will not affect our analysis. Finally, our results will show that homological connectivity occurs at $r\ll \rmax$, therefore the limitation on the filtration will have no practical effect.

\subsection{The Poisson process}\label{sec:poisson}

In this paper we study random \cech complexes constructed over point sets $\cP$  generated by finite  homogeneous Poisson processes.
Let $X_1,X_2,\ldots$, be a sequence of $\iid$ (independent and identically distributed) random variables, distributed uniformly on $\T^d$. Let $N\sim \pois{n}$ be a Poisson random variable, independent of the $X_i$-s. Define the spatial Poisson process $\cP_n$ as
\[
\cP_n := \set{X_1,X_2,\ldots, X_N}.
\]
Notice that the size of the set $\cP_n$ is random but finite, and $\mean{\abs{\cP_n}} = n$.
For every subset $A\subset \T^d$ we define $\cP_n(A) = \abs{\cP_n\cap A}$, i.e.~the number of points lying in $A$.
Two important properties of the process $\cP_n$ are the following: 
\begin{enumerate}
\item For every  $A\subset  \T^d$,  we have $\cP_n(A) \sim \pois{n\vol(A)}$.

\item For every $A,B\subset  \T^d$  such that $A\cap B = \emptyset$, the variables  $\cP_n(A),\cP_n(B)$ are independent.
\end{enumerate}
The last property is sometimes referred to as ``spatial independence''. It is worth noting that these two properties actually can be used as an equivalent definition of $\cP_n$.

In this paper we  study the random \cech complex $\cC_r :=\cC_r(\cP_n)$,  where $n$ will often be omitted.


\subsection{Notation}

The main results in this paper are asymptotic. We use  $a_n = o(b_n)$ do denote that $a_n/b_n \to 0$, and  $a_n\approx b_n$ to denote that $a_n/b_n \to 1$. We say that an event $E$ occurs with high probability (\whp) if $\limninf \prob{E} = 1$.

The calculations in this paper will involve some constant values. Those constants that carry important meaning will be denoted by  the letter `$D$' (with an identifying index). 
Other constants that appear only within the context of a single proof, and are not needed elsewhere will be denoted by $C_i$, where the index $i$ resets at the beginning of every proof.
Finally, there are numerous places where the value of the constant is  insignificant, and then we will denote it by $\const$. Notice that the value of $\const$  changes throughout the paper, even within the same equation.

Finally, a word on set notation. We defined an abstract $k$-simplex $\cX$ to be a set of $(k+1)$ elements. Throughout the paper we will refer to simplexes as sets indeed. However, in our calculations we will often assume that some arbitrary ordering is given on $\cX$, so that we can list it as $\cX = \set{x_1,\ldots, x_{k+1}}$. Abusing notation, we will then allow mixing between sets and tuples. For example, if $\cX$ is a $k$-simplex, and $\bx \in (\T^d)^{k+1}$ we will use ``$\cX = \bx$" to state that $X_i = x_i$ for all $1\le i \le k+1$. In addition, we will often use functions defined on sets $\cX$ of a fixed size. In these cases we will allow the input to be a tuple $\bx$ of the same size, as the function value will be independent of the ordering.

\section{Outline of the results}
\label{sec:results}


Let $\cP_n$ be a homogeneous Poisson process on $\T^d$, with intensity $n$. 
Define $\cC_r := \cC_r(\cP_n)$ to be the \cech complex generated by $\cP_n$. We will focus on the event,
\eqb\label{eq:hom_conn}
\cH_{k,r} := \set{H_k(\cC_s) \cong H_k(\T^d),\ \ \forall s\in[r,
\rmax]}.
\eqe
More concretely, by $\cong$ we mean that the map $j_*:H_k(\cC_s)\to H_k(\T^d)$ induced by the inclusion $B_r(\cP_n)\hookrightarrow \T^d$ and the Nerve Lemma \ref{lem:nerve}, is an isomorphism.
We refer to the occurrence of the event $\cH_{k,r}$ as `homological connectivity'. The restriction $s\le \rmax$ comes from our discussion in Section \ref{sec:flat_torus}. Note, however, that if we consider the union of balls $B_r(\cP_n)$ rather than the \cech complex $C_r(\cP_n)$, we can remove this restriction. We will prove that $\cH_{k,r}$ occurs when $r\ll \rmax$, in which case $B_r(\cP_n)\simeq \cC_r(\cP_n)$ by the Nerve Lemma \ref{lem:nerve}. Thus, for any practical purposes this restriction carries no significance.

The case  $k=0$ represents the usual sense of graph connectivity, discussed in the introduction.
The main results in this paper are the proofs for sharp phase transitions for the occurrence of $\cH_{k,r}$, for $1\le k \le d$. The behavior of $\cH_{k,r}$ is different between $k=1,\ldots, d-2$, and $k=d-1,d$. We phrase all the results in this paper in terms of $\Lambda = \omega_d nr^d$ \eqref{eq:lambda}.

\begin{thm}\label{thm:pt_hk_1}
Let $1\le k \le d-2$, and suppose that as $n\to\infty$, $w(n)\to\infty$. Then,
\[
\limninf \prob{\cH_{k,r}} = \begin{cases} 1 & \Lambda = \thres + w(n), \\
0 & \Lambda = \thres - w(n).\end{cases}
\]
\end{thm}

\begin{thm}\label{thm:pt_hk_2}
Suppose that as $n\to\infty$, $w(n)\to\infty$. Then, 
\[
\limninf \prob{\cH_{d-1,r}} = \limninf \prob{\cH_{d,r}} = \begin{cases} 1 & \Lambda = \log n + (d-1)\logg n  + w(n), \\
0 & \Lambda = \log n + (d-1)\logg n - w(n).\end{cases}
\]
\end{thm}

These results show the existence of multiple sharp phase transitions, with increasing threshold values: $\Lambda = \thres$ (for $1\le k \le d-2$), and $\Lambda = \log n + (d-1)\logg n$ (for $k=d-1,d$). In other words, homological connectivity occurs in an orderly fashion, where the degree of homology $k$ controls the second-order term of the threshold. From the discussion in the introduction, we know that $\cH_{0,r}$ occurs at $\Lambda = 2^{-d}\log n$, which is much earlier than the thresholds we have here. In addition, 
the threshold for $\cH_{d-1,r}$ and $\cH_{d,r}$ is the same as the coverage threshold \cite{bobrowski_vanishing_2017,flatto_random_1977}. This is not a coincidence, and we will discuss it later. Finally, notice that there is a ``gap" in the sense that at $\Lambda = \log n + (d-2)\logg n$ we do not observe any homological phase transition.

The main challenge in proving Theorems \ref{thm:pt_hk_1} and \ref{thm:pt_hk_2} is that homology describes global features of the complex that cannot be deduced by looking at local neighborhoods.
Nevertheless, when the complex is dense (i.e.~when $\Lambda\to\infty$) we will show that a significant amount of information about homology can be extracted from local information. This local information is related to Morse theory, and manifests itself in the critical faces introduced in Section \ref{sec:crit_faces}.

To analyze the critical faces, we define
\eqb\label{eq:Fk}
	F_{k,r} := \#\text{ critical $k$-faces $\cX \in \cC^k(\cP_n)$, with $\rho(\cX) \in (r,\rmax]$},
\eqe
where $\rho(\cX)$ was defined in \eqref{eq:crit_face_1}.
In other words, $F_{k,r}$ counts critical $k$-faces with radius larger than $r$. These critical faces will facilitate changes in the homology of the complex if we increase the radius beyond $r$. Therefore, our main effort will be to identify the point when $F_{k,r}=0$.
Recall from Section \ref{sec:crit_faces} that critical faces can be divided into two groups - \emph{positive} (creating cycles) and \emph{negative} (terminating cycles). We therefore define the following,
\eqb\label{eq:Fk_pos_neg}
\splitb
	\Cp_{k,r} &:= \#\text{ positive critical $k$ faces $\cX \in \cC^k(\cP_n)$, with $\rho(\cX)\in(r,\rmax]$},\\
	\Cn_{k,r} &:= \#\text{ negative critical $k$ faces $\cX \in \cC^k(\cP_n)$, with $\rho(\cX)\in(r,\rmax]$}.
\splite
\eqe

Notice that $F_{k,r} = \Cp_{k,r} + \Cn_{k,r}$. It is also worth noting that the faces accounted for by $\Cp_{k,r}$ generate cycles in $H_k$, while the faces in $\Cn_{k,r}$ terminate cycles in $H_{k-1}$ (see Section \ref{sec:crit_faces}). The quantities $F_{k,r}, \Cp_{k,r}, \Cn_{k,r}$ are interesting random variables by themselves. In addition, their analysis will shed light on the behavior of homology, and in particular will lead to the proof of Theorems \ref{thm:pt_hk_1} and $\ref{thm:pt_hk_2}$. In Section \ref{sec:pt_crit}, we start with the following result for  $F_{k,r}$.

\newtheorem*{tprop:mean_var}{Proposition \ref{prop:mean_var}}
\begin{tprop:mean_var}
Let $1\le k \le d$, and suppose that $r\to 0$  and $\Lambda\to \infty$. Then,
\[
	\mean{F_{k,r}} \approx \var{F_{k,r}} \approx D_k n\Lambda^{k-1}e^{-\Lambda},
\]
where $D_k>0$ will be defined in \eqref{eq:A_crit}.
\end{tprop:mean_var}

This result will leads almost immediatley to the following phase transition.

\newtheorem*{tprop:pt_crit}{Proposition \ref{prop:pt_crit}}
\begin{tprop:pt_crit}
Let $1 \le k \le d$, and suppose that $w(n)\to\infty$. Then
\[
	\limninf \prob{ F_{k,r} = 0} = \begin{cases} 1 & \Lambda = \thres + w(n),\\ 
	0 & \Lambda = \thres - w(n). \end{cases}
\]
\end{tprop:pt_crit}

Deriving similar phase transitions for
 the positive and negative faces separetly, is a much more challenging task. The following is the main result of Section \ref{sec:pn_faces}.

\newtheorem*{tprop:pt_neg}{Proposition \ref{prop:pt_neg}}
\begin{tprop:pt_neg}
 Let $1 \le k \le d-2$, and  $w(n)\to 0$. Then,
\[
	\limninf \prob{ \Cp_{k,r} = 0} = \limninf \prob{ \Cn_{k+1,r} = 0} = \begin{cases} 1 & \Lambda = \thres + w(n),\\
	0 & \Lambda = \thres - w(n). \end{cases}
\]
\end{tprop:pt_neg}
In other words, around the $k$-th  phase transition, \whp~we have $F_{k,r} = \Cp_{k,r}$ (since $\Cn_{k,r}=0$), and these critical $k$-faces generate the last cycles appearing in the \cech filtration. At the same time we also observe the last negative $(k+1)$-faces, that are to terminate those very last $k$-cycles that obstruct homological connectivity. Using Morse theory (see Section \ref{sec:crit_faces}), once Propositions \ref{prop:pt_crit} and \ref{prop:pt_neg} are proved, the proofs for Theorems \ref{thm:pt_hk_1} and \ref{thm:pt_hk_2} follow almost immediately.

Notice that Proposition \ref{prop:pt_neg} excludes two cases: $k=0$ and $k=d-1$. The point where $\Cn_{1,r}$ vanishes is when the complex becomes connected, i.e.~around $\Lambda = 2^{-d}\log n$. Thus, in the regime we are analyzing here we already have  \whp\  $\Cn_{1,r}=0$. As for $\Cn_{d,r}$, notice  $\beta_d(\T^d)=1$, and for every subset of $A \subset \T^d$ we have $\beta_d(A) = 0$. Thus,  as long as $F_{d,r} > 0$ it is always true that $\Cp_{d,r} = 1$ (as only a single $d$-cycle is to be created) and then $\Cn_{d,r} = F_{d,r}-1$. The vanishing of $\Cn_{d,r}$ will therefore follow the same phase-transition as $F_{d,r}$, i.e.~at $\Lambda = \log n + (d-1)\logg n$, and not at $\Lambda = \log n + (d-2)\logg n$ as Proposition \ref{prop:pt_neg} might suggest.

Once we proved the phase transitions in Theorems \ref{thm:pt_hk_1} and  \ref{thm:pt_hk_2}, we will move on to examine the behavior of the complex inside the critical window for each of these phase transitions. As we discussed in the introduction, we want to study the interferences to homological connectivity. The interesting phenomenon we discover here is that inside the critical window, obstructions to homological connectivity  always appear in the form of pairs of positive-negative faces, such that the positive $k$-simplex that generates a new $k$-cycle is a face of the $(k+1)$-simplex terminating that same cycle. Consequently, the difference in the critical radii is infinitesimally small, and these obstructing cycles are ``born" and ``die" instantaneously.
A cartoon picture demonstrating these cycles is in Figure \ref{fig:iso_faces}. Notice that similar pairings also occur in the random graph case -- the obstructions to connectivity are pairs of positive $0$-faces (isolated vertices) and their matching negative $1$-faces (the edges that connect them to the rest of the graph).
 
\begin{figure}[ht]
\centering
\includegraphics[width=0.5\textwidth]{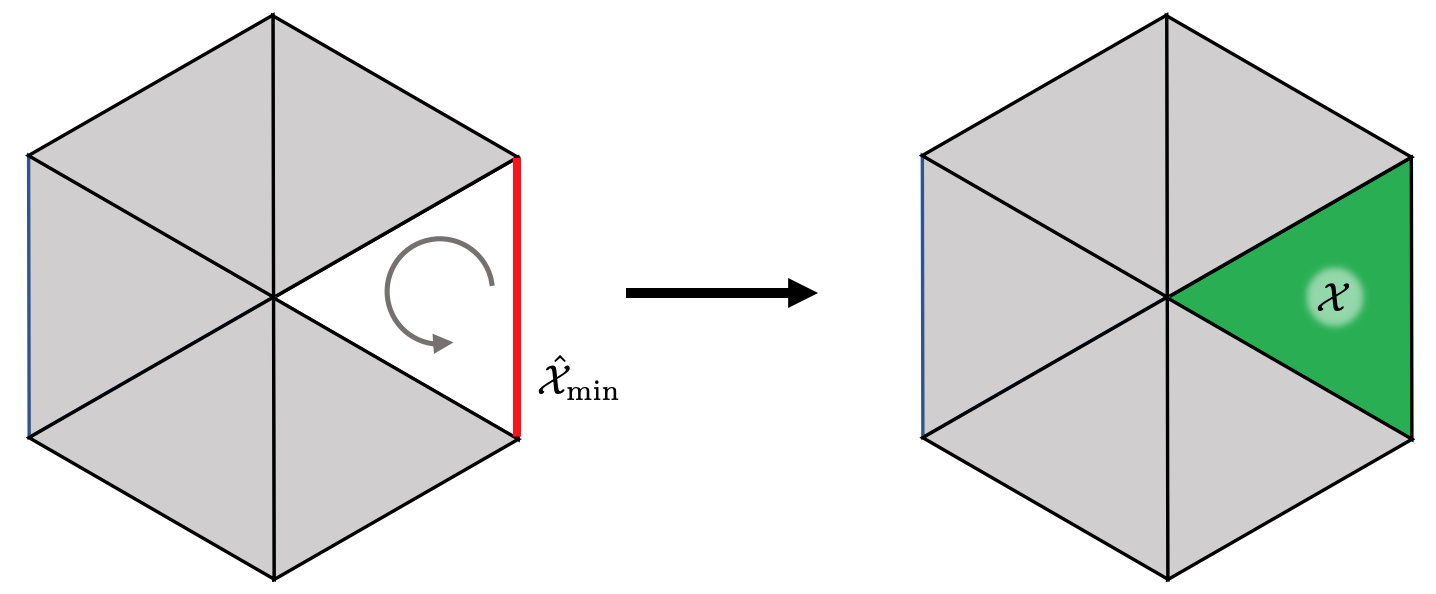}
\caption{
The last $k$-cycles are created by positive $k$-faces ($\hcXmin$), and are being instantaneously terminated by the $(k+1)$-coface $\cX$. In this example $k=1$. 
\label{fig:iso_faces}}
\end{figure}

More formally, for every $k$-simplex $\cX$ we will define $\hcXmin$ to be the ``nearest" $(k-1)$-face of $\cX$, in the sense that the plane containing $\hcXmin$ is closest to the center $c(\cX)$ among all the $(k-1)$-faces of $\cX$ (this will be rigorously defined later, also see Figure \ref{fig:phi}). Then, we define
\[
\splitb
    \Cnp_{k,r} := \#&\text{{\bf negative} critical $k$-faces  $\cX\in\cC^k(\cP_n)$ with $\rho(\cX)\in (r,\rmax]$,}\\ &\text{and such that $\hcXmin$ is a {\bf positive} critical $(k-1)$-face}.
\splite
\]
We will prove the following lemma.
\newtheorem*{tlem:pn_pairs}{Lemma \ref{lem:pn_pairs}}
\begin{tlem:pn_pairs}
Let $1 \le k \le d-2$. Suppose that $\Lambda = \log n + (k-1)\logg n + o(\logg n)$. Then,
\[	
\limninf \prob{ \Cp_{k,r}= \Cn_{k+1,r} =  \Cnp_{k+1,r} } = 1.
\]
\end{tlem:pn_pairs}
This lemma implies that for every $r$ inside the critical window, the only obstructions to homological connectivity are the positive-negative pairs captured by $\Cnp_{k+1,r}$. The lemma also excludes the possibility of a cycle generated at $\rho < r$, that is due to be terminated at $\rho > r$ (since the number of positive and negative faces is equal). This observation, together with Theorem 5.4 in \cite{bobrowski_vanishing_2017} lead to the following corollary.

\newtheorem*{tcor:pt_ihk_1}{Corollary \ref{cor:pt_ihk_1}}
\begin{tcor:pt_ihk_1}
Let $1\le k \le d-2$, and  $w(n)\to\infty$. Then,
\[
	\limninf \prob{H_k(\cC_r) \cong H_k(\T^d)} = \begin{cases} 1 & \Lambda \ge \log n + (k-1)\logg n - w(n),\\ 0 & \Lambda \le \log n + (k-2)\logg n - w(n).
	\end{cases}
\]
\end{tcor:pt_ihk_1}
This corollary implies that the instantaneous homology $H_k(\cC_r)$ exhibits a different phase transition than $\cH_{k,r}$. In particular, it shows that there are choices of $r$ for which \whp~$H_k(\cC_r)\cong H_k(\T^d)$, while $\cH_{k,r}$ does not hold. In other words, even when $H_k(\cC_r)\cong H_k(\T^d)$ \whp~the homology is still not stable. This is the main reason why we argue that $\cH_{k,r}$ is a  better interpretation for homological connectivity in this model. We conjecture that the exact transition for $H_k(\cC_r)$ occurs at $\Lambda = \log n + (k-2)\logg n$. However, its non-monotonicity also implies that  there might not even be such sharp result.

Notice that as before, the cases $k=d-1,d$ are excluded from  Lemma \ref{lem:pn_pairs} and  Corollary \ref{cor:pt_ihk_1}. From Propositions \ref{prop:mean_var} and \ref{prop:pt_crit}, when $\Lambda = \log n + (d-1)\logg n - w(n)$, we have $\Cn_{d,r} > 0$ while $\Cp_{d-1,r} = 0$.  In other words, around the phase transition $H_{d-1}(\cC_r)$ behaves in a monotone fashion, as no new $(d-1)$-cycles are generated. Consequently, the positive-negative pairing structure discussed above does not apply here. This leads to the following statement.

\newtheorem*{tcor:pt_ihk_2}{Corollary \ref{cor:pt_ihk_2}}
\begin{tcor:pt_ihk_2}
Let $k =d-1,d$, and let $w(n)\to\infty$. Then,
\[
	\limninf \prob{H_k(\cC_r) \cong H_k(\T^d)} = \begin{cases} 1 & \Lambda = \log n + (d-1)\logg n + w(n),\\ 0 & \Lambda = \log n + (d-1)\logg n - w(n).
	\end{cases}
\]
\end{tcor:pt_ihk_2}
In other words, for $k=d-1,d$ the phase transitions for $\set{H_k(\cC_r) \cong H_k(\T^d)}$ and $\cH_{k,r}$ are indeed the same.

Finally, in order to examine the distribution of the very last $k$-cycles, we focus on the regime where $\Lambda = \thres + \lambda$, and $\lambda \in \R$. Following Lemma \ref{lem:pn_pairs} and Proposition \ref{prop:pt_neg}, we know that in this regime \whp~$\Cn_{k+1,r} =  \Cp_{k,r}  = F_{k,r}$. 
Therefore, revealing the limiting distribution of $F_{k,r}$ will provide the distributions of the other terms.
From Proposition \ref{prop:mean_var} we have that
\[
\mean{F_{k,r}} \approx  \var{F_{k,r}}\approx D_k e^{-\lambda},
\]
Since these limits are finite and equal, one might expect a Poisson distribution in the limit. The result we will prove is stronger than point-wise convergence of $F_{k,r}$ to a Poisson random variable. Firstly, we will show that the entire point process representing the occurrences of the last critical $k$-faces converges to a Poisson counting process. Secondly, we will prove similar results for the positive and negative  faces separately.

Fix $k \in \set{1,\ldots,d}$. We say that $\lambda$ is ``$k$-critical" if the corresponding radius $r$ is critical for the filtration $\cC_r$ (i.e.~a new critical $k$-faces joins the complex at $r$). Similarly, define ``$k$-positive" and ``$k$-negative". Next, we define the counting processes (see \eqref{eq:crit_proc} for a more rigorous definition),
\[
\splitb
	N_k(t) &:= \# \text{$k$-critical values $\lambda$, such that $e^{-\lambda} \le t$},\\
		\Np_k(t) &:= \# \text{$k$-positive values $\lambda$, such that $e^{-\lambda} \le t$},\\
			\Nn_k(t) &:= \# \text{$(k+1)$-negative values $\lambda$, such that $e^{-\lambda} \le t$}.
\splite
\]
In other words, the process $N_k$ (resp.~$\Np_k, \Nn_k$) counts the occurrences of the last critical faces (resp.~positive, negative),  in a reversed order. 
The following theorem states that the distribution of all three processes  converges  to that of a homogeneous Poisson counting process.

\newtheorem*{tthm:pois_proc}{Theorem \ref{thm:pois_proc}}
\begin{tthm:pois_proc}
Let $1 \le k \le d$, and let $V_k(t)$ ($t\ge 0$) be a homogeneous Poisson (counting) process with rate $D_k$. Then,
\[
	N_k \Rightarrow V_k,
\] 
where `$\Rightarrow$' refers to the weak convergence of the final dimensional distributions. In other words, for all $m, t_1,\ldots, t_m$, we have a multivariate weak convergence,
\[
	(N_k(t_1),\ldots, N_k(t_m)) \Rightarrow
	(V_k(t_1)\ldots, V_k(t_m)).
\]
Similarly, for $1\le k \le d-1$ we have $\Np_k \Rightarrow V_k$, and for $1\le k \le d-2$ we have $\Nn_k \Rightarrow V_k$.
\end{tthm:pois_proc}

The following is a point-wise conclusion from the proof of Theorem \ref{thm:pois_proc}.

\newtheorem*{tcor:pois_var}{Corollary \ref{cor:pois_var}}
\begin{tcor:pois_var}
Let $1\le k \le d$, and set $\Lambda = \thres + \lambda$, then
\[
	\dtv{F_{k,r}}{ Z_{\lambda}} \to 0,
\]
where $Z_{\lambda} \sim \pois{D_k e^{-\lambda}}$, and $d_{\mathrm{TV}}$ is the total-variation distance.\\
The same limit holds for $\Cp_{k,r}$ ($1\le k \le d-1$) and $\Cn_{k+1,r}$ ($1\le k \le d-2$).

\end{tcor:pois_var}

For $1\le k \le d-2$,  Lemma \ref{lem:pn_pairs} implies that $\Nn_k(t)$ is in fact the  process that counts the occurrences of the last pairs of positive-negative faces obstructing homological connectivity in degree $k$  (in a reversed order). In other words, Theorem \ref{thm:pois_proc} provides a full description for the limiting distributions of the arrival times of the last obstructions before $\cH_{k,r}$ takes place.

As mentioned earlier,  $\cH_{d,r}$ occurs when the last critical $d$-face joins the complex, implying that the one-before-last $d$-face facilitates $\cH_{d-1,r}$. Therefore, the process $N_d(t)$ is the process counting the obstructions to connectivity in both degrees $(d-1)$ and $d$. The first jump in this process (as the order is reversed) corresponds to connectivity for $k=d$, while the second jump is for $k=d-1$.

Finally, the limiting Poisson process in Theorem \ref{thm:pois_proc}, enables us to derive the exact limiting probability for $\cH_{k,r}$ inside the critical window.

\newtheorem*{tthm:crit_prob}{Theorem \ref{thm:crit_prob}}
\begin{tthm:crit_prob}
Let $1\le k\le d$. For $k\ne d-1$, if $\Lambda = \thres + \lambda$, then
\[
	\limninf \prob{\cH_{k,r}} = 
	e^{-D_k e^{-\lambda}}.
\]
For $k=d-1$, if $\Lambda = \log n +(d-1)\logg + \lambda$ then
\[
	\limninf \prob{\cH_{d-1,r}} = 
	e^{-D_d e^{-\lambda}}(1+D_de^{-\lambda}).
\]
\end{tthm:crit_prob}

\begin{rem}
As we stated earlier, the radius at which $\cH_{d,r}$ occurs is the same as the coverage radius studied in \cite{flatto_random_1977}. The case $k=d$ in Theorem \ref{thm:crit_prob} therefore provides the exact limit for the coverage probability. To the best of our knowledge, this is an interesting new result by itself.
\end{rem}

This concludes the main results in this paper. Notice that we presented statements  similar to the properties of random graphs \eqref{eq:ER_1}-\eqref{eq:ER_4} discussed in the introduction. The only missing piece is the connection to isolated faces, which we discuss next.
The remainder of the paper will be devoted to proving the statements that we presented in this section.

\subsubsection*{A few words on isolated faces.}
Recall our discussion on isolated vertices in random graphs and isolated faces in random simplicial complexes. The phenomenon that was observed in \cite{erdos_random_1959,kahle_sharp_2014,linial_homological_2006}, was that the obstructions to connectivity are always the remaining isolated objects (vertices in graphs, faces in simplicial complexes), and once the last isolated object is covered, connectivity occurs. In addition,  the number of remaining isolated vertices/faces has a limiting Poisson distribution.

A similar behavior occurs here as well. Notice that whenever a critical $k$-face joins the complex, it is isolated, due to condition (2) in Lemma \ref{lem:crit_face}. We will show later (see Lemma \ref{lem:F_4}), that for the positive-negative pairs that appear inside the critical window, the positive $k$-face remains isolated right until the moment when its matching negative $(k+1)$-face appears. In other words, the last cycles generated in the \cech filtration are formed by isolated critical $k$-faces, and they vanish once all these faces are covered. Consequently, as in other models, the point where homological connectivity occurs is exactly the moment when the last of these isolated faces gets covered. Note, however, that our conclusion is for \emph{critical} isolated faces and not for any isolated faces. Similar calculations to the ones we  present in this paper, can show that for $k \ge 2$,  even within the critical window, isolated faces may appear that are \emph{not} critical. However, these faces  have no effect on the homology of the complex and therefore their isolation status is irrelevant to homological connectivity.
In the random graphs model, as well as in the Linial-Meshulam random $k$-complex, all isolated vertices/faces are critical by definition, and therefore the phenomenon we observe here in fact aligns, and generalizes in way, what is already known for other models.
 
Next, recall the ``hitting-times" phrasing \eqref{eq:st_complex} discussed in the introduction. We can rephrase the discussion in the previous paragraph in terms of the following random times,
\eqb\label{eq:st_geom}
\splitb
T_k^{\text{iso}} &:= \inf\set{r : \cC_s \text{ has no isolated critical $k$-faces for all $s\ge r$}},\\
T_k &:= \inf\set{r : \cH_{k,r} \text{ holds}} .
\splite
\eqe
A direct corollary from Proposition \ref{prop:pt_neg} together with Lemmas \ref{lem:pn_pairs} and \ref{lem:F_4} is that \whp~$T_k^{\text{iso}} = T_k$ for all $1\le k \le d$. Further, if we define $T_k' = \exp(-n \omega_d T_k^d + \log n +(k-1)\logg n)$ then Theorem \ref{thm:crit_prob} is equivalent to the statement 
\eqb\label{eq:limit_times}
	T_k' \xrightarrow{\cL} \mathrm{Exponential}(D_k),\ \ k\ne d-1,\qquad \text{and} \qquad
	T_{d-1}' \xrightarrow{\cL} \mathrm{Gamma(2,D_d)}.
\eqe
Finally, our results  also relate to the  MSA interpretation discussed above. Notice that in the LM complex (as well as the \erdren graph), 
the faces of the MSA (MST) are merely the negative faces in the filtration. The Poisson-process limit for $\Nn_k$ provided in 
Theorem \ref{thm:pois_proc} is therefore analogous to the Poisson-process limit discussed earlier for the heaviest faces in the MSA.
In addition, all these processes can be thought of as representing the so-called ``death times" in \emph{persistent homology} (cf. \cite{carlsson_topology_2009,edelsbrunner_persistent_2008}). Similarly, $\Np_k$ represents the latest ``birth times" in persistent homology.

\section{Phase transitions for critical faces}\label{sec:pt_crit}

The road to proving the phase transitions in Theorems \ref{thm:pt_hk_1}-\ref{thm:pt_hk_2} begins in the analysis of the critical faces. 
Recall the definition of $F_{k,r}$ in \eqref{eq:Fk}.
The main goal in this section is to prove   Propositions \ref{prop:mean_var} and \ref{prop:pt_crit} (only the expectation part of Proposition \ref{prop:mean_var} will be proved here, while the variance result is postponed to Section \ref{sec:pois_limit}). 

\begin{prop}\label{prop:mean_var}
Let $1\le k \le d$, and suppose that $r\to 0$  and $\Lambda\to \infty$. Then,
\[
	\mean{F_{k,r}} \approx \var{F_{k,r}} \approx D_k n\Lambda^{k-1}e^{-\Lambda},
\]
where $D_k>0$ will be defined in \eqref{eq:A_crit}.
\end{prop}
Using Proposition \ref{prop:mean_var}, the following Proposition is straightforward.	
\begin{prop}\label{prop:pt_crit}
Let $1 \le k \le d$, and $w(n)\to\infty$. Then
\[
	\limninf \prob{ F_{k,r} = 0} = \begin{cases} 1 & \Lambda = \thres + w(n),\\
	0 & \Lambda = \thres - w(n). \end{cases}
\]
\end{prop}
 
We start with a few definitions.
Recall the conditions for critical faces stated in Lemma \ref{lem:crit_face}. To verify these conditions we define
\[
\splitb
    \hcrit(\cX) &:= \indf{c(\cX)\in \sigma(\cX)},\\
    \gcrit(\cX,\cP) &:= \indf{B(\cX)\cap \cP = \emptyset},
\splite
\]
where $\cX,\cP$ are finite sets of $\T^d$. In other words $\hcrit$ verifies condition (1) in Lemma \ref{lem:crit_face}, while $\gcrit$ verifies condition (2). 
Next, we define
\[
\splitb
    h_{r}(\cX) &:= \hcrit(\cX)\indf{\rho(\cX)\in (r,\rmax]},\\
	g_{r}(\cX,\cP) &:= \gcrit(\cX,\cP)h_{r}(\cX).
\splite
\]
Recall from our discussion in Section \ref{sec:flat_torus} that for the definitions related to critical faces to be valid, we need that $E(\cX)\cap B_{\rmax}(\cX)\ne \emptyset$. Since this is true for all critical faces with in $\set{\cC_r}_{r=0}^{\rmax}$, we will always implicitly assume that this condition holds. Notice that
\eqb\label{eq:def_ck}
	F_{k,r} = \sum_{\cX\in \cC^k(\cP)} g_{r}(\cX,\cP).
\eqe

We start by an exact (non-asymptotic) evaluation for the expectation of $F_{k,r}$.

\begin{lem}\label{lem:mean_ck}
For $k\ge 1$, suppose that $r<\rmax$ and let $\Lambda = \omega_d nr^d$, and $\Lambda_{\max} = \omega_d n\rmax^d$. Then,
\[
	\mean{F_{k,r}} = D_k (k-1)!n\param{e^{-\Lambda} \sum_{j=0}^{k-1}\frac{\Lambda^j}{j!} - e^{-\Lambda_{\max}}\sum_{j=0}^{k-1}\frac{\Lambda_{\max}^j}{j!}} ,
\]
where $D_k>0$  is defined in \eqref{eq:A_crit}.
\end{lem}

\begin{proof}
This is a direct corollary of Proposition 6.1 in \cite{bobrowski_vanishing_2017}. Nevertheless, we will provide a proof here for two reasons: (a) we will use a different change of variable than the one in \cite{bobrowski_vanishing_2017}, and (b) the steps of the proof and the definitions within will be of use for us later.

We start by applying Palm theory (Theorem \ref{thm:palm}), 
\eqb\label{eq:mean_ck}
	\mean{F_{k,r} } = \frac{n^{k+1}}{(k+1)!} \mean{g_{r}(\cX', \cP_n \cup\cX')},
\eqe
where $\cX'$ is an independent set of $\iid$ points.
The properties of the Poisson process $\cP_n$ imply
\[
\splitb
\cmean{g_{r}(\cX', \cP_n\cup\cX')}{\cX'=\bx} &= h_{r}(\bx) \prob{B(\bx)\cap\cP_n = \emptyset},\\
&= h_{r}(\bx) e^{-n\omega_d \rho^d(\bx)},
\splite
\]
where we used the definition of $B(\bx)$ as a ball of radius $\rho(\bx)$.
Therefore,
\eqb\label{eq:mean_integral}
\mean{g_{r}(\cX', \cP_n\cup\cX')} = \int_{(\T^d)^{k+1}} h_{r}(\bx) e^{-n\omega_d \rho^d(\bx)}d\bx.
\eqe
In order to evaluate the last integral we will be using a Blaschke-Petkantschin formula,  discussed in Appendix \ref{sec:bp}.
The main idea is to use a generalized polar-coordinate system, following the transformation $\cX \to (c(\cX), \rho(\cX),\Pi(\cX),\theta(\cX))$ presented in Section \ref{sec:crit_faces}.  This way we transform  $\bx \in (\T^d)^k$ into $(c,\rho,\Pi,\bth)$ where $c\in \T^d$, $\rho\in [0,\infty)$, $\Pi \in \Gr(d,k)$, and $\bth\in (\S^{k-1})^{k+1}$.
Notice that using the these new variables, we have that $\hcrit$ is independent of $c,\rho,\Pi$, and therefore we will denote $\hcrit(\bx) = \hcrit(\bth)$.
Thus, applying the BP-formula \eqref{eq:bp_torus_sep}, we have
\eqb\label{eq:mean_gr}
\mean{g_{r}(\cX', \cP_n\cup\cX') }=  \Abp\int_{r}^{\rmax} \rho^{dk-1} e^{-n\omega_d\rho^d} d\rho \int_{(\S^{k-1})^{k+1}} \hcrit(\bth)(\vsimp(\bth))^{d-k+1}d\bth,
\eqe
where $\vsimp$ is the volume of the $k$-simplex spanned by $\bth$.
Taking the change of variables $n\omega_d \rho^d \to t$ we have 
\[
\int_{r}^{\rmax}  \rho^{dk-1} e^{-n\omega_d\rho^d} d\rho = \frac{1}{d (n\omega_d)^k}\int_{\Lambda}^{\Lambda_{\max}} t^{k-1} e^{-t}dt = \frac{(k-1)!}{d (n\omega_d)^k}\param{e^{-\Lambda} \sum_{j=0}^{k-1}\frac{\Lambda^j}{j!} - e^{-\Lambda_{\max} }\sum_{j=0}^{k-1}\frac{\Lambda_{\max}^j}{j!}},
\]
where we used properties of the incomplete Gamma function.
Defining
\eqb\label{eq:A_crit}
D_k := \frac{\Abp}{(k+1)!d \omega_d^k}\int_{(\S^{k-1})^{k+1}}\hcrit(\bth) (\vsimp(\bth))^{d-k+1}d\bth,
\eqe
and putting everything back into \eqref{eq:mean_ck} completes the proof.

\end{proof}

\begin{proof}[Proof of Proposition \ref{prop:mean_var} - Part I (expectation)]
Just use 
 Lemma \ref{lem:mean_ck}, and take the leading order term (recall that $\Lambda\to\infty$). This proves the result for the expectation. The proof for the variance is considerably more intricate, and requires  definitions that will only be presented later. It is therefore postponed to Section \ref{sec:pois_limit}.

\end{proof}
Note that from Proposition \ref{prop:mean_var} we have that if $w(n)\to \infty$ then
\[
\limninf \mean{ F_{k,r}} = \begin{cases} 0 & \Lambda = \thres + w(n),\\
	\infty & \Lambda = \thres - w(n).\end{cases}
\]
This phase-transition for the expectation leads us to prove Proposition \ref{prop:pt_crit}.


\begin{proof}[Proof of Proposition \ref{prop:pt_crit}]

Taking $\Lambda = \thres + w(n)$ we have that $\mean{F_{k,r}} \approx D_k e^{-w(n)} \to 0$. Using Markov's inequality we therefore have that $\prob{F_{k,r}=0} \to 1$.

For $\Lambda = \thres - w(n)$, we use Chebyshev's inequality.
\[
	\prob{F_{k,r} = 0} = 
		\prob{F_{k,r} \le 0} \le \prob{|F_{k,r} - \mean{F_{k,r}}| \ge \mean{F_{k,r}}} \le \frac{\var{F_{k,r}}}{\mean{F_{k,r}}^2}.
\]
Using Proposition \ref{prop:mean_var} we have $\var{F_{k,r}}\approx \mean{F_{k,r}}$, and since $\mean{F_{k,r}} \approx D_k e^{w(n)}\to \infty$, we have $\prob{F_{k,r}=0} \to 0$, completing the proof.

\end{proof}

To conclude, in this section we proved that at $\Lambda = \log n + (k-1)\logg n + w(n)$ the critical $k$-faces vanish. This implies that no more $k$-cycles are created at larger radii, and also no $(k-1)$-cycles are waiting to be terminated. These statements, however, are not enough to prove the phase-transition statements in Theorems \ref{thm:pt_hk_1} and \ref{thm:pt_hk_2}. To prove these theorems, we will need to prove similar phase transitions for the negative and positive critical faces separately, which is the purpose of the next section.

\section{Phase transitions for positive and negative critical faces}
\label{sec:pn_faces}

In order to close the gap between the phase transition in Proposition \ref{prop:pt_crit} and the ones in Theorems \ref{thm:pt_hk_1}, we will use the notion of negative critical faces discussed in Section \ref{sec:crit_faces}.
 Our main goal in this section is to prove the following.
\begin{prop}\label{prop:pt_neg}
 Let $1 \le k \le d-2$, and $w(n)\to 0$. Then,
\[
	\limninf \prob{ \Cp_{k,r} = 0} = \limninf \prob{ \Cn_{k+1,r} = 0} = \begin{cases} 1 & \Lambda = \thres + w(n),\\
	0 & \Lambda = \thres - w(n). \end{cases}
\]
\end{prop}
Notice that Proposition \ref{prop:pt_neg} is \emph{not} a simple corollary of Proposition \ref{prop:pt_crit}. From Proposition \ref{prop:pt_crit} we can conclude that $F_{k+1,r}$ vanishes at $\Lambda = \log n + k\logg n +w(n)$. However, here we want to show that $\Cn_{k+1,r}$ vanishes \emph{earlier}, i.e.~at $\Lambda = \log n + (k-1)\logg n +w(n)$. In addition, the fact that $F_{k,r}>0$ at $\Lambda = \thres - w(n)$ does not guarantee that $\Cp_{k,r}>0$.

The proof of Proposition \ref{prop:pt_neg} is considerably more challenging than Proposition \ref{prop:pt_crit} for the following reason.
Determining whether a $k$-simplex $\cX\in \cC^k(\cP_n)$ is critical requires us to verify local conditions (represented by $\hcrit,\gcrit$). However, once a critical face is spotted, determining whether it is positive or negative is a much harder and non-local task.
 Nevertheless, in \cite{bobrowski_vanishing_2017} we discovered a \emph{sufficient} local condition to verify that a face is positive. In this section we will revisit this condition, and slightly modify it for our purposes.
 This sufficient condition for positive faces, immediately implies a \emph{necessary} condition for the existence of negative critical faces, which in turn will provide an upper bound for $\Cn_{k,r}$. That will be the heart of the proof of Proposition \ref{prop:pt_neg}.
We will divide the proof into a few key steps.
First, we a quick detour, to discuss the main (non-random) topological and geometric ingredients we will use later.

\subsection{Topological ingredient:\\
A sufficient condition for positive faces}\label{sec:suff_cond}

In this section we revisit the sufficient condition from \cite{bobrowski_vanishing_2017} with some modifications.
This sufficient condition relies heavily on a quantity that measures the distance between the center of a simplex $c(\cX)$ and the nearest face on its boundary $\partial \sigma(\cX)$, see Figure \ref{fig:phi}(a). 

\begin{figure}[ht]
    \centering
    \begin{subfigure}{0.45\textwidth}
    \centering
    \includegraphics[scale=0.25]{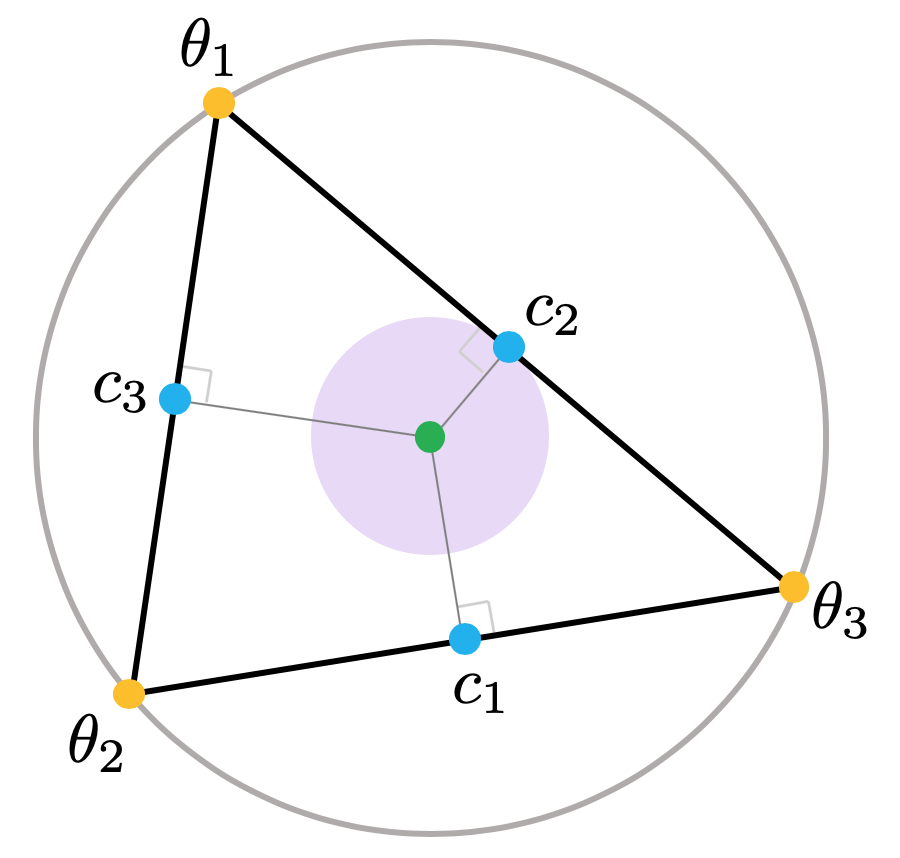}
    \caption{}
    \end{subfigure}
  \begin{subfigure}{0.45\textwidth}
  \centering
    \includegraphics[scale=0.25]{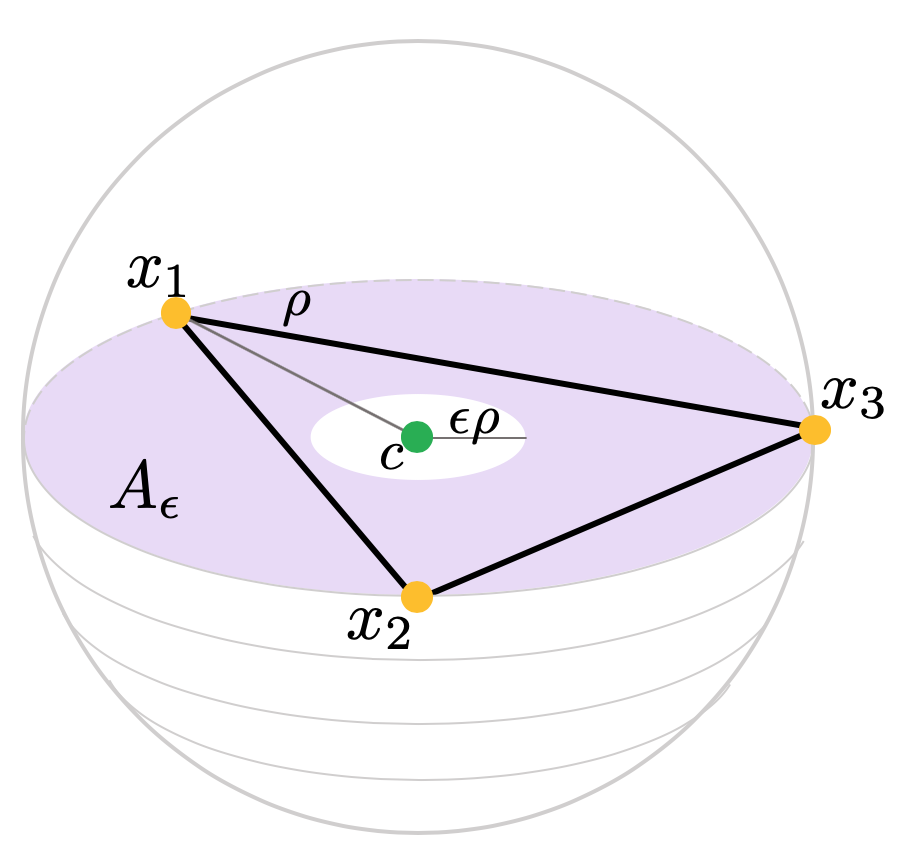}
    \caption{}
    \end{subfigure}
    \caption{(a) The distance to the  nearest center, denoted by  $\phi$. In this example $k=2$, so $\S^{k-1}$ is a circle, and we have 3 points on it -- $\bth = (\theta_1,\theta_2,\theta_3)$. Each of the $1$-faces (edges) has its own center, so that $c_{i} = c(\hat\bth_i)$. Then, $\phi(\bth) = \min(\abs{c_{1}}, \abs{c_{2}}, \abs{c_{3}}) = \abs{c_{2}}$. Notice that $\phi$ in this example is also the radius of the largest ball centered at the origin and inscribed by the simplex $\sigma(\bth)$. This is always true when $\hcrit(\bth)=1$. (b) The annulus $A_{\eps}(\bx)$. Here $\bx = (x_1,x_2,x_3)\in (\R^3)^3$, and $A_\eps(\bx)$ is a 3-dimensional annulus around $c=c(\bx)$ with radii in $[\eps\rho(\bx),\rho(\bx)]$. Notice that in this example, $\eps < \phi(\bx)$, and then we have $\partial\sigma(\bx)\subset A_\eps(\bx)$.}
    \label{fig:phi}
\end{figure}

Let $\bth = (\theta_1,\ldots,\theta_{k+1})\in (\S^{k-1})^{k+1}$ be in general position on the $(k-1)$-sphere. Denote $\hat \bth_i := \bth\bs\set{\theta_i} :=  (\theta_1,\ldots, \theta_{i-1}, \theta_{i+1},\ldots,\theta_{k+1}) \in (\S^{k-1})^k$, and define
\eqb\label{eq:phi_T}
    \phi(\bth) = \min_{1\le i\le k} |c(\hat\bth_i)|.
\eqe
The quantity $\phi(\bth)$ measures the distance of the closest center to the origin, see Figure \ref{fig:phi}(a). Next, for any $\bx \in (\T^d)^{k+1}$ recall the definition of $\theta(\bx)$ \eqref{eq:crit_face_1} as the spherical coordinates of the points in $\bx$. Then, we can define
\eqb\label{eq:phi_X}
    \phi(\bx) := \phi(\theta(\bx)).
 \eqe

Next, recall the definitions of $c(\bx),\rho(\bx)$ \eqref{eq:crit_face_1} and $B(\bx)$ \eqref{eq:crit_face_2}, and define
\eqb\label{eq:A_eps}
A_{\eps}(\bx) := \mathrm{cl}({ B(\bx) } \backslash B_{\eps \rho(\bx)}(c(\bx))),
\eqe
where $\mathrm{cl(\cdot)}$ stands for the closure, and $B_r(x)$ is the ball of radius $r$ around $x$. See Figure \ref{fig:phi}(b). In other words $A_\eps$ is a $d$-dimensional annulus centered at $c(\bx)$ with radii in the range $[\eps\rho,\rho]$.
The following is a slight variation of Lemma 7.1 in \cite{bobrowski_vanishing_2017}, and provides a sufficient condition using $\phi(\bx)$ and $A_{\eps}(\bx)$.

\begin{lem}\label{lem:crit_positive}
Let $\cP \subset \T^d$, and let $\cX\in\cC^k(\cP)$ be a critical $k$-face, with $2\le k \le d-1$. If there exist $\rho_0 < \rho(\cX)$ and $\eps_0 \le \phi(\cX)$ such that 
\[
A_{\eps_0}(\cX) \subset B_{ \rho_0}(\cP),
\] 
then $\cX$ is a \emph{positive} critical face.
\end{lem}

\begin{rem}
Lemma 7.1 in \cite{bobrowski_vanishing_2017} is very similar to Lemma \ref{lem:crit_positive} here. The main difference is that in \cite{bobrowski_vanishing_2017} instead of $\phi(\cX)$ we used 
\eqb\label{eq:phi_bw}
\frac{1}{2}\frac{\inf_{x\in \partial\sigma(\cX)}\abs{x-c(\cX)}}{\rho(\cX)}.
\eqe
This quantity can also be defined as half the radius of the largest ball around $c(\cX)$ that does not intersect the boundary of $\sigma(\cX)$.
The differences between the definitions are twofold. The first difference is that we wish to remove the $\frac{1}{2}$ factor. The second is that we defined $\phi$ via the distances to the $(k-1)$-planes containing the faces of $\sigma(\cX)$, whereas in \eqref{eq:phi_bw} we take the distances to the faces themselves. One can show that if $\hcrit(\cX) = 1$ these two definitions coincide (as in Figure \ref{fig:phi}(a)), but otherwise this is not  true.
\end{rem}

\begin{proof}[Proof of Lemma \ref{lem:crit_positive}]
The main idea in this proof is the same as in Lemma 7.1 in \cite{bobrowski_vanishing_2017}. We will repeat the shared parts briefly for completeness. 
Intuitively, if $\cX$ is a negative critical $k$-face, then the nontrivial $(k-1)$-cycle it terminates is represented by the boundary $\partial \sigma(\cX)$. 
Under the conditions in the lemma, $\partial \sigma(\cX)$ can also be viewed as a (singular) $(k-1)$-cycle in the annulus $A_{\eps_0}(\cX)$. However, since the $(k-1)$-th homology of the annulus is trivial ($A_\eps \simeq\S^{d-1}$), this means that $\partial\sigma(\cX)$ is a trivial $(k-1)$-cycle, and therefore $\cX$ cannot be negative. For a more elaborated intuitive explanation we refer the reader to \cite{bobrowski_vanishing_2017}.
In the following, we use $\bC_k, \bZ_k, \bB_k$ to represent the chains, cycles, and boundaries groups, respectively (we use the bar notation to avoid confusion with other symbols defined earlier).

For a given simplex $\cX\in \cC^k(\cP)$, 
fix $\sigma = \sigma(\cX), c = c(\cX), \rho=\rho(\cX), B=B(\cX), A_{\eps_0} = A_{\eps_0}(\cX)$, and $\phi= \phi(\cX)$. Also denote $\cC_r = \cC_r(\cP)$ and $B_r = B_r(\cP)$.
In the following we will refer to $\sigma$ both as a geometric simplex and as a chain in $\bC_k$.
Notice that since $\sigma$ is added to the \cech complex at radius $\rho$, and since $\hcrit(\cX)=1$, then there exists $\rho_1<\rho$ such that all the faces of $\sigma$ are included in $\cC_{\rho_1}$. We will denote $\rho_2 = \max(\rho_0,\rho_1) < \rho$, so that at $\rho_2$ we have both coverage of the annulus $A_{\eps_0}$ and we know that $\partial\sigma$ is in the complex $\cC_{\rho_0}$.

Next, notice that $\partial\sigma$ is  a $(k-1)$-chain in $\cC_{\rho_2}$ (i.e.~$\partial \sigma \in \bC_{k-1}(\cC_{\rho_2})$), and moreover it is a $(k-1)$-cycle (i.e.~$\partial\sigma\in \bZ_{k-1}(\cC_{\rho_2})$, since $\partial^2\sigma = 0$). We will show next that $\partial \sigma$ is also a boundary in $\cC_{\rho_2}$, denoted $\partial\sigma\in \bB_{k-1} (B_{\rho_2})$, i.e.~there exists $\gamma \in \bC_k(\cC_{\rho_2})$ such that $\partial \sigma = \partial \gamma$ (recall that $\sigma\not\in \cC_{\rho_2}$). We argue that if indeed $\partial\sigma\in \bB_{k-1} (B_{\rho_2})$, then adding the critical face $\sigma$ must generate a new $k$-cycle.
To show this, consider the chain $(\gamma - \sigma)\in \bC_k(\cC_\rho)$. This is a $k$-cycle (since $\partial(\gamma-\sigma) = 0$), and it is not a boundary, since as $\sigma$ does not have any co-faces in $\cC_{\rho}$ (due to condition (2) in Lemma \ref{lem:crit_face}).
In other words,  if $\partial\sigma$ is a boundary in $\cC_{\rho_2}$, then $\sigma$ is a positive $k$-face. Thus, to conclude the proof, we need to show that our assumptions imply that $\partial\sigma$ is indeed a boundary in $\cC_{\rho_2}$.

The main idea in the proof of Lemma 7.1 \cite{bobrowski_vanishing_2017} used the following  observations.
\begin{enumerate}
    \item As a subset of $\T^d$, the boundary $\partial \sigma$ satisfies $\partial\sigma \subset A_{\eps_0}$, for   any $\eps_0\le \phi$, 
    \item If $A_{\eps_0} \subset B_{\rho_0}\subset B_{\rho_2}$, then as a singular $(k-1)$-chain, we have $\partial \sigma\in \bB_{k-1}(B_{\rho_2})$. 
\end{enumerate}
The first claim is true from the definition of $\phi$. Recall that $\phi \rho$ is the distance between the center $c$ and the nearest of the $(k-1)$-planes that contain the faces of $\sigma$ (see Figure \ref{fig:phi}(a)). 
In particular, all points on the boundary of $\sigma$ are at distance at least $\phi\rho$ from $c$. Therefore, $\partial \sigma \subset \mathrm{cl}(B\bs B_{\phi\rho}(c)) \subset A_{\eps_0}$, see Figure \ref{fig:phi}(b).
For the second claim notice that since $A_{\eps_0}$ is homotopy equivalent to $\S^{d-1}$, we have that $H_{k-1}(A_{\eps_0})\cong H_{k-1}(\S^{d-1}) = 0$ for $2\le k\le d-1$. This means that every singular cycle $\gamma \in \bZ_{k-1}(A_{\eps_0})$ must be a boundary in $\bB_{k-1}(A_{\eps_0})$. However, since we assume that $A_{\eps_0}\subset B_{\rho_2}$, then we also have $\gamma\in \bB_{k-1}(B_{\rho_2})$.

Finally, the Nerve Lemma \ref{lem:nerve} provides a homotopy-equivalence between $\cC_{\rho_2}$ and $B_{\rho_2}$, which induces homomorphisms between the simplicial chains $\bC_k(\cC_{\rho_2})$ and the singular chains $\bC_k(B_{\rho_2})$. These homomorphisms preserve cycles and boundaries.
Since we showed that  $\partial\sigma\in \bB_{k-1} (B_{\rho_2})$, then we also have $\partial \sigma\in \bB_{k-1} (\cC_{\rho_2})$. This completes the proof.

\end{proof}

\subsection{Geometric ingredient: \\
The distance to the boundary}\label{sec:dist_bound}

The sufficient condition for positive faces we presented in the previous section relies heavily on $\phi(\cX)$.
In this section we provide some quantitative estimates, that will be used later.

Let $\bth = (\theta_1,\ldots, \theta_k) \in (\S^{k-1})^k$ be $k$ points on the unit $(k-1)$-sphere. 
Define
\eqb\label{eq:S_k}
    S_k(\alpha) := \set{\bth\in (\S^{k-1})^k : |c(\bth)| \le \alpha} \subset (\S^{k-1})^k.
\eqe
Notice that $|c(\bth)|$ measures the distance 
between the $(k-1)$-dimensional affine plane containing $\bth$ and the origin (the center of the sphere), see Figure \ref{fig:phi}(a).

\begin{lem}\label{lem:V_k_lip}
For $k\ge 2$, define
\[
    V_k(\alpha) := \vol(S_k(\alpha)) = \int_{(\S^{k-1})^k} \indf{|c(\bth)|\le \alpha} d\bth.
\]
Then $V_k(\alpha)$ is a Lipschitz function in $[0,L]$ for any $L<1$.
\end{lem}

\begin{proof}
Define
\[
f(\bth) := \indf{|c(\bth)|\le \alpha}.
\]
To evaluate $V_k(\alpha)$ we will use  a version of the BP formula for integrating points on the sphere, presented in Lemma \ref{lem:bp_sphere}.
Since we assume the points $\bth$ are on the unit sphere $\S^{k-1}$, then  $\abs{c(\bth)}^2 + \rho^2(\bth) = 1$. 
Therefore, we can write $f(\bth) = \indf{\rho(\bth)\ge \sqrt{1-\alpha^2}}$, implying that $f$ is  shift and rotation  invariant.
Using the notation of Lemma \ref{lem:bp_sphere} we can write $f(\bth) = f(\rho\bvphi) = \indf{\rho \ge \sqrt{1-\alpha^2}}$.
Using the BP formula in \eqref{eq:BP_sphere} we  have,
\[
\splitb
V_k(\alpha) = 
\int_{(\S^{k-1})^k} f(\bth)d\bth &= \Abp^\circ \int_{\sqrt{1-\alpha^2}}^1 \rho^{k^2-2k} (1-\rho^2)^{-1/2}d\rho \int_{(\S^{k-2})^k} \vsimp(\bvphi)d\bvphi\\
&= \const \int_{0}^{\alpha} (1-t^2)^{(k^2-2k-1)/2} dt,
\splite
\]
where we used the change of variables $\rho\to \sqrt{1-t^2}$.
If $L<1$, then we observe that $V_k(\alpha)$ is continuously differentiable  in $[0,L]$ (for $k\ge 3$ this is true also for $L=1$), and therefore - Lipschitz.
This completes the proof. 

\end{proof}

Suppose next that $\bth = (\theta_1,\ldots,\theta_{k+1})\in (\S^{k-1})^{k+1}$ are in general position on the $(k-1)$-sphere, and recall the definition of $\phi(\bth)$ in \eqref{eq:phi_T}. Define
\[
S^{\phi}_k(\alpha) := \set{\bth \in (\S^{k-1})^{k+1} : \phi(\bth) \le \alpha},
\]
i.e.~$S_k^\phi(\alpha)$ contains all configurations where the nearest center is at distance less than $\alpha$.
\begin{cor}\label{cor:phi_lip}
For $k\ge 2$, define
\[
    V_k^\phi(\alpha) := \vol(S_k^\phi(\alpha)) = \int_{(\S^{k-1})^{k+1}} \indf{\phi(\bth)\le \alpha}d\bth.
\]
Then $V_k^\phi(\alpha)$ is Lipschitz in $[0,L]$, for any $L<1$.
\end{cor}

\begin{proof}
Let $0\le \alpha_1 < \alpha_2 \le L$.
As the minimum of $(k+1)$ variables, we have
\[
\indf{\phi(\bth)\in (\alpha_1,\alpha_2]} \le \sum_{i=1}^{k+1} \indf{|c(\hat\bth_i)| \in (\alpha_1,\alpha_2]}.
\]
Therefore,
\[
\splitb
    |V_k^\phi(\alpha_2)-V_k^\phi(\alpha_1)| &= \int_{(\S^{k-1})^{k+1}} \indf{\phi(\bth)\in (\alpha_1,\alpha_2]}d\bth \\
    &\le (k+1)  \int_{\S^{k-1}}d\theta_1\int_{(\S^{k-1})^k} \indf{|c(\hat\bth_1)|\in(\alpha_1,\alpha_2]}d\hat\bth_1 
    \\
    & = \const (V_k(\alpha_2)-V_k(\alpha_1))\\
    & \le \const|\alpha_2 - \alpha_1|,
\splite
\]
where in the last inequality we used Lemma \ref{lem:V_k_lip}. This concludes the proof.

\end{proof}


\subsection{The vanishing of $\Cn_{k+1,r}$}

We are now ready to prove Proposition \ref{prop:pt_neg}. To do so we want to find a bound for $\Cn_{k+1,r}$ that uses local conditions derived from Lemma \ref{lem:crit_positive}. To this end define
\[
\gn_{r,a,b}(\cX, \cP) := g_{r}(\cX, \cP) \indf{\phi(\cX) \in (a,b]} \indf{A_{a}(\cX) \not\subset B_{(1-a/3)\rho(\cX)}(\cP)},
\]
and 
\[
     F_{k,r,a,b} := \sum_{\cX\in \cC^k(\cP_n)} \gn_{r,a,b}(\cX,\cP_n).
\]
In other words, we count critical $k$-faces with $\phi \in (a,b]$, such that the annulus $A_a$ is \emph{not} covered, which is a necessary condition for a face to be negative, via Lemma \ref{lem:crit_positive}.
Thus, we will use $F_{k+1,r,a,b}$ to bound $\Cn_{k+1,r}$ and to prove Proposition \ref{prop:pt_neg}.

\begin{lem}\label{lem:Fk_phi}
For $k\ge 2$, if $\Lambda\to\infty$ and $\frac{1}{\Lambda}\le a < b <1$ then there exist constants $ \Dn_{k,1}, \Dn_{k,2}>0$ such that
\[
	\mean{F_{k,r,a,b}} \le (b-a)\Dn_{k,1}  n\Lambda^{k-1} e^{-\Lambda(1+ \Dn_{k,2} a)}.
\]

In addition, for any $\eps>0$ we have
\[
	\mean{F_{k,r,0,\eps}} \le\eps \Dn_{k,1}  n\Lambda^{k-1} e^{-\Lambda}. 
	\]
\end{lem}

\begin{proof}
Similarly to the expectation calculations in  the proof of Lemma \ref{lem:mean_ck}, we can write
\eqb\label{eq:mean_hat_ck}
	\mean{F_{k,r,a,b}} =  \frac{n^{k+1}}{(k+1)!}\int_{(\T^d)^{k+1}} h_{r}(\bx) \ind\set{\phi(\bx)\in(a, b]} e^{-n\omega_d \rho^d(\bx)} p_{a}(\bx)d\bx,
\eqe

where 
\[
	p_{a}(\bx) := \cprob{A_{a}(\bx) \not\subset B_{(1-a/3)\rho(\bx)}(\bx \cup \cP_n)}{B(\bx)\cap \cP_n = \emptyset}.
\]
Defining,
\eqb\label{eq:cond_prob}	\P_\emptyset(\cdot) = \cprob{\cdot}{B(\bx)\cap \cP_n = \emptyset},
\eqe

then
\[
	p_{a}(\bx) \le \P_\emptyset(A_{a}(\bx) \not\subset B_{(1-a/3)\rho(\bx)}(\cP_n)).
\]
In the following we assume that $\bx$ is fixed, and remove it from the notation (so that $\rho=\rho(\bx), B=B(\bx)$, etc.). 
For every $k \ge 1$, define
\[
\hat A_m = \mathrm{cl}(A_{ma} \bs A_{(m+1)a}) ,
\]
i.e.~$\hat A_m$ is an annulus with radii range $[ma\rho, (m+1)a\rho]$. Taking $M = \floor{1/a}$ we then have
\[
A_{a} \subset \bigcup_{m=1}^{M}  \hat A_m,
\]
see Figure \ref{fig:annuli}.
Therefore,
\eqb\label{eq:p_eps_bound}
	p_{a}(\bx) \le \sum_{m=1}^{M} \P_\emptyset(\hat A_m \not\subset B_{(1-a/3)\rho}(\cP_n)).
\eqe
Next, in each annulus $\hat A_m$ take a $(ma\rho/2)$-net denoted $\cS_m$, such that for all $x\in \hat A_m$ there exists $s\in\cS_m$  with $\abs{s-x}\le ma \rho/2$. By scaling, this is equivalent to taking a $\frac{1}{2}$-net for an annulus with radii $[1,1+1/m]$, and therefore $\abs{\cS_m} \le \abs{\cS_1} := S_{\max}$ where $S_{\max}>0$ is a positive constant.

\begin{figure}[ht]
\centering
\includegraphics[scale=0.2]{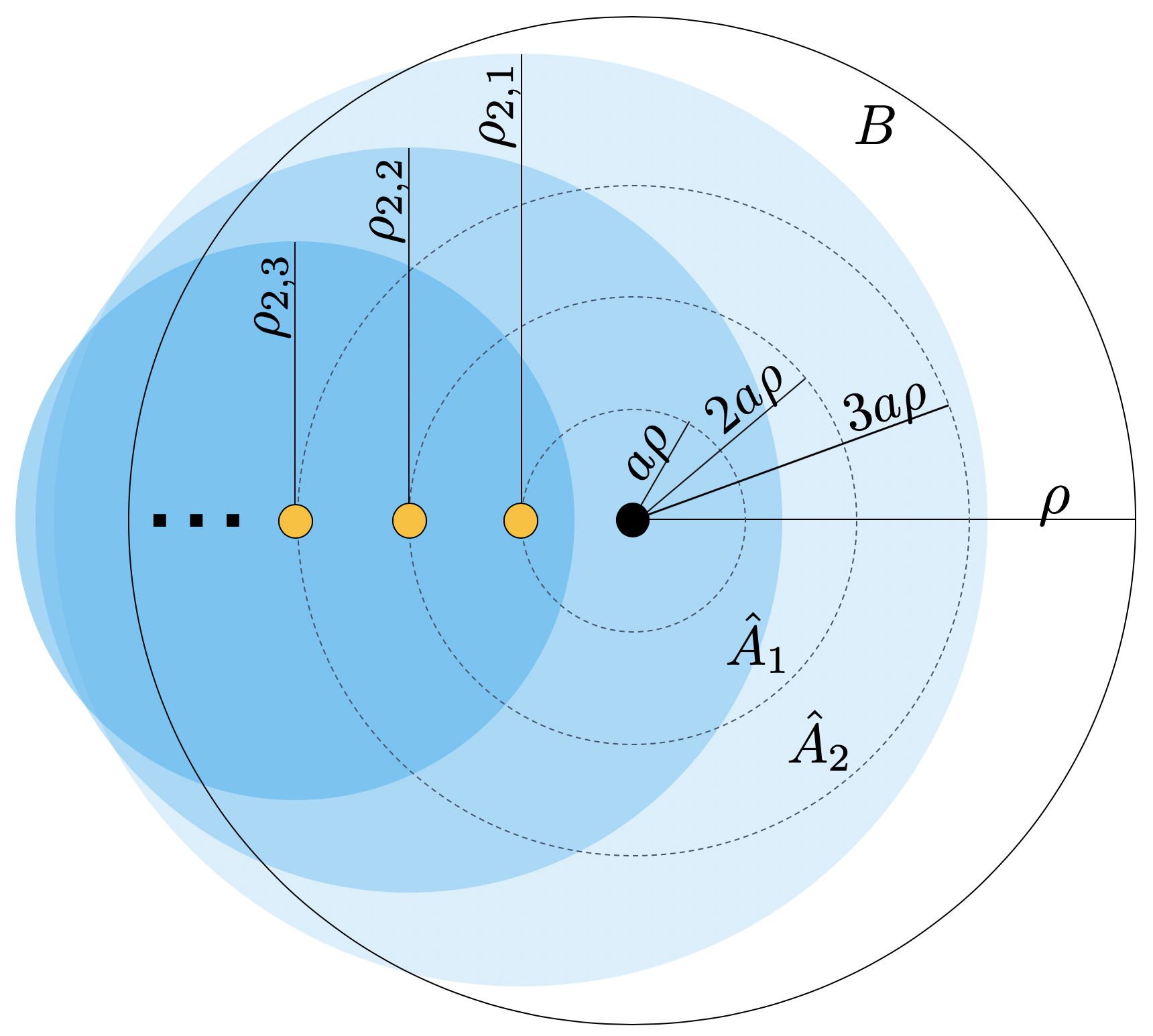}
\caption{\label{fig:annuli} Splitting annulus $A_a$ into a sequence of annuli $\hat A_1,\hat A_2,\ldots, \hat A_M$.  On each of the inner spheres we place a ball of radius $\rho_{2,m} = (1-(5/6)ma)\rho$. The area outside $B$ is $\vol(B_{\rho_2}\bs B)$ used in the proof.}
\end{figure}

Notice that if  $\hat A_m$ is not covered by $B_{(1-a/3)\rho}(\cP_n)$, then there exists $x\in \hat A_m$ such that $B_{(1-a/3)\rho}(x)\cap \cP_n = \emptyset$. By the triangle inequality,  there exists $s\in \cS_m$ such that $B_{(1-a/3-ma/2)\rho}(s)\cap \cP_n =\emptyset$.
Therefore, setting  $\rho_1 = (1-a(1/3+m/2))\rho$, we have
\[
\splitb
\P_\emptyset(\hat A_m \not\subset B_{(1-a/3)}\rho(\cP_n)) &\le  \sum_{s\in\cS_m} \P_\emptyset(B_{\rho_1}(s) \cap \cP_n = \emptyset) \\
&\le \abs{\cS_m} \cdot \max_{s\in \cS_m}\P_\emptyset\param{B_{\rho_1}(s) \cap \cP_n = \emptyset}.
\splite
\]
From \eqref{eq:cond_prob}, and using the spatial independence of the Poisson process,	 we have
\[
\P_\emptyset\param{B_{\rho_1}(s) \cap \cP_n = \emptyset} = e^{-n \vol(B_{\rho_1}(s) \backslash B)}.
\]
First, note that for all $m\ge 1$, we have  $\frac{1}{3}+\frac{m}{2} \le \frac{5}{6}m$. Therefore,
\[
\P_\emptyset\param{B_{\rho_1}(s) \cap \cP_n = \emptyset} \le  e^{-n \vol(B_{\rho_2}(s) \backslash B)},
\]
where $\rho_2 = (1-\frac{5}{6}ma)\rho$.
The points in $\hat A_m$ where the volume on the right hand side is the smallest, are the ones on the inner sphere, which are at distance  $m a \rho$ from the center $c$, see Figure \ref{fig:annuli}. For such points $s$,  using the notation of Lemma \ref{lem:vdiff} we have
\[
\vol(B_{\rho_2}(s)\bs B)= \rho^d \vdiff(ma,5/6).
\] 
From Lemma \ref{lem:vdiff}, there exists $C_1 > 0$ such that $\vdiff(ma,5/6) \ge C_1 ma$. Putting it all back into \eqref{eq:p_eps_bound} yields
\[
p_{a} \le S_{\max} \sum_{m=1}^{M}  \param{e^{-C_1  a n\rho^d}}^m \le S_{\max} \frac{e^{-C_1  a n\rho^d}}{1-e^{-C_1  a n\rho^d}} \le \const e^{-\Dn_{k,2} a \omega_d n\rho^d} \le \const e^{-\Dn_{k,2} a\Lambda},
\]
where $\Dn_{k,2} = C_1 /\omega_d$, and we used the facts that $\omega_d n \rho^d \ge \Lambda$ and $a \ge \frac{1}{\Lambda}$.

Putting this bound on $p_{a}$ back into \eqref{eq:mean_hat_ck}, and applying the BP formula \eqref{eq:bp_torus_sep} we have
\[
\splitb
	\mean{F_{k,r,a,b}} &\le \const n^{k+1} e^{-\Dn_{k,2}a\Lambda}\int_r^{r_{\max}} \rho^{dk-1} e^{-n\omega_d\rho^d}d\rho\\
	&\times \int_{(\S^{k-1})^{k+1}} \hcrit(\bth) \ind\set{\phi(\bth) \in (a,b]}\vsimp(\bth)d\bth\\
&\le \const n\Lambda^{k-1} e^{-\Lambda(1+ \Dn_{k,2} a)} \int_{(\S^{k-1})^{k+1}} \indf{\phi(\bth)\in (a,b]}d\bth,\splite
\]
where we used the fact that $\vsimp$ is bounded on the sphere.
Using Corollary \ref{cor:phi_lip} we have 
\[
\int_{(\S^{k-1})^{k+1}} \indf{\phi(\bth)\in (a,b]}d\bth = V_k^\phi(b) - V_k^\phi(a) \le \const(b-a).
\]
Putting it all together we have
\[
	\mean{F_{k,r,a,b}} \le \const (b-a) n\Lambda^{k-1} e^{-\Lambda(1+ \Dn_{k,2} a)}.
\]
Taking $D_{k,1}$ to be the last value of $\const$,  completes the first part of the proof.
For the second part, we start the same way as  \eqref{eq:mean_hat_ck} with $a=0,b=\eps$. Notice that in this case $A_0 = B\not \subset B_{\rho}(\cP_n)$, and therefore $p_0(\bx)=1$. This completes the proof.

\end{proof}

Finally, we are ready to prove Proposition  \ref{prop:pt_neg}. 

\begin{proof}[Proof of Proposition \ref{prop:pt_neg}]

 
Let $1\le k \le d-2$, and suppose that $\cX\in \cC^{k+1}(\cP_n)$ is a negative critical $(k+1)$-face. 
Using Lemma \ref{lem:crit_positive},  since $\cX$ is negative, we have that for every $\rho_0 < \rho(\cX)$ and every $\eps_0\le \phi(\cX)$, 
\[
A_{\eps_0} \not\subset B_{\rho_0}(\cP_n).
\]
Therefore, enumerating over all possible values of $\phi(\cX)$ we have
\[
\Cn_{k+1,r} \le F_{k+1,r,0,\eps_1} + \sum_{m=1}^{M-1} F_{k+1,r,\eps_m, \eps_{m+1}},
\]
for any sequence $\eps_1 < \eps_2 < \cdots < \eps_M = 1$. Using Lemma \ref{lem:Fk_phi}, we have
\[
    \mean{\Cn_{k+1,r}} \le \Dn_{k,1} n\Lambda^{k}e^{-\Lambda}\param{\eps_1 + \sum_{m=1}^{M-1}(\eps_{m+1}-\eps_m) e^{-\Dn_{k,2}\Lambda \eps_m}}.
\]
Taking $\eps_m = m/\Lambda$ we then have
\[
\mean{\Cn_{k+1,r}} \le \Dn_{k,1} n\Lambda^{k-1}e^{-\Lambda} \sum_{m=0}^\infty e^{-m\Dn_{k,2}}= \const n\Lambda^{k-1}e^{-\Lambda}.
\]
Finally, taking $\Lambda =\thres +w(n)$ we have
\[
    \mean{\Cn_{k+1,r}} \le \const e^{-w(n)} \to 0,
\]
and using Markov's inequality proves that $\prob{\Cn_{k+1,r}=0}\to 1$.

Next, consider $\Cp_{k,r}$, and suppose that $\Lambda = \thres - w(n)$. Assume first that $w(n) = o(\logg n)$.
 Then we can write $\Lambda = \log n + (k-2)\logg n + \tilde w(n)$, where $\tilde w(n) = \logg n -w(n)\to\infty$.
 Therefore, using what we just proved, we have w.h.p~$\Cn_{k,r}=0$, implying that $\Cp_{k,r} = F_{k,r}$, so that all critical $k$-faces are \emph{positive}.
 Using Proposition \ref{prop:pt_crit} we conclude that $\prob{\Cp_{k,r} = 0} \to 0$. 
For larger possible $w(n)$ (smaller $r$) we just use the fact that $\Cp_{k,r}$ is decreasing in $r$. 
On the other hand, when $\Lambda = \thres + w(n)$ from Proposition \ref{prop:pt_crit} we have $F_{k,r} = 0$ \whp~implying that $\prob{\Cp_{k,r}=0}\to 1$ as well.  This proves the phase-transition for $\Cp_{k,r}$.

Finally, from Lemma \ref{lem:perc} below, 
we know that  $\Cn_{k+1,r} \ge \Cp_{k,r}$. When $\Lambda = \thres - w(n)$, we have that \whp\ $\Cp_{k,r} >0$, and therefore $\Cn_{k+1,r} > 0$. That completes the proof.

\end{proof}

\subsection{Big vs.~small cycles}
In this section we will show that for $\Lambda \ge \log n$, \whp~all the $k$-cycles of the torus ($k<d$) already appear in the complex $\cC_r$. We name these cycles ``big", and notice that these cycles are never terminated.  Therefore, all the positive critical $k$-faces we discuss in this paper generate ``small" $k$-cycles that must be eventually terminated by a matching $(k+1)$-negative face. This, in particular, implies the fact that \whp\ $\Cn_{k+1,r}\ge \Cp_{k,r}$, which we used earlier.

More formally, let $i_*: H_k(B_r(\cP_n))\to H_k(\T^d)$ be the map induced by the inclusion $i:B_r(\cP_n)\hookrightarrow \T^d$. From the Nerve Lemma \ref{lem:nerve}, there is an isomorphism $f_*:H_k(\cC_r) \to H_k(B_r(\cP_n))$, and we can define the map
\[
	j_* : H_k(\cC_r)\to H_k(\T^d), \quad j_* = i_*\circ f_*.
\]
The image $\im(j_*)$ represents all the nontrivial $k$-cycles in $\cC_r$ that are mapped to nontrivial $k$-cycles of the torus, i.e.~the big cycles. We will prove next that in our setting, \whp\ $\im(j_*) = H_k(\T^d)$ means that all the big $k$-cycles have a representative in $\cC_r$. Therefore, all the cycles that we will be discussing in this paper are small cycles that are  mapped zero in the homology of the torus, and are bound to disappear.
\begin{lem}\label{lem:perc}
Let $1\le k \le d-1$, and suppose that $\Lambda \ge \log n$. Then,
\[
	\limninf\prob{\im(j_*) = H_k(\T^d)} = 1.
\]
\end{lem}

The proof of Lemma \ref{lem:perc} is more geometric in nature, and makes a heavier use of the Morse-theoretic framework introduced in \cite{bobrowski_distance_2014, bobrowski_vanishing_2017}. For an introduction to Morse theory see \cite{milnor_morse_1963}. Its adaptation to the distance function can be found in \cite{ferry_when_1976, gershkovich_morse_1997}.

Let $\cP\subset \T^d$ be a finite set and define $d_{\cP}:\T^d\to \R$ to be the distance function 
\[
	d_{\cP}(x) = \min_{p\in \cP} d_T(x,p).
\]
In Morse theory we consider the sublevel-set filtration 
$\set{d_{\cP}^{-1}([0,r])}_r$. Then, according to Morse theory, homology changes only at critical levels, and it can be shown that 
critical points of $d_{\cP}$ with Morse index $k$ are responsible for either generating a $k$-cycle, or terminating a $(k-1)$-cycle in the sublevel-set filtration. For the full discussion and definitions, see \cite{bobrowski_vanishing_2017}.
Observing that $d_{\cP}^{-1}([0,r]) = B_r(\cP) \simeq \cC_r(\cP)$ has lead to the Morse theoretic analysis we used in this paper. The critical faces introduced in Section \ref{sec:crit_faces} are such that if $\cX$ is a critical $k$-face then $c(\cX)$ is a critical point of index $k$ for $d_{\cP}$ and $\rho(\cX)$ is the corresponding critical value.

For the proof of Lemma \ref{lem:perc}, we will 
consider the set $V_r(\cP) := \mathrm{cl}(\T^d\bs  B_r(\cP))$ -- the closure of the complement. The filtration $\set{V_r(\cP)}_r$ is in fact the superlevel-set filtration of $d_{\cP}$, and it can be shown that the changes in homology also occur at the critical levels, where an index $k$ critical point either generates a $(d-k)$-cycle or terminates a $(d-k-1)$-cycle.
In particular, the connected components in $V_r$ are generated by critical points of index $d$. Notice, that the superlevel-set filtration goes in a reversed order, i.e.~we let $r$ decrease from $\infty$ to $0$.

The main idea of the proof of Lemma \ref{lem:perc} is  that when $\Lambda \ge \log n$ the set $V_r(\cP_n)$ consists of tiny connected components, implying  that $B_r(\cP_n)$ is large enough to contain all the  big cycles. 

\begin{proof}[Proof of Lemma \ref{lem:perc}]

Recall that every connected component in $V_r(\cP_n)$ emerges from a critical point of index $d$ with $\rho\in (r,\rmax]$, and define
\[
	\hat F_{d,r} := \# \text{critical $d$-faces $\cX$ with $\rho(\cX) \in (r,\rmax]$, and such that $A_{r,4r}(\cX) \not\subset B_r(\cP_n)$,}
\]
where $A_{r_1,r_2}(\cX) := \mathrm{cl}(B_{r_2}(c(\cX))\bs B_{r_1}(c(\cX)))$ is a closed annulus centered at the critical point $c(\cX)$, with radii in $[r,4r]$. In Lemma \ref{lem:hat_F_d} we will show that when $\Lambda \ge \log n$ we have that \whp\ $\hat F_{d,r} = 0$. This implies that all the components of $V_r(\cP_n)$ are contained in balls of radius $r$.

 Suppose that $V_r(\cP_n)$ consists of  $M$ connected components (where $M$ is random) denoted $C_1,\ldots, C_M$. For each component we can then choose an index $d$ critical point $p_i \in C_i$ (there might be more than one point in each component, in which case we choose arbitrarily). 
 The fact that $\hat F_{d,r} = 0$ implies that (a) $V_r(\cP_n) \subset \bigcup_i B_r(p_i)$, and (b) for every $i\ne j$ we have $B_{2r}(p_i) \cap B_{2r}(p_j) = \ems$.
Lemma \ref{lem:MV}  is then used to prove that $\im( i_*) = H_k(\T^d)$. Recall that $j_* = i_*\circ  f_*$. Since $f_*$ is an isomorphism and $\im(i_*) = H_k(\T^d)$ we conclude that $\im(j_*) = H_k(\T^d)$ as well.
This completes the proof.

\end{proof}

To conclude this section, we have two lemmas to prove.

\begin{lem}\label{lem:hat_F_d}
If $\Lambda \ge \log n$, then 
\[
	\prob{\hat F_{d,r} > 0} \to 0.
\]
\end{lem}

\begin{proof}
Similarly to the proof of Lemma \ref{lem:Fk_phi},
\[
	\mean{\hat F_{d,r}} = \frac{n^{d+1}}{(d+1)!} \int_{(\T^d)^{d+1}} h_r(\bx) e^{-n\omega_d
	\rho^d(\bx)} p(\bx)d\bx,
\]
where
\[
p(\bx) := \cprob{A_{r,4r}(\bx) \not\subset  B_r(\cP_n\cup \bx)}{B(\bx)\cap \cP_n = \emptyset}.
\]
Fix $\eps \in (0,1/2)$, and define $r_0 = (1+\eps)r$. Since $\Lambda_0 = \omega_d nr_0^d \ge (1+\eps)^d\log n$, we know from Proposition \ref{prop:mean_var} that $\mean{F_{d,r_0}}\to 0$. Therefore, 
\[
	\mean{\hat F_{d,r}} \approx \frac{n^{d+1}}{(d+1)!} \int_{(\T^d)^{d+1}} h_r(\bx)\indf{\rho(\bx)\in (r,r_0]} e^{-n\omega_d
	\rho^d(\bx)} p(\bx)d\bx.
\]
Next, following \eqref{eq:cond_prob}	, we can write
\[
p(\bx)  \le \P_\emptyset(A_{r,4r}(\bx) \not\subset B_r(\cP_n)).
\]
To bound the last probability, we take a $(r/2)$-net $\cS$ on the anuulus $A_{r,4r}(\bx)$. The size of $\cS$ can be taken as constant. Then,
\eqb\label{eq:bound_p}
	\P_{\emptyset}(A_{r,4r}(\bx)\not\subset B_r(\cP_n)) \le  \sum_{s\in \cS}\P_{\emptyset}(B_{r/2}(s) \cap \cP_n = \emptyset),
\eqe
where
\[
\P_{\emptyset}(B_{r/2}(s) \cap \cP_n = \emptyset) = e^{-n \vol(B_{r/2}(s)\bs B(\bx))}.
\]
Notice that since $\rho(\bx)\le r_0$, we have 
\[
\vol(B_{r/2}(s)\bs B(\bx)) \ge \vol(B_{r/2}(s)\bs B_{r_0}(\bx)).
\]
The points where the volume on the right is the smallest, are on the inner sphere of $A_{r,4r}$ which is of radius $r$. In addition, since $\eps \le 1/2$, for every $s$ on that sphere, we have that a fixed fraction of the volume of the ball $B_{r/2}(s)$ is outside $B_{r_0}(c(\bx))$, and therefore $\vol(B_{r/2}(s)\bs B_{r_0}(c(\bx))) =\const r^d$. Putting everything back into \eqref{eq:bound_p} yields $
p(\bx) \le \const e^{-\const \Lambda}$, and therefore,
\[
\mean{\hat F_{d,r}} \le \const n\Lambda^{d-1} e^{-\Lambda(1+\const)}.
\]
When $\Lambda \ge \log n$, the last term  goes to zero. Using Markov's inequality completes the proof.

\end{proof}

\begin{lem}\label{lem:MV}
Let $S \subset \T^d$ be a closed set, and $S^c = \mathrm{cl}(\T^d \bs S)$. Suppose that there exist $p_1,\ldots, p_M \in \T^d$ and $r \in(0,\rmax)$ such that (a) $B_{2r}(p_i)\cap B_{2r}(p_j) =\ems$ for all $i\ne j$, and (b) $S^c \subset \bigcup_{i} B_{r}(p_i)$. 
Define $i_* : H_k(S)\to H_k(\T^d)$ the map induced by inclusion. Then 
\[
\im(i_*) = H_k(\T^d).
\]
\end{lem}

\begin{proof}
Define $A = \bigcup_i B_{2r}(p_i)$, and $B = \mathrm{cl}(\T^d \bs \bigcup_i B_r(p_i))$. Then $\T^d = \mathrm{int}(A) \cup \mathrm{int}(B)$, and therefore using the Meyer-Vietoris sequence (cf. \cite{hatcher_algebraic_2002}), we have the exact sequence
\[
	H_k(A)\oplus H_k(B) \xrightarrow{(i_*^{at}, i_*^{bt})} H_k(\T^d) \xrightarrow{\partial_*} H_{k-1}(A\cap B),
\]
where $i_*^{at}, i_*^{bt}$ are the maps induced by the inclusions $A\hookrightarrow \T^d$ and $B\hookrightarrow \T^d$, respectively.
Notice that $A$ is a union of disjoint balls, and therefore for all $k>0$ we have $H_k(A) = 0$. In addition, the intersection $A\cap B$ is a union of disjoint $d$-dimensional annuli, which are homotopy equivalent to  $(d-1)$-spheres. Therefore, for all $2\le   k \le d-1$ we have $H_{k-1}(A \cap B) = 0$. In other words, for every $1\le k \le d-1$, we have an exact sequence of the form
\[
	H_k(B) \xrightarrow{i_*^{bt}} H_k(\T^d) \xrightarrow{\partial_*} 0.
\]
Exactness implies that $\im(i_*^{bt}) = \ker(\partial_*) = H_k(\T^d)$. For $k=1$ we have the exact sequence
\[
	H_1(B) \xrightarrow{i_*^{bt}} H_1(\T^d) \xrightarrow{\partial_*} H_0(A\cap B) \xrightarrow{i_*^{ab}} H_0(A)\oplus H_0(B).
\]
Notice that there is a bijection between the components of $A$ (balls) and $A\cap B$ (annuli). Therefore, the map $i_*^{ab}$ is a bijection too. Exactness then implies that $\im(\partial_*) = \ker(i_*^{ab}) = 0$. In other words, $\partial_*$ is the zero map, as before.

Finally, since $S^c \subset \bigcup_i B_r(p_i)$ we have that $B \subset S$.  Consider the sequence
\[
	H_k(B) \xrightarrow{i_*^{bs}} H_k(S) \xrightarrow{i_*^{st}} H_k(\T^d),
\]	
where $i_*^{bs}, i_*^{st}$ are induced by the corresponding inclusion maps.
Since $i_*^{bt} = i_*^{st} \circ i_*^{bs}$, and  we showed that $\im(i_*^{bt}) = H_k(\T^d)$ we conclude that $\im(i_*^{bt}) = H_k(\T^d)$. That completes the proof.

\end{proof}

\section{Phase transitions for homological connectivity}\label{sec:pt_proof}

The foundations we laid in the previous sections allow us now to prove the main results of this paper - the sharp phase transition for homological connectivity presented in Theorems \ref{thm:pt_hk_1} - \ref{thm:pt_hk_2}.

 \begin{proof}[Proof of Theorem \ref{thm:pt_hk_1}]
 
 When $\Lambda = \log n + (d-1)\logg n + w(n)$, we showed in \cite{bobrowski_vanishing_2017}, based on the analysis in \cite{flatto_random_1977}, that \whp~$B_r(\cP_n) = \T^d$. In other words, the balls of radius $r$ cover the entire torus. 
 Since this is true for $r\to 0$, it must hold for $\rmax$ as well.
This implies that \whp~$H_k(\cC_{\rmax}) \cong H_k(\T^d)$.

Fix $1\le k \le d-2$.
 If $\Lambda = \thres + w(n)$,  using Proposition \ref{prop:pt_neg} we have  \whp~$\Cp_{k,r}=\Cn_{k+1,r}=0$. 
 Recall from Section \ref{sec:crit_faces}, that the $k$-th homology only changes at either positive $k$-faces or negative $(k+1)$-faces. Since there are no such faces in $(r,\rmax]$ we must have $H_k(\cC_{s}) \cong H_k(\cC_{\rmax})$ for all $s\in [r,\rmax]$.  On the other hand,  we argued above that $H_k(\cC_{\rmax}) \cong H_k(\T^d)$. Therefore, we conclude that \whp~for all $s\in [r,\rmax]$ 
 \[
    H_k(\cC_s) \cong H_k(\T^d),
 \]
 in other words -- $\cH_{k,r}$ holds.
 
 For the other direction, suppose that $\Lambda = \log n + (k-1)\logg n - w(n)$. From Proposition \ref{prop:pt_neg} we have that \whp~$\Cp_{k,r} >0$. Thus, there exist radii in $(r,\rmax]$ where new positive $k$-faces join the complex, adding new generators to $H_k$, implying that necessarily $\cH_{k,r}$ does not occur here.
 This completes the proof.
 
 \end{proof}

\begin{proof}[Proof of Theorem \ref{thm:pt_hk_2}]
When $\Lambda = \log n + (d-1)\logg n + w(n)$, we have that \whp~$B_r(\cP_n) = \T^d$, implying that $\cH_{d-1,r}$ and $\cH_{d,r}$ both hold.

When $\Lambda = \log n + (d-1)\logg n - w(n)$, then from Proposition \ref{prop:pt_crit} we have \whp~$F_{d,r}>0$. This, in particular implies that $\T^d \not \subset B_{r}(\cP_n)$ and therefore $H_d(\cC_r) = 0$, and $\cH_{d,r}$ does not hold.
Finally, recall that for this choice of $\Lambda$ we have $\Cp_{d,r} = 1$ (only a single $d$-cycle will be created), and therefore $\Cn_{d,r} = F_{d,r}-1$. The same second-moment argument used in the proof of Proposition \ref{prop:pt_crit}, can be used to show that \whp~$F_{d,r} > 1$, and thus $\Cn_{d,r}>0$. This implies that $\cH_{d-1,r}$ does not hold as well, completing the proof.

\end{proof}

\section{The structure of  cycles in the critical window}\label{sec:crit_window}

In this section we focus on the critical window, i.e.~the case where
\[
\Lambda = \thres + o(\logg n).
\]
As we discussed in Section \ref{sec:results}, we want to show that $k$-cycles in this regime exhibit a very simple structure, where the positive $k$-simplex generating the cycle is a face of the negative $(k+1)$-simplex terminating it. Next, we will turn this into a formal statement, and prove it.

Let $\cX = \set{x_1,\ldots,x_{k+1}}\in \cC^k(\cP)$, and define $\hat\cX_i := \cX\bs \set{x_i}$. Also, recall the definition of $\theta(\cX) = \set{\theta_1(\cX),\ldots, \theta_{k+1}(\cX)} $ \eqref{eq:crit_face_1}, and define $\hat\theta_i(\cX) := \theta(\cX) \bs \set{\theta_i(\cX)}$.
Next, define
\eqb\label{eq:xmin}
    \hat \cX_{\min} := \argmin_{\hcX_i}|c(\hat\theta_i(\cX)) |.\eqe 
In other words, $\hcXmin$ is the $(k-1)$-face of $\cX$ whose distance to the center is the smallest.
Notice that in the random setting, almost surely for every face the minimum is unique, and therefore $\hcXmin$ is well-defined.
We refer to the face $\hcXmin$ as the ``nearest" face of $\cX$. 

Our main argument in this section is that inside the critical window, if $\cX$ is a negative $(k+1)$-face then $\hcXmin$ is a positive $k$-face ($k\le d-2$), see Figure \ref{fig:iso_faces}.
This provides us with a bijection between positive and negative critical faces, that describes the creation and destruction of the last $k$-cycles in the complex, before reaching homological connectivity. In other words, all the last cycles in the complex are created and destroyed within the critical window, where the positive face is on the boundary of the negative face. To this end, define
\[
\splitb
    \Cnp_{k,r} := \#&\text{negative critical $k$-faces  $\cX\in\cC^k(\cP_n)$ with $\rho(\cX)\in (r,\rmax]$,}\\ &\text{and such that $\hcXmin$ is a positive critical $(k-1)$-face},
\splite
\]
Our main goal in this section is to prove the following lemma.

\begin{lem}\label{lem:pn_pairs}
Let $1 \le k \le d-2$. Suppose that $\Lambda = \thres + o(\logg n)$. Then,
\[	
\limninf \prob{ \Cp_{k,r} = \Cn_{k+1,r} = \Cnp_{k+1,r}} = 1.
\]
\end{lem}
The implications of this lemma are twofold. Firstly, $\Cn_{k+1,r} = \Cnp_{k+1,r}$ implies that all negative $(k+1)$-faces are such that their nearest $k$-face is positive. This is the pairing phenomenon we discussed earlier. Secondly, $ \Cp_{k,r} = \Cn_{k+1,r}$ implies that that all of the remaining $k$-cycles are of this pairs form, and  they are both generated and terminated after $r$. In other words, none of the $k$-cycles obstructing homological connectivity was created before $r$.


We will start by proving that $\Cn_{k+1,r} = \Cnp_{k+1,r}$, which will take most of this section. 
Recall, that for $\hcXmin$ to be a critical $k$-face, Lemma \ref{lem:crit_face} states two conditions that must hold. The first condition, $c(\hcXmin)\in\sigma(\hcXmin)$, holds deterministically, as stated in  the next lemma.

\begin{lem}\label{lem:xmin_crit}
For $k\ge 2$, let $\cX$ be a $k$-face satisfying $c(\cX)\in\sigma(\cX)$.
Assuming that $\hcX_{\min}$ is unique, then
\[
	c(\hcXmin)\in \sigma(\hcXmin).
\]
\end{lem}

\begin{rem}
In the following proof, as well as  other parts of this section, we  provide estimates for quantities related to points in the neighborhood of a critical face $\cX$.
Recall (Section \ref{sec:flat_torus}) that the neighborhood $B_{\rmax}(c(\cX))$ can be isometrically embedded as a concentric ball in $\R^d$. Thus, our analysis will implicitly treat  $\cX$ under this embedding, and $|\cdot|$ will denote the Euclidean norm in this neighborhood.
In addition, for any set of $m+1$ points $\cY$ in this neighborhood we define $\Pi(\cY)$ as the $m$-dimensional affine plane in $\R^d$ that contains $\cY$, under the embedding. We will also  use $d(x, \Pi)$ to denote the Euclidean distance between a point $x$ and a plane $\Pi$.
\end{rem}

\begin{proof}[Proof of Lemma \ref{lem:xmin_crit}]
Denote $c = c(\cX), \cmin = c(\hcXmin), \sigma = \sigma(\cX), \sigma_{\min} = \sigma(\hcXmin)$. 
Recall that $\sigma$ and $\sigma_{\min}$ are geometric open simplexes, and denote $\bar \sigma = \sigma \cup \partial \sigma$ - the closed simplex.
In addition, let $\Pi_{\min} = \Pi(\hcXmin)$ be the $(k-1)$-dimensional affine plane containing $\hcXmin$.
Suppose that $c \in \sigma$, and $\cmin \not\in \sigma_{\min}$.
Notice that if $\cmin \in \partial \sigma_{\min}$, then $\cmin$ belongs to the boundary of another $(k-1)$-face $\sigma'\subset \sigma$,  contradicting  $\sigma_{\min}$ being the unique nearest face. Therefore, we will  assume that $\cmin \not \in \bar\sigma_{\min}$.
In addition, since $\cmin \in \Pi_{\min}$, and $\bar\sigma \cap \Pi_{\min} = \bar\sigma_{\min}$, we have that $\cmin \not\in \bar\sigma$ as well. 
Now, consider the line segment connecting $c$ and $\cmin$, represented by $x_{\alpha} = \alpha c + (1-\alpha)\cmin$, $\alpha \in [0,1]$. 
Define
\[
	\alpha_{\min} = \min\set{\alpha \in [0,1] : x_\alpha \in \bar\sigma}.
\]
Note that  $\bar\sigma$ is compact, $x_0=\cmin \not\in \bar\sigma$, and $x_1= c\in \bar\sigma$. This implies that 
 that $\alpha_{\min} $ exists, is strictly positive, and $x_{\alpha_{\min}} \in \partial \sigma$. 
 Since the line segment is orthogonal to $\Pi_{\min}$, we have that $x_{\alpha_{\min}} \not \in \bar\sigma_{\min}$, and  $|x_{\alpha_{\min}}-c| \le |\cmin - c|$. In other words, we found a point on the boundary of $\sigma$ closer to $c$ than $\cmin$, which contradicts $\hcXmin$ being the nearest face. Since we reached a contradiction, we conclude that $\cmin\in \sigma_{\min}$, completing the proof. 
 
\end{proof}

The second condition in Lemma \ref{lem:crit_face}  requires that $B(\hcXmin)\cap \cP =  \emptyset$. Proving this part will take most of this section. To this end, we first take a short detour, and revisit the topological and geometric arguments discussed in Section \ref{sec:pn_faces}.

\subsection{Topological ingredient (revisited):\\
A sufficient condition for positive faces}

Lemma \ref{lem:crit_positive} provided a local sufficient condition for a critical face $\cX$ to be positive. The main idea was to show that $\partial\sigma(\cX)$ is contained in an annulus $A_{\eps_0}(\cX)$ 
\eqref{eq:A_eps}, that is covered for some $\eps_0 \le\phi(\cX)$. 
However, in some of the cases we will evaluate later, $\phi(\cX)$ will be too small to guarantee that $A_{\eps_0}(\cX)$ is covered (recall that $A_{\eps_0}(\cX) \subset B(\cX)$ and $B(\cX)\cap\cP =\emptyset$, so covering $A_{\eps_0}(\cX)$ is not ``easy"). Therefore, we wish to extend the ideas from Section \ref{sec:pn_faces} to find a sufficient condition that involves \emph{thinner} annuli that are easier to cover. It turns out, this can be done for those cases where $\hcXmin$ is \emph{not} critical.
In these cases, we
can replicate the arguments from Section \ref{sec:pn_faces} replacing $\partial \sigma(\cX)$ with a slightly larger $(k-1)$-cycle, which in turn is contained in a an annulus with smaller volume.


We start with a few definitions.
For $\cX\in \cC^k(\cP)$ recall that 
\[
\phi(\cX) = \min_{1\le i\le k+1} |c(\hat\theta_i(\cX))|
= \min_{1\le i\le k+1} \frac{d(c(\cX), \Pi(\hcX_i))}{\rho(\cX)}.
\] 
Next, let $\cX,\cY\in \cC^k(\cP)$ be such that $\cX\cap \cY = \hcXmin$. Let $\hat\cY_i$,\ $i=1,\ldots,k$,\   be all the $(k-1)$-faces of $\cY$ \emph{except} for $\hcXmin$, and
define
\eqb\label{eq:phi_X_Y}
\phi(\cX,\cY) := \min_{1\le i \le k} \frac{d( c(\hcXmin), \Pi(\hat\cY_i))}{\rho(\cX)},
\eqe
In other words, $\phi(\cX,\cY)$ measures the distance (normalized) between $c(\hcXmin)$ and the nearest face of $\cY$ (except for $\hcXmin$).
 Finally, we define
\[
\phmin(\cX, \cY) = \min(\phi(\cX,\cX), \phi(\cX,\cY)).
\]
Notice that $\phmin$ measures the distance between $c(\hcXmin)$ and the nearest face on the boundary of the chain $\sigma(\cX)+\sigma(\cY)$ (see Figure \ref{fig:vol_simp}(a)).
We will also need to consider a slightly modified annulus.
Let,
\eqb\label{eq:hrho}
\hrho(\cX) := \rho(\cX) + |c(\hcXmin)-c(\cX)| = \rho(\cX)(1+\phi(\cX)), 
\eqe
and define
\eqb\label{eq:hat_A}
\hat A_\eps(\cX) := \mathrm{cl}(B_{\hat \rho(\cX)}(c(\hcXmin)) \bs B_{\eps\rho(\cX)}(c(\hcXmin)).
\eqe
The annulus $\hat A_\eps(\cX)$ has the same inner radius as $A_\eps(\cX)$ (see \eqref{eq:A_eps}), while the outer radius is slightly larger. In addition, it is centered around $c(\hcXmin)$ rather than $c(\cX)$. See Figure \ref{fig:vol_simp}(a).

Next, recall that if $\cX$ is a critical $k$-face
then Lemma \ref{lem:xmin_crit}  shows that $c(\hcXmin)\in \sigma(\hcXmin)$.
Thus, from Lemma \ref{lem:crit_face}, if $\hat\cX_{\min}$ is \emph{not} critical, we must have that $B(\hcXmin)\cap \cP\ne \emptyset$. We want to prove a slightly stronger statement. Define
\eqb\label{eq:hat_B}
I(\cX) := B^\circ_{\rho(\cX)}\param{\bigcap_{x\in\hcXmin}B_{\rho(\cX)}(x)},
\eqe
where $B^{\circ}_r(S)$ is the union of \emph{open} $r$-balls around $S$. Notice that $B(\hcXmin)\subset I(\cX)$, and therefore showing that $I(\cX)\cap \cP=\emptyset$ 
implies that $\hcXmin$ is critical.
Moreover, it shows that $\hcXmin$ remains an isolated face throughout the interval $[\rho(\hcXmin),\rho(\cX))$.

The following lemma uses these new definitions to provide a modified sufficient condition for faces to be positive, in the spirit of Lemma \ref{lem:crit_positive}.

\begin{lem}\label{lem:crit_pos_2}
Let $\cX\subset \cP$ be a critical $k$-face, with $2\le k\le d-1$.
Suppose that there exists $x_* \in I(\cX)$ and set $\cX^* = \hcXmin \cup \set{x_*}$.
If there exist $\rho_0 < \rho(\cX)$ and $\eps_0 \le \phmin(\cX,\cX^*)$ such that 
\[ 
\hat A_{\eps_0}(\cX) \subset B_{\rho_0}(\cP),
\]
then $\cX$ is a positive critical face.
\end{lem}

\begin{proof}
Fix $\cX,\cX^*$, and denote 
$c = c(\cX), \rho = \rho(\cX), \sigma = \sigma(\cX), \sigma^* = \sigma(\cX^*), \hat A_\eps = \hat  A_\eps (\cX)$, and $\phmin = \phmin(\cX,\cX^*)$. In addition, denote by $\tau = \tau(\cX^*)$ the radius at which $\cX^*$ joins the \cech filtration (notice that $\tau$ can be smaller than $\rho(\cX^*)$, see Remark \ref{rem:half_sphere}). Since we assume $x^*\in I(\cX)$, we have that $\tau < \rho$. 

Recall that in Lemma \ref{lem:crit_face} the goal was to show that there exists $\rho_0<\rho$ such that $\partial \sigma \in \bB_{k-1}(\cC_{\rho_0})$ (i.e.~it is a $(k-1)$-boundary) and that proved that $\sigma$ was positive.
We will use a similar idea here, but for $\partial(\sigma + \sigma^*)$ instead. Figure \ref{fig:vol_simp}(a) provides a sketch for this setup.

As in the proof of Lemma \ref{lem:crit_positive}, there exists $\rho_1 < \rho$ such that $\partial \sigma$
is included in $\cC_{\rho_1}$. Taking $\rho_2 = \max(\rho_1,\tau)$, then $\rho_2<\rho$, and we know that both $\partial\sigma$ and $\sigma^*$ are in $\cC_{\rho_2}$ (while $\sigma$  is not). 
Notice that as $(k-1)$-chains in $\cC_{\rho_2}$, $\partial \sigma$ and $\partial(\sigma+\sigma^*)$ are homologous. Therefore, $\partial\sigma\in \bB_{k-1}(\cC_{\rho_2})$ if and only if $\partial(\sigma+\sigma^*)\in \bB_{k-1}(\cC_{\rho_2})$.

Finally, take $\rho_3 = \max(\rho_2,\rho_0)$, then we have that $\hat A_{\phmin}\subset B_{\rho_3}(\cP)$ and also both $\partial\sigma$ and $\sigma^*$ are in $\cC_{\rho_3}$. Notice that $\sigma,\sigma^* \subset B_{\hat \rho(\cX)}(\cmin)$, and by definition of $\phmin$ all the $(k-1)$-faces of both $\sigma$ and $\sigma^*$ (except for $\sigma(\hcXmin)$) are at distance of at least $\rho\phmin$ away from $\cmin$. Therefore, we conclude that 
the annulus $\hat A_{\phmin}$ contains  $\partial(\sigma+\sigma^*)$. As in the proof of Lemma \ref{lem:crit_positive} the fact that $\hat A_{\phmin}$ is covered by $B_{\rho_0}$ is thus enough to argue that $\partial(\sigma+\sigma^*) \in \bB_{k-1}(\cC_{\rho_3})$, i.e.~it is a boundary. As stated above this implies that $\partial \sigma \in \bB_{k-1}(\cC_{\rho_3})$, and therefore $\sigma$ is positive, concluding the proof.

\end{proof}

\begin{figure}[ht]
    \centering
\begin{subfigure}{0.45\textwidth}
\centering
\includegraphics[width=0.7\textwidth]{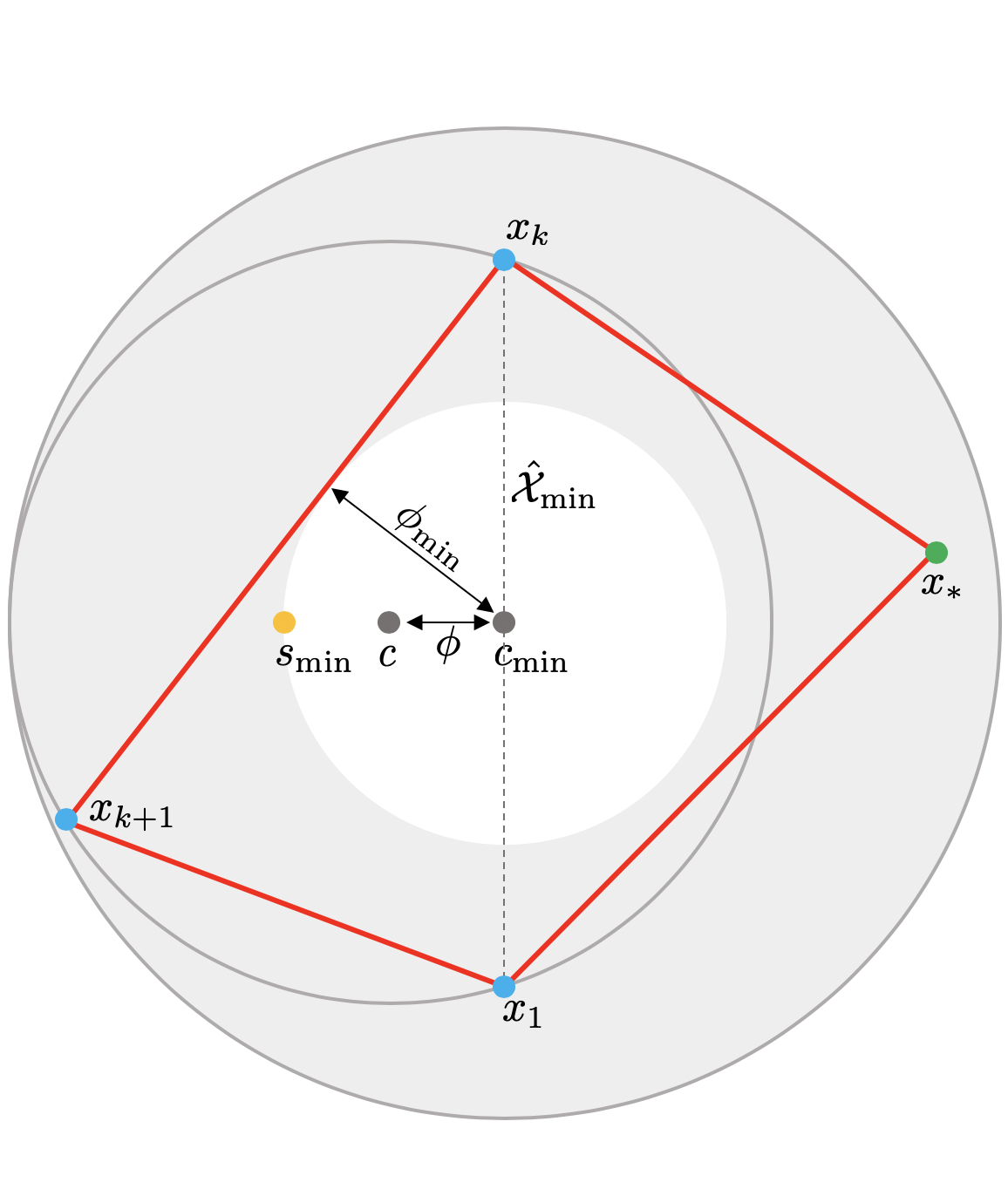}
    \caption{}
 \end{subfigure}
\begin{subfigure}{0.45\textwidth}
\centering
\includegraphics[width=0.7\textwidth]{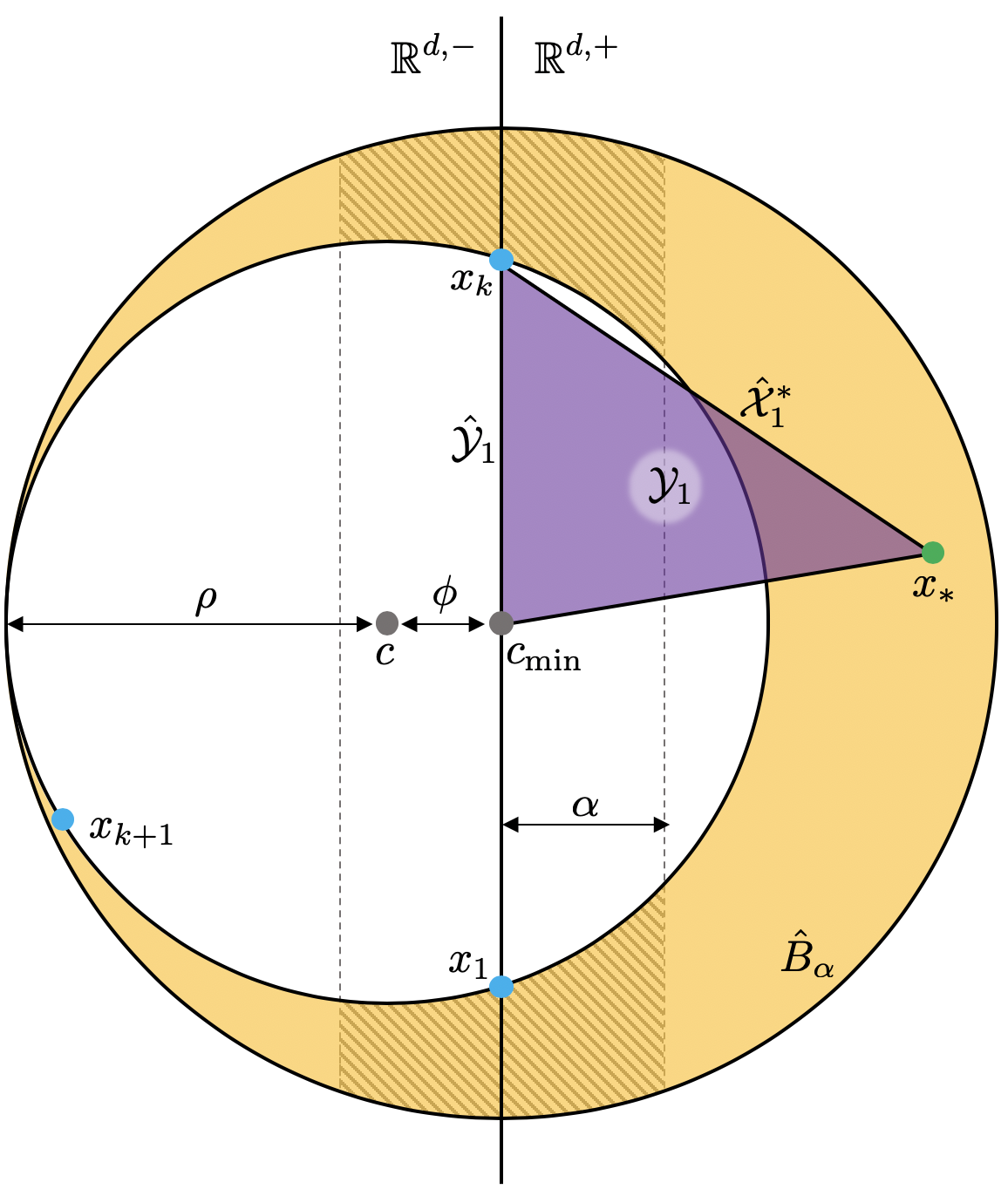}
    \caption{}
 \end{subfigure}
\caption{  \label{fig:vol_simp}
(a) A sketch for the new condition for positive faces. Here $\cX = \set{x_1,\dots,x_{k+1}}$, and $\cX^* = \hcX_{k+1}\cup \set{x_*}$. The value of $\phi$ is the distance between $c$ and the nearest face -- $\hcXmin = \hcX_{k+1}$ (the dashed edge), while $\phmin$ is the distance between $\cmin$ and the nearest face among $\hcX_1,\ldots, \hcX_k, \hcX^*_1,\ldots, \hcX^*_k$ (the red edges). The shaded area is the annulus $\hat A_{\phmin}$. Notice that $\partial \sigma(\cX) \not \subset \hat A_{\phmin}$, while  $\partial\sigma(\cX) + \partial \sigma(\cX^*)$ (the  cycle in red) is inside the annulus. If the  annulus is covered by the balls of radius $\rho_0 < \rho$, then this cycle must be trivial.
(b) A sketch for the estimates used in the proof of Lemma \ref{lem:dist_boundary}. The point configuration is the same as on the left. The half plane on the right (left) of $\hcXmin$ is denoted $\R^{d,+}$ ($\R^{d,-}$). The yellow shaded area is $\hat B$, where the no-pattern part is $\hat B_{\alpha}$. The simplex $\cY_1$ consists of the $(k-2)$-face $\hcX_{1,2} = \set{x_2,\ldots, x_k}$ together with the two vertices $\cmin$ and $x_*$.
}

\end{figure}

\subsection{Geometric ingredient (revisited):\\
The distance to the boundary}

To prove Proposition \ref{prop:pt_neg} we used the estimates on the distance to the nearest face on the boundary, discussed in Section \ref{sec:dist_bound}.
The proof of Lemma \ref{lem:pn_pairs}  requires much  subtler arguments that rely on  $\phmin$ as well, presented in this section. We will present all the lemmas together, before turning to the proofs, so that the reader may skip over the details.

Let $\bth = (\theta_1,\ldots,\theta_k) \in (\S^{k-1})^k$, and define $\hat\bth := (\theta_1,\ldots,\theta_{k-1}) \in (\S^{k-1})^{k-1}$. 
Recall the definition of $S_k(\alpha)\subset (\S^{k-1})^k$ in \eqref{eq:S_k}, and define
\[
S_k(\alpha,\beta) := \set{\bth \in (\S^{k-1})^k : \abs{c(\bth)}\le \alpha,\ |c(\hat\bth)-c(\bth)|\le\beta}  \subset S_k(\alpha).
\]

\begin{lem}\label{lem:vol_a_b}
For $k\ge 2$, define
\[
V_k(\alpha,\beta) := \vol(S_k(\alpha,\beta))
= \int_{(\S^{k-1})^k} \indf{|c(\bth)| \le \alpha}\indf{|c(\hat\bth)-c(\bth)|\le \beta}d\bth.
\]
If $\alpha < 1$ and $\beta/\sqrt{1-\alpha^2}< 1$, then there exists  $\const>0$ such
\[
V_k(\alpha,\beta) \le \const \frac{\alpha\beta}{\sqrt{1-\alpha^2}}.
\]
\end{lem}

For the next lemma, let $\bth \in (\S^{k-1})^{k+1}$.
Define $\hbth_i := \bth \bs \set{\theta_i}$ and $\hbth_{i,j} := \bth\bs\set{\theta_i,\theta_j}$,
and
\[
	\Scrit(\alpha,\beta,\gamma) := \set{\bth  : \hcrit(\bth) = 1,\ |c(\hbth_1)| \le \alpha,\ |c(\hbth_2)| \le \beta,\ |c(\hbth_{1,2})-c(\hbth_1)| \ge \gamma}\subset  (\S^{k-1})^{k+1}.
\]
The choice $i=1,j=2$ is arbitrary, and the following result holds for any other choice of  $i\ne j$.
\begin{lem}\label{lem:v_crit_a_b_c}
For $k\ge 2$ define,
\eqb\label{eq:v_crit}
\splitb
&\Vcrit(\alpha,\beta,\gamma) := \vol(\Scrit(\alpha,\beta,\gamma))\\
&\quad= \int_{(\S^{k-1})^{k+1}}  \hcrit(\bth) \indf{|c(\hbth_1)| \le \alpha,\ |c(\hbth_2)| \le \beta,\  |c(\hbth_{1,2})-c(\hbth_1)| \ge \gamma}d\bth.
\splite
\eqe
Suppose that $\beta = \beta(n)\to 0$ and $\gamma = \gamma(n) \to 0$ are such that $\beta =o(\gamma)$.
Then there exists $\const > 0$ such that for all $n$
\[
\Vcrit(\alpha,\beta,\gamma) \le \const \frac{\alpha\beta}{\gamma}.
\]
\end{lem}

Recall the definition $\hrho(\cX)$ \eqref{eq:hrho}, and define
\[
	\hat B(\cX) = B_{\hrho(\cX)}(c(\hcXmin)) \bs B(\cX),
\]
Note that $I(\cX) \subset B_{\hrho(\cX)}(c(\hcXmin))$. 
Recall that if $\cX$ is critical, then $B(\cX) \cap \cP = \ems$. We will aim to show that
 $\hat B(\cX)\cap \cP = \ems$ and so that $I(\cX) \cap \cP= \ems$,  implying that $\hcXmin$ is critical as well.
Let $\hat\Pi(\cX)$ be the $(d-1)$-dimensional affine plane  containing $c(\hcXmin)$ and orthogonal to the line through $c(\hcX)$ and $c(\hcXmin)$.
Define,
\eqb\label{eq:B_a}
    \hat B_\alpha(\cX) := \set{x\in \hat B(\cX) : d(x,\hat\Pi(\cX)) \ge \alpha\rho(\cX)},
\eqe
so that $\hat B_\alpha(\cX)$ contains the points in $\hat B(\cX)$ that are far from $\hat\Pi(\cX)$, and in particular are far from $\Pi(\hcXmin) \subset \hat\Pi(\cX)$. See Figure \ref{fig:vol_simp}(b) for an example.
Finally, recall the definition of $\phi(\cX,\cY)$ \eqref{eq:phi_X_Y}, as
 the normalized distance between $c(\hcXmin)$ and the nearest face of $\cY$ that is not $\hcXmin$ (assuming that $\cX\cap\cY = \hcXmin$). The following then holds.

\begin{lem}\label{lem:dist_boundary}
Let $\cX$ be a critical $k$-face, and let $x_* \in \hat B_\alpha(\cX)$. Denote $\cX^* = \hcXmin\cup\set{x_*}$, and
let $\hat\cX_i^*$\ ,\ $i=1\ldots,k$, be all the $(k-1)$-faces of $\cX^*$ that contain $x_*$. 

Suppose that $\phi(\cX, \cX) \ge \beta$, for some $\beta>0$. Then there exists $\const>0$ such that for all $n$
\[
    \phi(\cX, \cX^*) \ge 
    \const\alpha\beta.
\]
\end{lem}
 
 We will now present the proofs for lemmas in this section.
 
\begin{proof}[Proof of Lemma \ref{lem:vol_a_b}]
Define
\[
    f(\bth) := \indf{|c(\bth)|\le \alpha} \ind\{|c(\hat\bth)-c(\bth)| \le \beta\}.
\]
Note that  $c(\bth)$ and $c(\hat\bth)$ are the centers of two spheres, where the former contains the latter, and both are contained in $\S^{k-1}$.
Therefore, $|c(\hat\bth)-c(\bth)|^2 + \rho^2(\hat\bth) = \rho^2(\bth)$, and $|c(\bth)|^2 + \rho^2(\bth) = 1$, and we can write
\[
\splitb
f(\bth) &= \indf{\rho(\bth)\ge\sqrt{1-\alpha^2}} \ind\{\rho(\hat\bth) \ge \sqrt{\rho^2(\bth)-\beta^2}\} \\
&= \indf{\rho(\bth)\ge\sqrt{1-\alpha^2}} \ind\{\rho(\hat\bth)/\rho(\bth) \ge \sqrt{1-(\beta/\rho(\bth))^2}\}.
\splite
\]
We will use the BP formula \ref{lem:bp_sphere}, similarly to the proof of Lemma \ref{lem:V_k_lip}, and write $\bth = \rho\bvphi$, where $\rho = \rho(\bth)$, and $\bvphi\in (\S^{k-2})^k$. Denoting $\hat\bvphi = (\varphi_1,\ldots,\varphi_{k-1})$, notice that $\rho(\hat\bvphi) = \rho(\hat\bth)/ \rho(\bth)$. Therefore, we can write
\[
\splitb
f(\rho\bvphi) &= \indf{\rho\ge \sqrt{1-\alpha^2}} \indf{\rho(\hat\bvphi) \ge \sqrt{1-(\beta/\rho)^2}}\\
&\le \indf{\rho\ge \sqrt{1-\alpha^2}} \indf{\rho(\hat\bvphi) \ge \sqrt{1-\beta^2/(1-\alpha^2)}}.
\splite
\]
Using the BP formula, similarly to the proof of Lemma \ref{lem:V_k_lip},  we have
\[
\splitb
V_k(\alpha,\beta) &= 
\int_{(\S^{k-1})^k}  f(\bth)d\bth \\
&\le \Abp^\circ \int_{\sqrt{1-\alpha^2}}^1 \rho^{k^2-2k}(1-\rho^2)^{-1/2}d\rho
\int_{(\S^{k-2})^k}
\indf{\rho(\hat\bvphi) \ge \sqrt{1-\beta^2/(1-\alpha^2)}} \vsimp(\bvphi)d\bvphi \\
&\le \const V_k(\alpha)V_{k-1}(\beta/\sqrt{1-\alpha^2}).
\splite
\]
where we used the fact that $\vsimp$ is bounded from above, and the definition of $V_{k}(\alpha)$ . 
Since we assume $\alpha < 1$ and $\beta/\sqrt{1-\alpha^2}< 1$, we can apply Lemma \ref{lem:V_k_lip}, concluding the proof.

\end{proof}

\begin{proof}[Proof of Lemma \ref{lem:v_crit_a_b_c}]
To ease notation throughout the proof we will use the following notation
\[
c_i = c(\hbth_i), \quad c_{1,2} = c(\hbth_{1,2}),\quad \Pi_i = \Pi(\hbth_i),\quad \Pi_{1,2} = \Pi(\hbth_{1,2}).
\]
Fix $\hat\bth_1$ such that $|c_1|\le \alpha$ and $\abs{c_{1,2}-c_1} \ge \gamma$. We want to determine the region of all possible values for the missing coordinate ($\theta_1$), that will  satisfy $\hcrit(\bth)\indf{|c_2| \le \beta}=1$. 

Notice that $\S^{k-1} \subset \R^k$, $\Pi_i \subset \R^k$ is $(k-1)$-dimensional affine plane, and $\Pi_{1,2}\subset\R^k$ is  $(k-2)$-dimensional.
In addition, we denote by $\Pi_0$ the $(k-1)$-dimensional plane containing $\hat\bth_{1,2}$ and the origin. Notice also that the integrand in \eqref{eq:v_crit} is rotation and reflection invariant. Therefore, without loss of generality, we  assume that 
\eqb\label{eq:pi_0}
\Pi_0 = \set{(x^{(1)},\ldots, x^{(k)})\in \R^k : x^{(k)} = 0},
\eqe
see Figure \ref{fig:c12}(a).
 Without loss of generality, we will also assume that $\theta_2$ is below $\Pi_0$, i.e.~that $\theta_2^{(k)} < 0$.  In this case, in order to satisfy $\hcrit(\bth) = 1$, we must have that $\theta_1$ is above $\Pi_0$, otherwise all the points of $\bth$ will be contained in the lower hemisphere, implying that $\hcrit=0$ (see Remark \ref{rem:half_sphere}). With this in mind, our next goal is to provide an upper  bound for $\theta_1^{(k)}$, i.e.~such that the integrand in \eqref{eq:v_crit} is nonzero.

\begin{figure}[ht]
\centering
\begin{subfigure}{0.45\textwidth}
\includegraphics[scale=0.28]{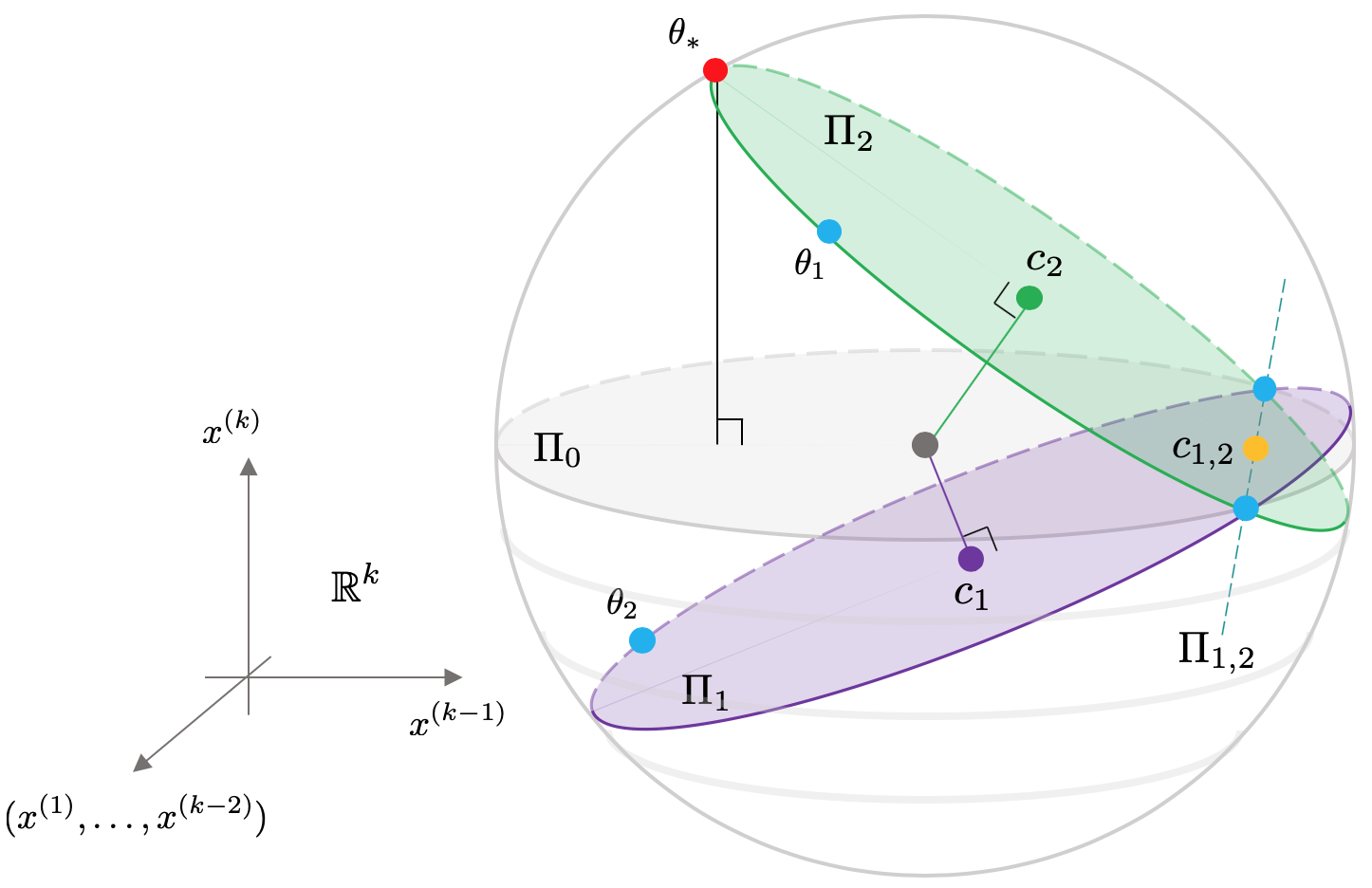}
\caption{}
\end{subfigure}
\begin{subfigure}{0.45\textwidth}
\includegraphics[scale=0.28]{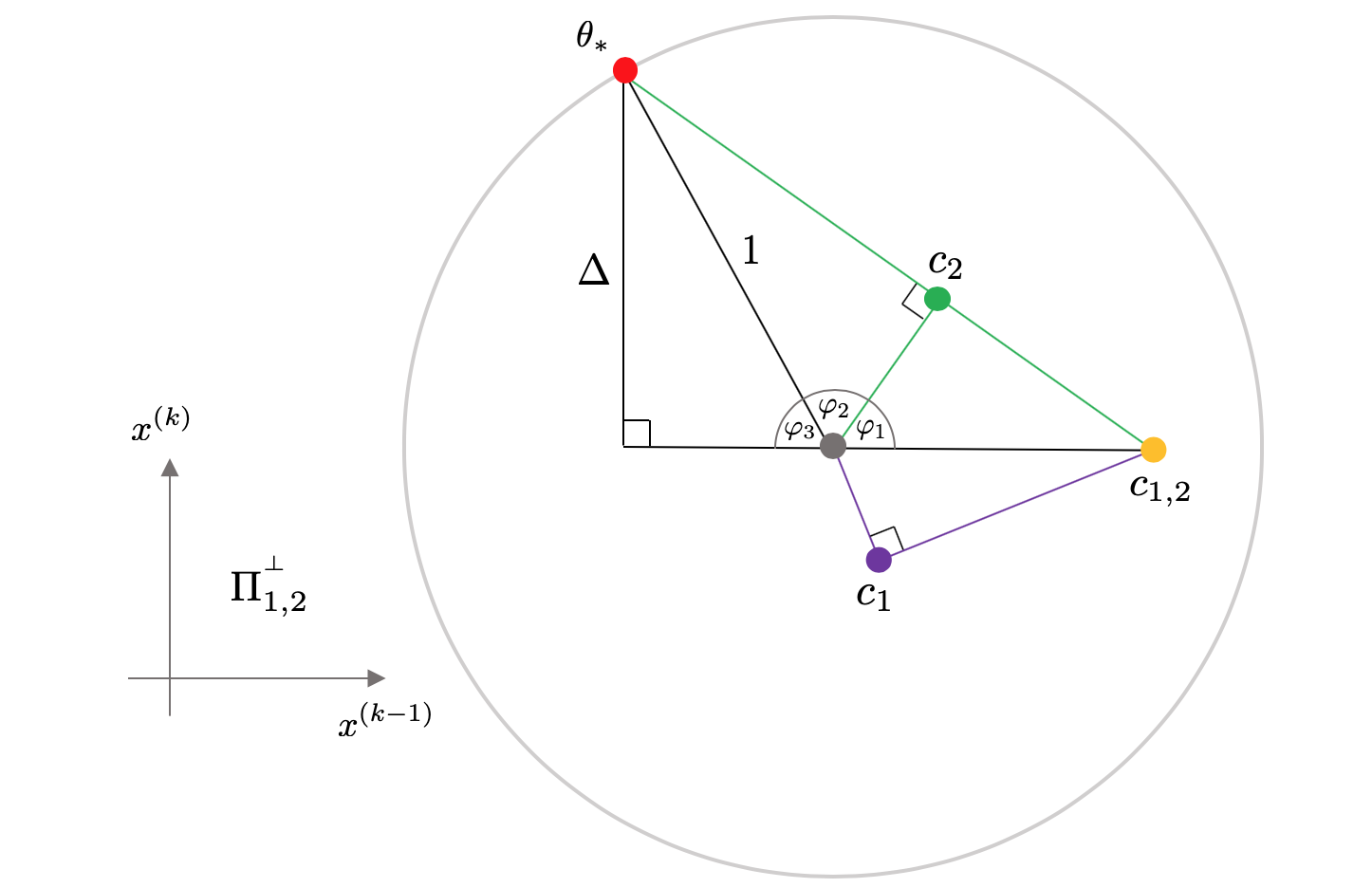}
\caption{}
\end{subfigure}
\caption{\label{fig:c12} Sketches for the point configurations examined in Lemma \ref{lem:vol_a_b}. (a) A sketch for $\S^{k-1}\subset\R^k$ and the corresponding coordinates of $\bth$, as well as the affine spaces $\Pi_0,\Pi_1,\Pi_2$, and $\Pi_{1,2}$. Notice, that we assume $\Pi_0$ is the horizontal plane given in \eqref{eq:pi_0}.
(b) The corresponding 2-dimensional projection on $\Pi_{1,2}$.}
\label{fig:my_label}
\end{figure}
In the following  assume that $c_2$ is given (as well as $c_1,c_{1,2}$ since we fixed $\hat\bth_1$).
Denote by $L_i$ the line going through $c_i$ and $c_{1,2}$ ($i=1,2$), and $L_0$ be the line going through $c_{1,2}$ and the origin. Also denote by $S_i$ the $(k-2)$-sphere containing $\hat\bth_i$ (centered at $c_i$), and by $S_{1,2}$ the $(k-3)$-sphere containing $\hat\bth_{1,2}$ (centered at $c_{1,2}$).
Note that $S_{1,2} =  S_1 \cap S_2 \subset \S^{k-1}$. Therefore, we have that $L_i \perp \Pi_{1,2}$ for $i=0,1,2$. In other words,  the four points $0, c_1, c_2, c_{1,2}$ all lie in a linear space orthogonal  to $\Pi_{1,2}$ and containing $c_{1,2}$, denoted $\Pi^{^\perp}_{1,2}$.
Since $\Pi_{1,2}$ is $(k-2)$-dimensional, then $\Pi^{^\perp}_{1,2}$ is 2-dimensional. Notice that from \eqref{eq:pi_0} we have that the vector $(0,\ldots,0, 1)$ is orthogonal to $\Pi_{1,2}\subset\Pi_0$. Thus, without loss of generality (due to the rotation invariance), we will assume that
\eqb\label{eq:Pi_12_perp}
    \Pi^{^\perp}_{1,2} = \set{(x^{(1)},\ldots, x^{(k)})\in \R^k : x^{(1)} = \cdots = x^{(k-2)} = 0}.
\eqe

Also notice that since  $\hbth_{1,2}$ lies in $\Pi_0$, we also have $c_{1,2} \in \Pi_0$, and thus $c^{(k)}_{1,2} = 0$. Further, since we assumed $\theta_1^{(k)} >0$ and $\theta_2^{(k)}<0$, we have  $c^{(k)}_{1} < 0$ , and $c^{(k)}_{2} > 0$. This setup yields the 2-dimensional picture in Figure  \ref{fig:c12}(b).
Let $\theta_*$ be the point in $\Pi^{^\perp}_{1,2}$ where the ray emanating from $c_{1,2}$ and passing through $c_2$ intersects with $\S^{k-1}$. We argue that $\theta_1^{(k)} \le \theta_*^{(k)}$.
Indeed, if $\theta \in S_2\subset \S^{k-1}$ then
\[
    (1)\ \ \sum_{i=1}^{k} (\theta^{(i)})^2 = 1\ , \qquad\quad
    (2)\ \ \sum_{i=1}^{k-2}(\theta^{(i)})^2 + (\theta^{(k-1)}-c_2^{(k-1)})^2 + (\theta^{(k)}-c_2^{(k)})^2 =   1 - |c_2|^2,
\]
where in $(2)$ we used the facts that $c_2 \in \Pi^{^\perp}_{1,2}$, assumption  \eqref{eq:Pi_12_perp}, and $\rho^2(\hat\bth_2) = 1-|c_2|^2$. It can be shown that the value of $\theta^{(k)}$ is maximized when $\theta^{(1)} = \cdots = \theta^{(k-2)} = 0$. In other words, the largest value of $\theta^{(k)}$ is when $\theta\in S(\hat\bth_2) \cap \Pi^{^\perp}_{1,2}$, and then
\[
\splitb
    (\theta^{(k-1)})^2 + (\theta^{(k)})^2 &= 1  \\
     (\theta^{(k-1)}-c_2^{(k-1)})^2 + (\theta^{(k)}-c_2^{(k)})^2 &=   1 - |c_2|^2.
\splite
\]
This implies that the point $(\theta^{(k-1)}, \theta^{(k)}) \in \Pi_{1,2}^{^\perp}$ lies in the intersection of two circles, i.e.~it can only be one of two points.
By definition, $\theta_*$ is one of these points. 
Since $c_{1,2}^{(k)} = 0$ and $c_2^{(k)} > 0$, we have that $\theta_*^{(k)} > 0$, and  the second point of intersection has  a negative value in the $k$-th coordinate. 
The conclusion is that for all points $\theta\in S_2$ we have $\theta^{(k)} \le \theta_*^{(k)}$.
Thus, since $\theta_1\in S_2$, in order to bound $\theta_1^{(k)}$ we  can find a bound on $\theta_*^{(k)}$. 

Denote $\Delta = \theta_*^{(k)}$. Considering the angles $\varphi_1,\varphi_2,\varphi_3$ marked in Figure  \ref{fig:c12}(b), we have
\[
    \varphi_1 = \cos^{-1}\param{\frac{|c_2|}{\sqrt{|c_1|^2 + |c_{1,2}-c_1|^2}}} \ge \cos^{-1} \param{\frac{\beta}{\gamma}} = \frac{\pi}{2} - \frac{\beta}{\gamma} + o(\beta/\gamma),
\]
and
\[
    \varphi_2 = \cos^{-1}(\abs{c_2}) \ge \cos^{-1}(\beta) = \frac{\pi}{2} - \beta + o(\beta),
\]
where we used the fact that $|c_2|\le \beta$ and $|c_{1,2}-c_1| \ge \gamma$.
Since we assume that $\gamma\to 0$ we have $\beta = o(\beta/\gamma)$ and therefore,
\[
    \varphi_3 = \pi-\varphi_1-\varphi_2 \le \beta/\gamma + o(\beta/\gamma).
\]
Since $\Delta = \sin(\varphi_3)$ we have 
\[
    \Delta \le \sin(\beta/\gamma + o(\beta/\gamma)) = \beta/\gamma + o(\beta/\gamma).
\]

To conclude, given $\hat\bth_1$ and $c_2$, we showed that $\theta_1$ lies inside a region, which up to rotation is of the form
\[  
D_\Delta := \set{(x^{(1)},\ldots, x^{(k)}) \in \S^{k-1} : 0\le x^{(k)} \le \Delta},
\]
and therefore
\[
    \int_{\S^{k-1}}\hcrit(\bth) \indf{|c(\hat\bth_2)|\le \beta} d\theta_1 \le \vol(D_\Delta) \le \const \Delta \le \const\beta/\gamma.
\]
Putting this back into \eqref{eq:v_crit}, we have 
\[
\splitb
\Vcrit(\alpha,\beta,\gamma) &=  \int_{(\S^{k-1})^{k+1}}  \hcrit(\bth) \indf{|c(\hbth_1)| \le \alpha,\ |c(\hbth_2)| \le \beta,\  |c(\hbth_{1,2})-c(\hbth_1)| \ge \gamma}d\bth\\
&\le \const \frac{\beta}{\gamma}\int_{(\S^{k-1})^k} \indf{|c(\hat\bth_1) \le \alpha}d\hat\bth_1 \\
&= \const\frac{\beta}{\gamma}V_k(\alpha)\\
&\le \const \frac{\alpha\beta}{\gamma},
\splite
\]
where we used Lemma \ref{lem:V_k_lip}.
This completes the proof.

\end{proof}

\begin{proof}[Proof of Lemma \ref{lem:dist_boundary}]
Set $c = c(\cX), \cmin = c(\hcXmin), \rho = \rho(\cX)$, and $\hrho = \hrho(\cX)$.
Without loss of generality, suppose that $\hat\cX_{\min} =  \set{x_1,\ldots, x_k}$, so that $\cX^* = \set{x_1,\ldots, x_k, x_*}$, and $\hat\cX^*_i = \cX^* \bs \set{x_i}$. Consider the $k$-simplex $\cY_i = \hat\cX^*_i \cup \set{c_{\min}}$, and define $\hat\cY_i = \cY_i \bs \set{x_*}$.
In other words, we have
\[
\splitb
\hat\cX^*_i &= \set{x_1\ldots, x_{i-1}, x_{i+1},\ldots,  x_k, x_*},\\
\hat\cY_i &= \set{x_1\ldots, x_{i-1}, x_{i+1},\ldots,  x_k, c_{\min}},\\
\cY_i &= \set{x_1\ldots, x_{i-1}, x_{i+1}, \ldots, x_k, x_*, c_{\min}},\\
\splite
\]
see Figure \ref{fig:vol_simp}(b). Denoting $\Pi_i^* := \Pi(\hat\cX_i^*)$, our goal is to show that $d(\cmin , \Pi_i^*) \ge \const\alpha\beta\rho$ for all $1\le i \le k$.
Defining $\Pi_{\min} := \Pi(\hcXmin)$, notice that $\Pi(\hat\cY_i) = \Pi_{\min}$.  Using Lemma \ref{lem:simp_vol}, we can evaluate the volume of $\sigma(\cY_i)$ in two ways,
\[
    \vsimp(\cY_i)  = \frac{1}{k} \vsimp(\hat\cX^*_i) d(c_{\min},\Pi^*_i ) = \frac{1}{k} \vsimp(\hat\cY_i) d(x_*,\Pi_{\min} ).
\]
Therefore,
\[
     d(c_{\min},\Pi^*_i) = d(x_*,\Pi_{\min})\frac{\vsimp(\hat\cY_i)}{\vsimp(\hat\cX^*_i)}.
\]
Next, set $\hcX_{i,j} = \cX\bs \set{x_i,x_j}$, $c_{i,j} := c(\hat\cX_{i,j})$, and $\Pi_i = \Pi(\hat\cX_{i,k+1})$. Using Lemma \ref{lem:simp_vol} again yields
\[
\splitb
    \vsimp(\hat\cY_i) &= \frac{1}{k-1} \vsimp(\hat\cX_{i,k+1}) d( c_{\min}, \Pi_i),\\
    \vsimp(\hat\cX^*_i) &= \frac{1}{k-1}\vsimp(\hat \cX_{i,k+1}) d(x_*,\Pi_i),
\splite
\]
and therefore,
\eqb \label{eq:dist_vol}
d(\cmin, \Pi^*_i) = \frac{d(x_*, \Pi_{\min})
d(c_{\min}, \Pi_{i})}{ d(x_*,\Pi_{i})}.
\eqe

If $\phi(\cX,\cX)\ge \beta$, then necessarily,
$d(c_{\min}, \Pi_{i}) \ge \beta\rho$.
In addition, since $c_{i,k+1}$ and $x_*$ are both inside the ball $\hat B(\cX)$, we have
\[
    d(x_*, \Pi_{i}) \le |x_* - c_{i,k+1} | \le 2\hrho \le 3\rho. 
\]
Finally, since $x_* \in \hat B_\alpha(\cX)$,  and $\Pi_{\min} \subset \hat \Pi(\cX)$, we have $d(x_*, \Pi_{\min})  \ge \alpha \rho(\cX)$.
Putting it all  into \eqref{eq:dist_vol} we have
\[
    d( c_{\min}, \Pi^*_i) \ge \frac{\alpha\beta\rho}{3},
\]
completing the proof.
\end{proof}

\subsection{Pairing critical faces}

With the extended topological and geometric statements provided in the previous sections, we are now 
ready to prove Lemma \ref{lem:pn_pairs}, i.e.~that in the critical window we have
\[
\limninf \prob{ \Cp_{k,r} = \Cn_{k+1,r} = \Cnp_{k+1,r}} \to 1.
\]

\begin{proof}[Proof of Lemma \ref{lem:pn_pairs}]
We will split the proof into two parts.

\underline{Part I - proving that $\Cn_{k+1,r} = \Cnp_{k+1,r}$ :}

Define
\[
	\Cnnp_{k,r} = \Cn_{k,r} - \Cnp_{k,r},
\]
so that $\Cnnp_{k,r}$ counts the negative $k$-faces for which $\hcXmin$ is \emph{not} critical.
We will show that,
\[
\mean{\Cnnp_{k+1,r}} \to 0,
\]
and use Markov's inequality.
To do so, we will bound $\Cnnp_{k,r}$ by 
\eqb\label{eq:match_ineq}.
	\Cnnp_{k,r} \le \sum_{i=1}^4 F^{(i)}_{k,r},
\eqe
where:
\[
	F^{(1)}_{k,r} := \#\text{negative critical $k$-faces $\cX$ with $\phi(\cX) > \eps_1$},
\]
\[
	F^{(2)}_{k,r} := \#\text{negative critical $k$-faces $\cX$ with $\phi(\cX) \le \eps_1$ and $\phi(\cX,\cX)\le \eps_2$},
\]
\[
	F^{(3)}_{k,r} := \#\text{negative critical $k$-faces $\cX$ with $\phi(\cX) \le \eps_1$ and $(\hat B(\cX)\bs \hat B_{\eps_3}(\cX)) \cap \cP_n \ne \emptyset$} ,
\]
\[
	F^{(4)}_{k,r} := \#\text{negative critical $k$-faces $\cX$ with $\phi(\cX) \le \eps_1$ and $\phi(\cX, \cX) > \eps_2$ and $ \hat B_{\eps_3}(\cX) \cap \cP_n \ne \emptyset$} ,
\]
where we implicitly assume everywhere that $\rho(\cX) \in (r,\rmax]$.
Notice that the faces accounted for by the all the $F_{k,r}^{(i)}$'s together, include all configurations where
$\hat B(\cX) \cap \cP_n = \ems$, implying that $I(\cX)\cap \cP_n = \ems$, which in turn implies that
 $\hcXmin$ is not critical, justifying \eqref{eq:match_ineq}.
Finally, set
\eqb
\eps_1 = \frac{2}{\Dn_{k,2}}\frac{\log\Lambda}{\Lambda}\quad
\eps_2 = \frac{e^{-4w(n)}}{(\log \Lambda)^2}\quad
\eps_3 = \eps_1^{2/3},
\eqe
In the following lemmas, we will prove that with these choices, and if $\Lambda = \thres - w(n)$ and $w(n) = o(\logg n)$,  we have
\[
\limninf\mean{F_{k+1,r}^{(i)}} = 0,
\]
for all $1\le i \le 4$. Since all quantities are decreasing in $r$, this will be true for any larger $\Lambda$ as well. That will complete the proof of part I.

\underline{Part II - proving that $\Cp_{k,r} = \Cn_{k+1,r}$ : }

At this point, we have shown that  when $\Lambda =\thres + o(\logg n)$, all the remaining negative $(k+1)$-faces are matched with a positive $k$-face.  
We want to show that this matching is a bijection between all the negative $(k+1)$-faces and negative $k$-faces appearing at $\rho \in (r,\rmax]$.

First, we show that matched faces always ``stick together'', in the sense that for any given choice of $r$, there is no such matching for which $\rho(\hcXmin)<r$ and $\rho(\cX)>r$. 
To this end, we define
\[
	F_{k,r}^{(5)} := \#\text{negative $k$-faces $\cX$, with $\phi(\cX)\le \eps_1$,  such that $\rho(\cX) > r$ and $\rho(\hcXmin) < r$}.
\]
We will prove next that $\mean{F_{k+1,r}^{(5)}}\to 0$, so that all matched faces appear after $r$. 

Recall the definition of $I(\cX)$ \eqref{eq:hat_B}, and the fact that 
 \whp\ $F_{k+1,r}^{(i)} = 0$ for $i=1,\ldots,4$. Then \whp\ $I(\cX)\cap \cP_n = \emptyset$ for all the remaining negative faces in $(r,\rmax]$.
This indicates not only that $\hcXmin$ is critical, but also that it is isolated in the entire interval $[\rho(\hcXmin), \rho(\cX))$. This, in turn, implies that the matching between negative and positive faces is injective. Since both the positive and negative faces appear in $(r,\rmax]$, we can conclude that  $\Cn_{k+1,r} = \Cnp_{k+1,r} \le \Cp_{k,r}$.
From Lemma \ref{lem:perc}, we have that $\Cn_{k+1,r} \ge \Cp_{k,r}$, and therefore $\Cn_{k+1,r} = \Cp_{k,r}$ and the matching is indeed a bijection.  This completes the proof of part II.

\end{proof}

The rest of this section will thus be dedicated to prove the limits of $\mean{F_{k+1,r}^{(i)}}$ ($i=1,\ldots, 5$).

\begin{rem}
In the following, we will use expectation calculations similar to the ones we presented in Sections \ref{sec:pt_crit} and \ref{sec:pn_faces}. Therefore, we will highlight the main points and skip some of the steps that are similar to what we have already done.
\end{rem}

\begin{lem}\label{lem:F_1}
Let $k\ge 1$. If $\Lambda = \thres - w(n)$, then
\[
	\limninf\mean{F_{k+1,r}^{(1)}} = 0.
\]
\end{lem}

\begin{proof}
Using Lemma \ref{lem:Fk_phi}, and taking $a = \eps_1 = \frac{2}{\Dn_{k,2}}\frac{\log \Lambda}{\Lambda}$, $b = 1$, we have
\[
\splitb
		\mean{ F_{k,r}^{(1)}} = \mean{ F_{k,r,\eps_1,1}} &\le \Dn_{k,1} (1-\eps_1) n\Lambda^{k-1} e^{-\Lambda(1+ \Dn_{k,2} \eps_1)}\\
		&\le \Dn_{k,1} n \Lambda^{k-3} e^{-\Lambda }
\splite
\]
Therefore, if $\Lambda = \thres -w(n)$ then
\[
\mean{F_{k+1,r}^{(1)}} \le \Dn_k \Lambda^{-1} \param{\frac{\Lambda}{\log n}}^{k-1}e^{w(n)} \to 0,
\]
since $\Lambda \approx \log n$ and $w(n) = o(\logg n)$.
This completes the proof.

\end{proof}

\begin{lem}\label{lem:F_2}
Let $k\ge 1$. If $\Lambda = \thres - w(n)$, then
\[
	\limninf\mean{F_{k+1,r}^{(2)}} = 0.
\]
\end{lem}

\begin{proof}
Recall the definitions of $\theta(\cX), \hat\theta_i(\cX)$, and note that requiring $\phi(\cX) \le \eps_1$ and $\phi(\cX,\cX) \le \eps_2$ implies that there exist $i,j$ such that (a) $|c(\hat\theta_i(\cX))| \le \eps_1$, and (b)  $|c(\hat\theta_i(\cX)) - c(\hat\theta_j(\cX))| \le \eps_2$, implying that $|c(\hat\theta_j(\cX))| \le \eps_1+\eps_2$. 
Defining 
\[
	F_{k,r,i,j}^{(2)} = \# \text{negative critical $k$-faces $\cX$ with $|c(\hat\theta_i(\cX))|\le \eps_1$ and $|c(\hat\theta_j(\cX))|\le \eps_1+\eps_2$.}
\]
we therefore have
\[
	F_{k,r}^{(2)} \le \sum_{i\ne j} F_{k,r,i,j}^{(2)},
\]
and thus
 \[
	\mean{F_{k,r}^{(2)}} \le k(k+1) \mean{F_{k,r,1,2}^{(2)}}.
\]
We will further bound $F_{k,r,1,2}^{(2)}$ by two terms -
\[
	F_{k,r,1,2}^{(2,1)} := \#\text{negative critical $k$-faces with $|c(\hat\theta_1(\cX))|\le \eps_1$ and $|c(\hat\theta_{1,2}(\cX))-c(\hat\theta_1(\cX))| \le \eps_{1,2}$.}
\]
\[
\splitb
	F_{k,r,1,2}^{(2,2)} := \#&\text{negative critical $k$-faces with $|c(\hat\theta_1(\cX))|\le \eps_1$, $|c(\hat\theta_{1,2}(\cX))-c(\hat\theta_1(\cX))| > \eps_{1,2}$,} \\
	&\text{and $|c(\hat\theta_2(\cX))|\le \eps_1+\eps_2$},
\splite
\]
so that $F_{k,r,1,2}^{(2)} \le F_{k,r,1,2}^{(2,1)}+F_{k,r,1,2}^{(2,2)}$.

Repeating similar steps as in the proof of Lemma \ref{lem:Fk_phi}, we have 
\[
\mean{F_{k,r,1,2}^{(2,1)}} \le \const n\Lambda^{k-1}e^{-\Lambda} \int_{(\S^{k-1})^{k+1}} \indf{|c(\hbth_1) \le \eps_1,\ |c(\hbth_{1,2}) - c(\hbth_1)| \le \eps_{1,2}} d\bth
\]
Using Lemma \ref{lem:vol_a_b}, we have
\[
\mean{F_{k,r,1,2}^{(2,1)}} \le \const\frac{\eps_1\eps_{1,2}}{\sqrt{1-\eps_1^2}}
 n\Lambda^{k-1}e^{-\Lambda}.  \]
If we take
\[
	\eps_{1,2} = \frac{e^{-2w(n)}}{\log \Lambda},
\]
 then
\[
	\mean{F_{k,r,1,2}^{(2,1)}} \le \const e^{-2w(n)} n\Lambda^{k-2}e^{-\Lambda} .
\]
When $\Lambda = \thres-w(n)$ we therefore have
\[
	\mean{F_{k+1,r,1,2}^{(2,1)}} \le \const  e^{-w(n)} \to 0. 
\]

Similarly, evaluating $F_{k,r,1,2}^{(2,2)}$ we have
\[
\splitb
\mean{F_{k,r,1,2}^{(2,2)}} &\le \const n\Lambda^{k-1}e^{-\Lambda} \\&\times\int_{(\S^{k-1})^{k+1}} \indf{|c(\hbth_1)| \le \eps_1,|c(\hbth_2)| \le \eps_1+\eps_2,\ |c(\hbth_{1,2}) - c(\hbth_1)| > \eps_{1,2}} d\bth.
\splite
\]
In this case, notice that $\eps_1+\eps_2\approx \eps_2 = \eps_{1,2}^2$, therefore $(\eps_1+\eps_2)/ \eps_{1,2} \to 0$.
Thus, we can use Lemma \ref{lem:v_crit_a_b_c} and have
\[
\mean{F_{k,r,1,2}^{(2,2)}} \le \const \frac{\eps_1\eps_2}{\eps_{1,2}} n\Lambda^{k-1}e^{-\Lambda}  = 
\const e^{-2w(n)}n\Lambda^{k-2}e^{-\Lambda}.
\]
Taking $\Lambda = \thres - w(n)$ we have
\[
\mean{F_{k+1,r,1,2}^{(2,2)}} \le \const  e^{-w(n)} \to 0 .
\]
That  completes the proof.

\end{proof}

\begin{lem}\label{lem:F_3}
Let $k\ge 1$. If $\Lambda = \thres - w(n)$, then
\[
	\limninf\mean{F_{k+1,r}^{(3)}} = 0.
\]
\end{lem}

\begin{proof}
Given a $k$-face $\cX$, denote $\rho = \rho(\cX),\hrho = \hrho(\cX), \rho_{\min} = \rho(\hcXmin), B = B(\cX),\ B_{\min} = B(\hcXmin)$, and $\hat B = \hat B(\cX),\ \hat B_\alpha = \hat B_{\alpha}(\cX),\  \hat B^c_\alpha = \hat B \bs \hat B_{\alpha}$ (see \eqref{eq:B_a}). To bound $F^{(3)}_{k, r}$ we will bound the volume of $\hat B^c_{\eps_3}$.

First, recall that $\hat\Pi(\cX)$ is a
$(d-1)$-plane that contains $\cmin$.
Using this plane, we split $\R^d$ into two half planes $\R^{d,+}$ and $\R^{d,-}$ (``above'' and ``below" $\hat\Pi(\cX)$), so that $c \in \R^{d,-}$. See Figure \ref{fig:vol_simp}(b).
Correspondingly, we can split $\hat B_{\alpha}$ and $\hat B_{\alpha}^c$ into two parts - $\hat B_{\alpha}^{\pm} = \hat B_{\alpha}\cap \R^{d,\pm}$ and $\hat B_{\alpha}^{c,\pm} = \hat B_{\alpha}^c\cap \R^{d,\pm}$, respectively.
 Since $c\in \R^{d,-}$, and $\hat B = B_{\hat \rho}(\cmin)\bs B$, we have that $\vol(\hat B_{\alpha}^{c,+}) \ge  \vol(\hat B_{\alpha}^{c,-})$, and therefore we will bound only $\vol(\hat B_{\alpha}^{c,+})$.

Next, notice that $\hat B_\alpha^+$ is a difference between two spherical caps (see Figure \ref{fig:vol_simp}(b)), so that
\[
\vol(\hat B_{\alpha}^+) = \hrho^d \vcap \param{\frac{\alpha\rho}{\hrho}} - \rho^d \vcap(\phi+\alpha).
\]
Therefore,
\eqb\label{eq:vol_B_a}
\splitb
\vol(\hat B^{c,+}_{\alpha}) &= \vol(\hat B^+_0)-\vol(\hat B^+_\alpha)\\
	&=  \hrho^d\param{\vcap(0) - \vcap\param{\frac{\alpha\rho}{\hrho}}} - \rho^d \param{\vcap(\phi) - \vcap(\phi+\alpha)} \\
	&\le
	 \hrho^d\param{\vcap(0) - \vcap\param{\alpha}} - \rho^d \param{\vcap(\phi) - \vcap(\phi+\alpha)} 
\splite
\eqe
Notice that $\vcap(\phi) - \vcap(\phi+\alpha)$ is decreasing in $\phi$. Since $\phi \le \eps_1$, we have 
\[
\vcap(\phi) - \vcap(\phi+\alpha) \ge \vcap(\eps_1) - \vcap(\eps_1+\alpha).
\]
Also,
\[
	\hrho^d = \rho^d(1+\phi)^d \le \rho(1+\eps_1)^d = \rho^d(1+d\eps_1 + o(\eps_1)) 
\]
Therefore,
\[
\splitb
\vol(\hat B^c_{\alpha}) &\le \rho^d(\vcap(0) - \vcap\param{\alpha} + \vcap(\eps_1+\alpha)-\vcap(\eps_1) ) + \const \eps_1\rho^d(\vcap(0) - \vcap\param{\alpha}) 
\splite
\]

In addition from, \eqref{eq:vol_taylor_2} we have
\[
\vcap(\Delta) = C_0 - C_1\Delta + C_3\Delta^3 + o(\Delta^4), 
\]
where $C_0 = \frac{\omega_d}{2},\ C_1 = \omega_{d-1},\ C_3 = \frac{(d-1)\omega_{d-1}}{6}$. We will use this to estimate each of the terms in \eqref{eq:vol_B_a}, for $\alpha = \eps_3$. 

We know that $\vcap(0) = C_0$. Next, recall that $ \eps_1^2 = \eps_3^3$, then
\[
\vcap(\eps_3) = C_0 - C_1\eps_3 + C_3\eps_3^3 + o(\eps_3^4).
\]
Next,
\[
\splitb
\vcap(\eps_3+\eps_1) &= C_0 - C_1(\eps_3+\eps_1) + C_3(\eps_3+\eps_1)^3 + o(\eps_3^4) \\
&= C_0 - C_1(\eps_3+\eps_1) + C_3(\eps_3^3 + 3\eps_1\eps_3^2 + 3\eps_1^2\eps_3 + \eps_1^3) + o(\eps_3^4).\\
&= C_0 - C_1\eps_3 - C_1\eps_3^{3/2} + C_3\eps_3^3 + 3C_3 \eps_3^{7/2}+ 3C_3\eps_3^4 + o(\eps_3^4).
\splite
\]
Finally,
\[
\vcap(\eps_1) = C_0 - C_1\eps_1 + C_3\eps_1^3 + o(\eps_1^4) = C_0-C_1\eps_3^{3/2} + o(\eps_3^4).
\]
Putting it all back into \eqref{eq:vol_B_a}, we have
\[
\frac{\vol(\hat B^c_{\eps_3})}{\rho^d} \le 
\frac{2\vol(\hat B^{c,+}_{\eps_3})}{\rho^d} \le
\const \eps_3^{7/2}+ \const\eps_3^4  +  
\const \eps_3^{5/2} + o(\eps_3^4)
\le \const \eps_3^{5/2}.
\]
Now, recall that $F^{(3)}_{k,r}$ counts the negative $k$-faces with $\phi\le \eps_1$ and such that there $\hat B^c_{\eps_3} \cap \cP_n \ne \emptyset$. Given $\cX=\bx$, the probability that $\hat B^c_{\eps_3} \cap \cP_n \ne \emptyset$ is
\eqb\label{eq:prob_points}
	1-e^{-n\vol(\hat B_{\eps_3}^c(\bx))} \le  n\vol(\hat B_{\eps_3}^c(\bx)) \le\const \eps_3^{5/2}n\rho^d(\bx).
\eqe
Repeating the steps as in the proof of Lemma \ref{lem:Fk_phi}, and using \eqref{eq:prob_points}, we have
\eqb\label{eq:bound_Fk3}
\splitb
	\mean{F^{(3)}_{k,r}} &\le \const \eps_1 n^{k+1}  \int_{r}^{\rmax} \rho^{dk-1} e^{-n\omega_d \rho^d}\param{1-e^{- n\vol(\hat B_{\eps_3}^c)}}d\rho \\
	& \le  \const \eps_1\eps_3^{5/2} n^{k+2} \int_r^{\rmax} \rho^{d(k+1)-1} e^{-n\omega_d\rho^d}d\rho \\
	&\le \const \eps_3^{4} n \Lambda^k e^{-\Lambda}.
\splite
\eqe
Taking $\Lambda = \thres -w(n)$, and recall that $\eps_3 = \const\param{\frac{\log \Lambda}{\Lambda}}^{2/3}$, then  
\[
\mean{F^{(3)}_{k+1,r}} \le \const\frac{(\log \Lambda)^{8/3}}{\Lambda^{2/3}}e^{w(n)} \to 0.
\]
This completes the proof.

\end{proof}

\begin{lem}\label{lem:F_4}
Let $k\ge 1$. If $\Lambda = \thres -w(n)$, then
\[
	\limninf\mean{F_{k+1,r}^{(4)}} = 0.
\]
\end{lem}

\begin{proof}
Denote $c = c(\cX), \cmin = c(\hcXmin), \rho = \rho(\cX),\hrho = \hrho(\cX), \phi = \phi(\cX), \hphi = \phi(\cX,\cX), B = B(\cX), \hat B_{\eps_3} = \hat B_{\eps_3}(\cX)$, and $\hat \cP = \hat B_{\eps_3} \cap \cP_n $.
Recall that $F_{k,r}^{(4)}$ counts  the negative critical $k$-faces $\cX$ with $\phi \le \eps_1$, $\hphi > \eps_2$ and $ \hat \cP \ne \emptyset$. Our goal here is to apply the topological statement from Lemma \ref{lem:crit_pos_2} to show that such configurations cannot produce negative faces.

Given $\hat \cP \ne\emptyset$,
there exists $x_*\in \hat B_{\eps_3}$ such that $\cX^* = \hcXmin \cup \set{x_*} \in \cC^k(\cP_n)$.
Since $\hphi > \eps_2$,  we can apply Lemma \ref{lem:dist_boundary}, to conclude that $\phi(\cX, \cX^*) \ge \const\eps_2\eps_3$, and since $\hphi > \eps_2\gg \eps_2\eps_3$, we have
\[	
\phmin(\cX, \cX^*) \ge \const\eps_2\eps_3.
\]
Notice that $\eps_2 \eps_3 = \frac{e^{-4w(n)}}{(\log \Lambda)^{4/3}\Lambda^{2/3} }\gg \Lambda^{-3/4}$.
Thus, we will set 
 $\eps_0 = \Lambda^{-3/4}$, and aim to 
 show that $\hat A_{\eps_0}(\cX)$ is not covered.

Given $\cX = \bx$, similarly to the proof of Lemma \ref{lem:Fk_phi}, we define
\[
\hat p_{\eps_0}(\bx)= \P_{\emptyset} \param{\hat A_{\eps_0}(\bx) \not\subset B_{(1-{\eps_0}/3)\rho(\bx)}(\cP_n)}.
\]
Following the same steps as in the proof of Lemma  \ref{lem:Fk_phi}, we take an $\eps_0/2$-net $\cS$ for $\hat A_{\eps_0}$. The size of this net is bounded by
\[
	|\cS|\le \const \frac{\vol(\hat A_{\eps_0})}{\vol(B_{\eps_0\rho/2})} \le \frac{\rho^d(1+\eps_1)^d}{\rho^d(\eps_0/2)^d} \le \const {\eps_0}^{-d}.
\]
Then
\[
\hat p_{\eps_0}(\bx)\le 
\sum_{s\in \cS} \P_{\emptyset}\param{B_{\rho_1}(s) \cap \cP_n = \emptyset} = \const\eps_0^{-d} \max_{s\in \cS} e^{-n \vol(B_{\rho_1}(s)\bs B)}.
\]
where $\rho_1 = \rho(1-(5/6)\eps_0)$. Recall that $B$ is a ball of radius $\rho$ centered at $c$. Denote by $s_{\min}$ the point in $\hat A_{\eps_0}$ closest to $c$ (see Figure \ref{fig:vol_simp}(a)). Then $|s_{\min}-c| = \eps_0 - \phi \ge (11/12)\eps_0$ (since $\phi \le \eps_1\ll \eps_0$). Therefore, for every $s\in \cS$ we have
\[
		\vol(B_{\rho_1}(s)\bs B) \ge \vol(B_{\rho_1}(s_{\min})\bs B) \ge \rho^d\vdiff(\eps_0, 10/11)
\]
From Lemma \ref{lem:vdiff} we conclude that there exists $\const>0$ such that $\vol(B_{\rho_1}(s)\bs B) \ge \const \eps_0\rho^d$ for all $s\in \cS$.  Thus, we have
\[
	\hat p_{\eps_0}(\bx) \le \const \eps_0^{-d} e^{-\const \eps_0 n\rho^d}.
\]
Proceeding as in Lemma \ref{lem:Fk_phi} and since $\phi \le \eps_1$, we have
\[
\mean{F_{k,r}^{(4)}} \le \const \eps_0^{-d}\eps_1 n\Lambda^{k-1} e^{-\Lambda(1+ \const\eps_0)}
\]
Recall that  $\eps_0 = \Lambda^{-3/4}$, $\eps_1 = \const \frac{\log\Lambda}{\Lambda}$, and $\Lambda = \thres -w(n)$. Then
\[
\mean{F_{k+1,r}^{(4)}} \le \const \Lambda^{3d/4} \log \Lambda  e^{-\const \Lambda^{1/4} +w(n)} \to 0.
\]
This completes the proof.

\end{proof}

\begin{lem}\label{lem:F_5}
For $k\ge 1$, if $\Lambda= \thres-w(n)$ then
\[	\limninf\mean{F_{k+1,r}^{(5)}} = 0.
\]
\end{lem}

\begin{proof}
Fix $\cX$ and let $\rho = \rho(\cX), \rhmin = \rho(\hcXmin)$, and $\phi = \phi(\cX)$.

Suppose that $\rho \ge r(1+\eps_1^2)$.
Since $\phi \le \eps_1$, we have
\[
\rhmin = \rho\sqrt{1-\phi^2} \ge \rho\sqrt{1-\eps_1^2} \ge r(1+ \eps_1^2)(1-\eps_1^2/2 + o(\eps_1^2)) = r(1 +\eps_1^2/2 + o(\eps_1^2)).
\]
Therefore, for $n$ large enough we will have $\rho_{\min} > r$, so such faces are not accounted for by $F_{k,r}^{(5)}$. 
Therefore, we should only count faces for which $\rho \le r(1+\eps_1^2)$.
Similarly to Lemma \ref{lem:Fk_phi} we can show that the number of negative $k$-faces with $\rho \in [r,r(1+\eps_1^2)]$ and with $\phi\le \eps_1$ is bounded by
\[
\splitb
	\mean{F_{k,r,0,\eps_1}-F_{k,r(1+\eps_1^2),0,\eps_1}}
&\le \const \eps_1 n \param{\Lambda^{k-1}e^{-\Lambda} - (1+\eps_1^2)^{k-1}\Lambda^{k-1}e^{-(1+\eps_1^2)^d\Lambda)}} \\
&=\const \eps_1 n \Lambda^{k-1}e^{-\Lambda}\param{d\eps_1^2\Lambda + o(\eps_1^2\Lambda)},
\splite
\]
where we used the fact that $\eps_1^2\Lambda \to 0$.
Taking $\Lambda = \thres -w(n)$ we then have
\[
	\mean{F_{k+1,r}^{(5)}} \le \const \eps_1^3\Lambda^2 e^{w(n)} \le \const \frac{(\log \Lambda)^3}{\Lambda} e^{w(n)} \to 0.
\]
This completes the proof.

\end{proof}

\subsection{The positive-negative pairing and isolated faces}

At this point, for $1\le k \le d-2$, we have shown that in when $\Lambda =\thres + o(\logg n)$, there is a bijection between  remaining positive $k$-faces and negative $(k+1)$-faces. The statements we proved above provide additional  insight into the structure of homology towards connectivity. 

The  first observation we want make is that the $k$-cycle ``born'' at $\rho(\hcXmin)$ is the one that ``dies" at $\rho(\cX)$.
More precisely, we can show that the $k$-cycle $\partial \sigma(\cX)$ that is introduced at $\rho(\hcXmin)$ (when $\hcXmin$ is added), at no point $\rho \in(\rho(\hcXmin),\rho(\cX))$ becomes homologous to any other $k$-cycle $\gamma$ that was generated earlier (i.e.~$\gamma \in \bZ_k(\cC_{\rho'}(\cP_n))$ for $\rho' < \rho(\hcXmin)$). 

Lemma \ref{lem:F_4} proves that each of the positive $k$-faces $\hcXmin$ enters the complex at radius $\rho(\hcXmin)$ and remains \emph{isolated} through the entire interval $[\rho(\hcXmin), \rho(\cX))$. This means that $\hcXmin$ is not on the boundary of any $(k+1)$-simplex, and therefore $\partial \sigma(\cX)$ cannot be homologous to any other $k$-cycle that does not include $\hcXmin$, and therefore could not have been born earlier. 

The second observation is related to the vanishing of isolated faces, mentioned in the introduction. In most models studied so far for random graphs and complexes, there is a strong link between the vanishing of isolated vertices/faces and connectivity. The results we proved earlier show that a similar phenomenon happens here as well. The obstructions to homological connectivity are the positive-negative pairs we saw above, where the positive $k$-face remains isolated until the negative face joins the complex. Thus, the point where the last $k$-cycle is terminated is exactly the point where the last isolated critical $k$-face gets covered. From that point onwards, no new critical isolated $k$-faces are going to appear in the complex. 

Notice, however, that our results pose an important distinction compared to the LM random complex. Homological connectivity occurs when the last isolated \emph{critical} face gets covered. There could be  isolated faces that are not critical, and thus have no effect on homology.
In the ER-graph as well as the LM-complex, all isolated vertices/faces are indeed critical. Therefore, the phenomenon we observe here is merely a refined version of the same behavior.

\subsection{The phase transition for the instantaneous homology}

Recall, that our definition for homological connectivity $\cH_{k,r}$ is such that $H_k(\cC_s) \cong H_k(\T^d)$ for all $s\ge r$. 
However, one may still ask about the probability that $H_k(\cC_r) \cong H_k(\T^d)$ for a specific $r$.
As opposed to $\cH_{k,r}$, this event is non-monotone, suggesting that there might not be a sharp phase transition at all. Nevertheless, the results we proved earlier, together with earlier results in \cite{bobrowski_vanishing_2017} lead to the following statements.

\begin{cor}\label{cor:pt_ihk_1}
Let $1\le k \le d-2$, and $w(n)\to\infty$. Then,
\[
	\limninf \prob{H_k(\cC_r) \cong H_k(\T^d)} = \begin{cases} 1 & \Lambda \ge \log n + (k-1)\logg n - w(n),\\ 0 & \Lambda \le \log n + (k-2)\logg n - w(n).
	\end{cases}
\]
\end{cor}

\begin{cor}\label{cor:pt_ihk_2}
Let $k =d-1,d$, and let $w(n)\to\infty$. Then,
\[
	\limninf \prob{H_k(\cC_r) \cong H_k(\T^d)} = \begin{cases} 1 & \Lambda = \log n + (d-1)\logg n + w(n),\\ 0 & \Lambda = \log n + (d-1)\logg n - w(n).
	\end{cases}
\]
\end{cor}

\begin{proof}[Proof of Corollary \ref{cor:pt_ihk_1}]
For $1\le k \le d-2$, the case where $\Lambda \le \log n + (k-2)\logg n -w(n)$ was proved in \cite{bobrowski_random_2019,bobrowski_vanishing_2017}, by showing the existence of positive critical $k$-faces that are isolated. Suppose now that $\Lambda \ge \log n + (k-1)\logg n - w(n)$.
If $H_k(\cC_r)\not\cong H_k(\T^d)$ then either there is a generator in $H_k(\T^d)$ that is not yet present in $H_k(\cC_r)$, or  there is a generator in $H_k(\cC_r)$ that is trivial in $H_k(\T^d)$.
The former case is excluded by Lemma \ref{lem:perc}. The latter case implies that there is a positive $k$-face appearing in $(0,r)$ that is matched with  negative $(k+1)$-face that appears in $(r,\rmax]$. However, 
when $\Lambda \ge\thres - w(n)$, Lemma \ref{lem:F_5} shows that \whp~such event cannot occur.
Thus, we reached a contradiction, and therefore we conclude that $H_k(\cC_r)\cong H_k(\T^d)$.

\vspace{-5pt}
\end{proof}

\begin{proof}[Proof of Corollary \ref{cor:pt_ihk_2}]
For $k=d-1$, recall that when $\Lambda = \log n + (d-1)\logg n - w(n)$, we have $\Cn_{d,r} > 0$, while $\Cp_{d-1,r} = 0$. Thus, there must be $(d-1)$-cycles existing at $r$ that are to be terminated in $(r,\rmax]$. Therefore, $H_{d-1}(\cC_r)\not\cong H_{d-1}(\T^d)$.

For $k=d$, recall that  only a single $d$-cycle is formed throughout the filtration, and that when $\Lambda = \log n + (d-1)\logg n - w(n)$ we have $\Cp_{d,r} = 1$. Therefore, the single cycle in $H_d(\T^d)$ is to be generated in $(r,\rmax]$, implying that
 $H_d(\cC_r)\not\cong H_d(\T^d)$.

Finally, when $\Lambda = \log n + (d-1)\logg n + w(n)$, we showed that both $\cH_{d,r}$ and $\cH_{d-1,r}$ hold. This implies that $H_{d-1}(\cC_r) \cong H_{d-1}(\T^d)$ and $H_d(\cC_r)\cong H_d(\T^d)$, concluding the proof.

\end{proof}

\section{The occurrences of the last cycles}\label{sec:pois_limit}

The conclusion from Section \ref{sec:crit_window}
is that when $\Lambda = \thres + o(\logg n)$, the interference to homological connectivity is controlled by $\Cnp_{k+1,r}$ --  the number of positive-negative pairs  appearing in $(r,\rmax]$. In addition, from Lemma \ref{lem:pn_pairs} and Proposition \ref{prop:pt_neg} we have that $\Cnp_{k+1,r} = \Cp_{k,r} = F_{k,r}$ (\whp). 
Thus, in order to analyze the distribution for the number of obstructions remaining in $(r,\rmax]$ it is enough to study the distribution of $F_{k,r}$.

Recall from Proposition \ref{prop:mean_var} that 
\[
	\mean{F_{k,r}} \approx 
D_k n\Lambda^{k-1} e^{-\Lambda}.
\]
Thus, when $\Lambda = \thres - w(n)$ we still have many cycles. However, when $\Lambda = \thres + \lambda$ for $\lambda \in \R$, then from Proposition \ref{prop:mean_var}
\[
\mean{F_{k,r}} \approx  \var{F_{k,r}}\approx D_k e^{-\lambda},
\]
so the number of interferences is finite in expectation and variance, suggesting a Poisson limit. In this section we will show not only that there exists a point-wise Poisson limit for $F_{k,r}$, but also that the limiting distribution for the entire process of occurrences is that of a Poisson process in $\R$ (with respect to $\lambda$).
These Poisson limits will then enable us to derive the exact limiting probability of $\cH_{k,r}$.

\subsection{Poisson limits}

For any finite $n$, there exist finitely many critical $k$-faces with distinct radii (\whp). Denote by $\cR_k\subset \R^+$ the (finite) set of all critical radii, and by $\Rp_k, \Rn_k$ the positive and negative critical radii, respectively.
Define $\Delta_{k,n}:\R^+\to\R^+$ as
\[
	\Delta_{k,n}(r) := n(\log n)^{k-1} e^{-\omega_d nr^d},
\]
so that
\eqb\label{eq:lambda_delta}
	 \Lambda = \thres  + \lambda\quad\Leftrightarrow \quad   \Delta_{k,n}(r) = e^{-\lambda}.
\eqe

For $t\in [0,\infty)$, define
\eqb\label{eq:crit_proc}
\splitb
	N_k(t) &:= \# {r \in \cR_k : \Delta_{k,n}(r) \le t},\\
		\Np_k(t) &:= \# {r \in \Rp_k : \Delta_{k,n}(r) \le t},\\
			\Nn_{k}(t) &:= \# {r \in \Rn_{k+1} : \Delta_{k,n}(r) \le t}.
\splite
\eqe
In other words, the process $N_k$ counts the occurrences of the last critical $k$-faces, but in a reversed order. Similarly, $\Np_k$ counts the last positive $k$-faces, and $\Nn_k$ counts the negative $(k+1)$-faces.
The following theorem asserts that $N_k,\Np_k, \Nn_k$ all converge weakly to the same Poisson counting process.

\begin{thm}\label{thm:pois_proc}
Let $1 \le k \le d$, and let $V_k(t)$ ($t\ge 0$) be a homogeneous Poisson (counting) process with rate $D_k$. Then,
\[
	N_k \Rightarrow V_k,
\] 
where `$\Rightarrow$' refers to weak convergence of the final dimensional distributions. In other words, for all $m, t_1,\ldots, t_m$, we have a multivariate weak convergence
\[
	(N_k(t_1),\ldots, N_k(t_m)) \Rightarrow
	(V_k(t_1)\ldots, V_k(t_m)).
\]
Similarly, for $1\le k \le d-1$ we have $\Np_k \Rightarrow V_k$, and for $1\le k \le d-2$ we have $\Nn_k \Rightarrow V_k$.
\end{thm}

The following is a point-wise conclusion from the proof of Theorem \ref{thm:pois_proc}
\begin{cor}\label{cor:pois_var}
Let $1\le k \le d$, and set $\Lambda = \thres + \lambda$, then
\[
	\dtv{F_{k,r}}{ Z_{\lambda}} \to 0,
\]
where $Z_{\lambda} \sim \pois{D_k e^{-\lambda}}$, and $d_{\mathrm{TV}}$ is the total-variation distance.\\
The same limit holds for $\Cp_{k,r}$ ($1\le k \le d-1$) and $\Cn_{k+1,r}$ ($1\le k \le d-2$).
\end{cor}

\begin{proof}[Proof of Theorem \ref{thm:pois_proc}]
We first prove that $N_k \Rightarrow V_k$.
Fix $t_0 > 0$, and set $\R_0 = [0,t_0]$.
 Define
\eqb\label{eq:time_proc}
	\cL_{k,t_0} := \set{\Delta_{k,n}(r): r\in \cR_k} \cap \R_0,
\eqe
so that $\cL_{k,t_0}$ is a finite random point process in $\R_0$.
Let $\mu_{k,t_0}$ be a homogeneous  Poisson process in $\R_0$ with rate $D_k$. 
We will prove that for every $t_0$ the point process $\cL_{k,t_0}$ converges to $\mu_{k,t_0}$ in the Kantorovich-Rubinstein distance (see \cite{decreusefond_functional_2016}), which implies weak convergence. 
Using Theorem 3.1 in \cite{decreusefond_functional_2016}, we have
\[
d_{\mathrm{KR}}(\cL_{k,t_0},\mu_{k,t_0})\le \dtv{L_{k,t_0}}{M_{k,t_0})} + 2(\var{\cL_{k,t_0}(\R_0)} - \mean{\cL_{k,t_0}(\R_0)}),
\]
where $L_{k,t_0}, M_{k,t_0}$ are the intensity measures of $\cL_{k,t_0}, \mu_{k,t_0}$ respectively.
First, notice that $\cL_{k,t_0}(\R_0) = F_{k,r_0}$ where $\omega_d n r_0^d = \thres + \lambda_0$, and $\lambda_0 = -\log(t_0)$.
Using Proposition \ref{prop:mean_var}, together with the fact that $\mean{F_{k,r_0}} \approx D_k e^{-\lambda_0}$, we have  
\[
(\var{\cL_{k,t_0}(\R_0)} - \mean{\cL_{k,t_0}(\R_0)}) \to 0.
\]
We therefore need to bound the total variation distance between the intensity measures $L_{k,t_0}$ and $M_{k,t_0}$.  
Throughout, we will assume that $n$ is large enough so that 
\eqb\label{eq:n_large}
 \lambda_0 > -\logg n.
\eqe
Let $B\in \cB(\R_0)$ be a Borel set, and denote
\[
B' := -\log(B),\qquad	B'' := \log n + (k-1)\logg n +B'.
\]
Since $\cL_{k,t_0}(B)$ counts the number of critical $k$-faces with $\Delta_{k,n}(\rho) \in B$, we can follow
the same steps as in the proof of Lemma \ref{lem:mean_ck}, and have
\[
L_{k,t_0}(B) = \mean{\cL_{k,t_0}(B)} = D_k n\int_{s\in B''} s^{k-1} e^{-s}ds.
\]
Taking a change of variables $s\to  \log n + (k-1)\logg n + \lambda$, we have
\[
\splitb
L_{k,t_0}(B) &= D_k n\int_{\lambda\in  B'} (\log n + (k-1)\logg n + \lambda)^{k-1} e^{-(\log n + (k-1)\logg n +\lambda)}d\lambda \\
&= D_k\int_{\lambda\in  B'} \param{1+ \frac{(k-1)\logg n + \lambda}{\log n}}^{k-1} e^{-\lambda}d\lambda.
\splite
\]
On the other hand, since $\mu_{k,t_0}$ is a homogeneous Poisson process,
\[
	M_{k,t_0}(B) = D_k \int_B dt = D_k\int_{\lambda\in B'} e^{-\lambda}d\lambda.
\]
From \eqref{eq:n_large}, have that $L_{k,t_0}(B) > M_{k,t_0}(B)$, and therefore
\[
\splitb
	|L_{k,t_0}(B)-M_{k,t_0}(B)| &= \const \int_{B'} \param{\param{1+\frac{(k-1)\logg n+ \lambda}{\log n}}^{k-1}-1}e^{-\lambda} d\lambda\\
&= \const \sum_{i=1}^{k-1}\binom{k-1}{i} \int_{B'}\param{\frac{(k-1)\logg n + \lambda}{\log n}}^i e^{-\lambda} d\lambda\\ 
&= \const \sum_{i=1}^{k-1}\sum_{j=0}^{i}\binom{k-1}{i}\binom{i}{j} \frac{((k-1)\logg n)^{i-j}}{(\log n)^i}\int_{B'} \lambda^j e^{-\lambda}d\lambda \\
&\le \const \sum_{i=1}^{k-1}\sum_{j=0}^{i}\frac{(\logg n)^{i-j}}{(\log n)^i}\int_{\lambda_0}^\infty\lambda^j e^{-\lambda} d\lambda.
\splite
\]
Since $\int_{\lambda_0}^\infty\lambda^j e^{-\lambda}d\lambda$ is a finite positive constant, we have that for all $B\in \cB(\R)$
\[
|L_{k,t_0}(B)-M_{k,t_0}(B)| \le \const \frac{\logg n}{\log n}.
\]
Since this bound is independent of $B$, we conclude that
\[
	\dtv{L_{k,t_0}}{M_{k,t_0}} \to 0.
\]
Thus, we proved that $\cL_{k,t_0} \xrightarrow{\mathrm{KR}} \mu_{k,t_0}$, which  implies that for every $t_1,\ldots, t_m \in [0,t_0]$ we have
\eqb\label{eq:weak_fin}
	(N_k(t_1), \ldots, N_k(t_m)) \Rightarrow 	(V_k(t_1), \ldots, V_k(t_m)).
\eqe
Since $t_0$ can be as large as we want, this convergence holds for all $t_1,\ldots, t_m \in [0,\infty)$, and thus we showed that $N_k\Rightarrow V_k$.
Next, for every $t_1,\ldots, t_m$ we have
\[
	\prob{\exists i : \Np_k(t_i) \ne N_k(t_i)}
	\le \sum_{i=1}^m \prob{\Np_k(t_i) \ne N_k(t_i)} = \sum_{i=1}^m \prob{\Cp_{k,r_i} \ne F_{k,r_i}},
\]
where $t_i = \Delta_{k,n}(r_i)$. From Proposition \ref{prop:pt_neg}, for $1\le k \le d-1$  we  have that $\prob{\Cp_{k,r_i}\ne F_{k,r_i}} \to 0$. This, together with \eqref{eq:weak_fin}, implies that,
\[
	(\Np_k(t_1), \ldots, \Np_k(t_m)) \Rightarrow 	(V_k(t_1), \ldots, V_k(t_m)).
\]
Similar inequalities together with Lemma \ref{lem:pn_pairs}, proves the same result for the process $\Nn_{k}$ for $1\le k \le d-2$. This completes the proof.



\end{proof}

\begin{proof}[Proof of Corollary \ref{cor:pois_var}]
Let $r$ be such that $\omega_d nr^d = \thres + \lambda$.
Set $t = e^{-\lambda}$, then from \eqref{eq:lambda_delta} we have that $\rho >r$ if and only if $\Delta_{k,n}(\rho) < t$.
Take any $t_0 >  t$, and recall the definition of $\cL_{k,t_0}$ in \eqref{eq:time_proc},
we thus have $F_{k,r} = \cL_{k,t_0}([0,t])$.

Next, set $Z_{\lambda} = \mu_{k,t_0}([0,t])$, then $Z_\lambda\sim \pois{D_k t}=\pois{D_k e^{-\lambda}}$.
Thus, for any $A\subset \N$ we have
\[
	\abs{\prob{F_{k,r}\in A}- \prob{Z_\lambda \in A}}= \abs{\prob{\cL_{k,t_0}([0,t]) \in A} - \prob{\mu_{k,t_0}([0,t]) \in A}}. 
\]
In the previous proof we showed that 
$\cL_{k,t_0} \xrightarrow{\mathrm{KR}} \mu_{k,t_0}$, which implies $\cL_{k,t_0} \xrightarrow{\mathrm{TV}} \mu_{k,t_0}$. Therefore,
\[
	 \dtv{F_{k,r}}{Z_\lambda} = \sup_{A\subset \N} |\prob{F_{k,r}\in A} -\prob{Z_\lambda \in A}|\le \dtv{\cL_{k,t_0}([0,t])}{\mu_{k,t_0}([0,t])} \to 0,
\]
completing the proof for $F_{k,r}$. To prove convergence for $\Cp_{k,r}$, notice that
\[
	\prob{\Cp_{k,r}\in A} = \prob{F_{k,r}\in A} + \delta_{k,r}(A),
\]
where 
\[
\delta_{k,r}(A) = \prob{\Cp_{k,r}\in A,\ \Cp_{k,r} \ne F_{k,r}}-\prob{F_{k,r}\in A,\ \Cp_{k,r} \ne F_{k,r}} .
\]
Therefore,
\[
	\abs{\prob{\Cp_{k,r}\in A}- \prob{Z_\lambda\in A}} \le \abs{\prob{F_{k,r}\in A}- \prob{Z_\lambda\in A}} + \abs{\delta_{k,r}(A)}.
\]
Note that $\abs{\delta_{k,r}(A)} \le \prob{\Cp_{k,r} \ne F_{k,r}} $. Thus, for $1\le k \le d-1$ we have
\[
\sup_{A\subset \N} |\prob{\Cp_{k,r}\in A} -\prob{Z_\lambda \in A}|\le \sup_{A\subset \N} |\prob{F_{k,r}\in A} -\prob{Z_\lambda \in A}| + \prob{\Cp_{k,r} \ne F_{k,r}}  \to 0.
\]
Similarly, we can show for $\Cn_{k+1,r}$ ($1\le k \le d-2$), completing the proof.

\end{proof}

\subsection{Limiting probabilities}
The Poisson limit  proved in Theorem \ref{thm:pois_proc} allows us to derive the following limiting probabilities.

\begin{thm}\label{thm:crit_prob}
Let $1\le k\le d$. For $k\ne d-1$, if $\Lambda = \thres + \lambda$, then
\[
	\limninf \prob{\cH_{k,r}} = 
	e^{-D_k e^{-\lambda}}.
\]
For $k=d-1$, if $\Lambda = \log n +(d-1)\logg + \lambda$ then
\[
	\limninf \prob{\cH_{d-1,r}} = 
	e^{-D_d e^{-\lambda}}(1+D_de^{-\lambda}).
\]
\end{thm}

\begin{proof}
For $1\le k \le d-2$, recall that $\cH_{k,r}$ holds if and only if $\Cn_{k+1,r} = 0$. Therefore, we can use Corollary \ref{cor:pois_var} and have
\[
	\prob{ \cH_{k,r}} = \prob{\Cn_{k+1,r} = 0} \to  e^{-D_k e^{-\lambda}}.
\]
Similarly, for $k=d$, we showed earlier that $\cH_{d,r}$ holds if and only if $F_{d,r} = 0$, and therefore
\[
	\prob{\cH_{d,r}} = \prob{F_{d,r} =0} \to e^{-D_d e^{-\lambda}}.
\]
Finally, for $k=d-1$,  we showed earlier that $\cH_{d-1,r}$ holds if and only if $F_{d,r} \le 1$. Therefore,
\[
\prob{\cH_{d,r}} = \prob{F_{d,r} \in \set{0,1}} \to  e^{-D_de^{-\lambda}} + D_d e^{-\lambda}  e^{-D_de^{-\lambda}}.
\]
That completes the proof.

\end{proof}

\subsection{The limiting variance}
To prove the Poisson limits in this section, we heavily relied on the fact that $\mean{F_{k,r}}\approx \var{F_{k,r}}$. This was stated in Proposition \ref{prop:mean_var}, but not yet proved. We are now ready to pay this debt.

\begin{proof}[Proof of Proposition \ref{prop:mean_var} - the limiting variance]

Recall the definition of $F_{k,r}$ then
\[\splitb
\mean{F_{k,r}^2 } &=   \mean{\sum_{\cX_1 ,\cX_2 \in \cC^k(\cP_n) } g_{r}(\cX_1,\cP_n) g_{r}(\cX_2,\cP_n)} \\
&= \sum_{j=0}^{k+1} \mean{\sum_{ \abs{ \cX_1 \cap \cX_2 } =j} g_{r}(\cX_1,\cP_n) g_{r}(\cX_2,\cP_n) }\\
&= \sum_{j=0}^{k+1}I_j,
\splite 
\]
where $I_j$ denotes the $j$-th term in the sum.
Notice that $I_{k+1} = \mean{F_{k,r}}$. Therefore, we have
\eqb\label{eq:var_ck}
	\var{F_{k,r}} = \meanx{F_{k,r}^2}-\mean{F_{k,r}}^2 =  \mean{F_{k,r}} + \sum_{j=1}^k {I_j}+({I_0} - \meanx{F_{k,r}}^2),
\eqe
Our goal next is to bound the terms
on the right hand side.

\noindent\underline{Bounding ${I_j},\ \ j\ge 1$:}\\

 Using Palm theory (using Corollary \ref{cor:palm}) we have
\eqb\label{eq:mean_I_j}
\begin{split}
{I_j} &=   \mean{ \sum_{\abs{ \cX_1 \cap \cX_2 } = j}   g_{r}(\cX_1,\cP_n) g_{r}(\cX_2,\cP_n) }  \\ 
&=    \frac{n^{2k+2-j}}{ j!((k+1-j)!)^2} \meanx{  g_{r}(\cX'_1, \cX'\cup \cP_n) g_{r}(\cX'_2, \cX' \cup \cP_n)  } \\ 
& =  \const n^{2k+2-j}\int_{(\T^d)^{2k+2-j}} h_{r}( \bx_1) h_{r}( \bx_2) \hsep(\bx_1,\bx_2) e^{-n \vuni(\bx_1, \bx_2)} d\bx \\
&= \const n^{2k+2-j}\int_{\cA} e^{-n \vuni(\bx_1, \bx_2)} d\bx,
\end{split}
\eqe
where
\[
\splitb
\bx &= (x_1,\ldots, x_{2k+2-j})\subset (\T^d)^{2k+2-j},\\
\bx_1 &= (x_1,\ldots, x_{k+1}),\\
\bx_2 &= (x_1,\ldots, x_j, x_{k+2},\ldots, x_{2k+2-j}),\\
\hsep(\bx_1,\bx_2) &= \ind\set{\bx_1 \cap B(\bx_2) = \emptyset,\ \bx_2 \cap B(\bx_1) = \emptyset}, \\
\cA_j &= \set{\bx \in (\T^d)^{2k+2-j} : h_{r}(\bx_1)h_{r}(\bx_2)\hsep(\bx_1,\bx_2) = 1}.
\splite
\]
Note that $\hsep$ verifies that none of the points of $\bx_i$ are inside the open ball $B(\bx_j)$, which is required for both subsets to generate a critical face (see Lemma \ref{lem:crit_face}). Also, notice that the symmetry between $\bx_1,\bx_2$ yields,
\eqb\label{eq:int_vuni}
\splitb
\int_{\cA_j} e^{-n \vuni(\bx_1, \bx_2)}d\bx &= \int_{\cA_j} e^{-n \vuni(\bx_1, \bx_2)}\indf{\rho(\bx_1) \ge \rho(\bx_2)}d\bx \\
&+\int_{\cA_j} e^{-n \vuni(\bx_1, \bx_2)}\indf{\rho(\bx_1) < \rho(\bx_2)}d\bx \\
&= 2\int_{\cA_j'} e^{-n \vuni(\bx_1, \bx_2)},
\splite
\eqe
where
\[
    \cA'_j = \set{\bx\in \cA_j : \rho(\bx_1) \ge \rho(\bx_2)}.
\]
Thus, from here on we will assume that $\rho(\bx_1)\ge \rho(\bx_2)$.
In addition, we define $\bx_{12} = (x_1,\ldots, x_j)$ to be the points shared by $\bx_1,\bx_2$, and 
\[
\splitb
	c_i &= c(\bx_i),\quad 	\rho_i = \rho(\bx_i),\quad i=1,2,\\
	c_{12} &= c(\bx_{12}),\quad 
	\rho_{12} = \rho(\bx_{12}),
\splite
\]
Notice that if $\rho_2^2 \le \rho^2_1 - 
|c_1-c_2|^2$, then at least half of the ball $B(\bx_2)$ is contained inside $B(\bx_1)$. The requirement $\hsep(\bx_1,\bx_2)=1$ then implies that all the points of $\theta(\bx_2)$ lie on one hemisphere of $\S^{k-1}$. Following Remark \ref{rem:half_sphere} we have $\hcrit(\bx_2) = 0$, so $\bx\not\in \cA_j$.
Therefore, we will assume from here on that
\eqb\label{eq:rho2_ineq}
	\rho_1^2 - |c_1-c_2|^2 \le \rho^2_2 \le \rho_1^2.
\eqe
We now fix $\eps_j,\delta_j>0$ to be determined later, and split our calculations into three parts, denoted $I_j^{(i)}$, $i=1,2,3$.
Define
\[
\splitb
    I_j^{(1)} &:= \int_{\cA'_j} \indf{ |c_1-c_{12}| \le \eps_j\rho_1} e^{-n\vuni(\bx_1,\bx_2)} d\bx, \\
    I_j^{(2)} &:= \int_{\cA'_j} \indf{|c_1-c_{12}| >\eps_j\rho_1, |c_1-c_2|\le \delta_j\rho_1} e^{-n\vuni(\bx_1,\bx_2)} d\bx \\
    I_j^{(3)} &:= \int_{\cA'_j} \indf{|c_1-c_{12}| >\eps_j\rho_1, |c_1-c_2|> \delta_j\rho_1} e^{-n\vuni(\bx_1,\bx_2)} d\bx,
\splite
\]
so that 
\eqb\label{eq:Ij_sum}
    I_j =\const n^{2k+2-j}(I_j^{(1)}+I_j^{(2)}+I_j^{(3)}).
\eqe




\underline{Bounding $I_j^{(1)}$:}\

The extended BP-formula in Lemma \ref{lem:bp_partial} allows us to convert only part of the variables to polar coordinates, while keeping the rest in their original form.
In our case, we use it to apply a change of variables to $\bx_2$ while keeping $\bx_{12}$ unchanged. This yields,
\[
\splitb
I_j^{(1)} &= \param{\frac{k!}{(j+1)!}}^{d-k+1}\int_{\bx_1} d\bx_1\int_{c_2} dc_2  \int_{\Pi_0^{^\perp}} d\Pi_0^{^\perp} \int_{\hat\bth_2} d{\hat\bth_2} \\
&\times h_{r}(\bx_1) h_{r}(\bx_2) \indf{\rho_1\ge \rho_2,\ |c_1-c_{12}|\le \eps_j\rho_1}  \\
&\times   |c_2-c_{12}|^{-(d-k)}\rho_2^{(d-1)(k+1-j)}\param{\frac{\vsimp(\bth_2)}{\vsimp(\bth_{12})}}^{d-k+1} e^{-n \vuni(\bx_1,\bx_2)}.
\splite
\]
where $\bth_2,\bth_{12}$ are the spherical coordinates for $\bx_2,\bx_{12}$, respectively, and $\hat\bth_2 = \bth_2 \bs \bth_{12}$.
First, notice that $\rho_2 \le \rho_1$. In addition, using Corollary \ref{cor:ratio_vol}, we have that $\vsimp(\bth_2)/\vsimp(\bth_{12})\le 1$. Thus, we have
\[
\splitb
I_j^{(1)} &\le \const \int_{d\bx_1} d\bx_1 \int_{c_2}dc_2 \int_{\Pi_0^{^\perp}}d\Pi_0^{^\perp} \int_{\hat\bth_2}d\hat\bth_2 \\
&\times h_{r}(\bx_1) \indf{\rho_1\ge \rho_2,\ |c_1-c_{12}|\le \eps_j\rho_1}     |c_2 - c_{12}|^{-(d-k)}\rho_1^{(d-1)(k+1-j)} e^{-\omega_d n\rho_1^d}.
\splite
\]
Next, we will change $c_2$ into polar coordinates around $c_{12}$, so that $c_2 =c_{12}+ \tau\psi$, where $\tau\in [0,\infty)$, and $\psi\in \S^{d-j}$ ($c_2 \in \Pi^{^\perp}(\bx_{12})\cong \R^{d+1-j}$). Therefore, $dc_2 = \tau^{d-j}d\tau d\psi$, and then
\[
\splitb
I_j^{(1)} &\le \const \int_{d\bx_1} d\bx_1 \int_{\tau} \tau^{k-j} d\tau \int_{\psi}d\psi \int_{\Pi_0^{^\perp}}d\Pi_0^{^\perp} \int_{\hat\bth_2}d\hat\bth_2 \\
&\times h_{r}(\bx_1) \indf{\rho_1\ge \rho_2,\ |c_1-c_{12}|\le \eps_j\rho_1}\rho_1^{(d-1)(k+1-j)}e^{-\omega_d n\rho_1^d},
\splite
\]
where we used the fact that $|c_2-c_{12}| = \tau$.
Since we consider only $\rho_2 \le \rho_1$, then $\tau = |c_2-c_{12}| \le |c_1-c_{12}|  \le \eps_j\rho_1$. Also notice that $|c_1-c_{12}|\le \eps_j \rho_1$ implies that there exists a face of $\sigma(\bx_1)$ at distance less than $\eps_j\rho_1$ from $c_1$. By the definition of $\phi(\bx_1)$ \eqref{eq:phi_X} we have $\phi(\bx_1)\le \eps_j$.
Altogether, we have
\eqb\label{eq:ineq_Ij1}
\splitb
I_j^{(1)} &\le \const \int_{d\bx_1} d\bx_1  h_{r}(\bx_1) \indf{\phi(\bx_1)\le\eps_j}(\eps_j\rho_1)^{k+1-j}\rho_1^{(d-1)(k+1-j)}e^{-\omega_d n\rho_1^d}\\
&=\const \eps_j^{k+1-j} 
\int_{d\bx_1} d\bx_1  h_{r}(\bx_1) \indf{\phi(\bx_1)\le\eps_j} \rho_1^{d(k+1-j)}e^{-\omega_d n\rho_1^d}\\
&\le\const \eps_j^{k+2-j} 
\int_{d\rho_1} \rho_1^{dk-1}\rho_1^{d(k+1-j)}e^{-\omega_d n\rho_1^d}d\rho_1\\
&\le\const \eps_j^{k+2-j}n^{-(2k+1-j)} \Lambda^{2k-j}e^{-\Lambda},
 \splite
\eqe
where in the last inequality we followed similar steps to the proof of Lemma \ref{lem:Fk_phi}.

\underline{Bounding $I_j^{(2)}$:}\ 

As before, we apply the
 change of variables in Lemma \ref{lem:bp_partial}, 
\[
\splitb
I_j^{(2)} &=\const \int_{\bx_1} d\bx_1\int_{c_2} dc_2  \int_{\Pi_0^{^\perp}} d\Pi_0^{^\perp} \int_{\hat\bth_2} d_{\hat\bth_2} \\
&\times h_{r}(\bx_1) h_{r}(\bx_2) \indf{\rho_1\ge \rho_2,\ |c_1-c_{12}|> \eps_j\rho_1,\ |c_2-c_1| \le \delta_j\rho_1}  \\
&\times   |c_2-c_{12}|^{-(d-k)}\rho_2^{(d-1)(k+1-j)}\param{\frac{\vsimp(\bth_2)}{\vsimp(\bth_{12})}}^{d-k+1} e^{-n \vuni(\bx_1,\bx_2)},
\splite
\]
Since  $|c_1 -c_2| \le \delta_j\rho_1$, and since later we will choose $\delta_j = o(\eps_j)$, we have
\[
    |c_2-c_{12}| \ge |c_1-c_{12}|-|c_1-c_2| \ge (\eps_j-\delta_j)\rho_1 \ge \eps_j\rho_1/2,
\]
Also, since $\vsimp(\bth_2)/\vsimp(\bth_{12})\le 1$, we have
\[
\splitb
I_j^{(2)} &\le \const \int_{\bx_1} d\bx_1 h_{r}(\bx_1)e^{-n\omega_d\rho_1^d}(\eps_j \rho_1)^{-(d-k)} \rho_1^{(d-1)(k+1-j)}\int_{c_2}dc_2    \indf{|c_2-c_1| \le \delta_j\rho_1}.
\splite
\]
The last integral is merely the volume of a ball of radius $\delta_j\rho_1$ around $c_1$ in the plane $\Pi^{^\perp}(\bx_{12}) \cong \R^{d-j+1}$, which is  equal to $\omega_d (\delta_j\rho_1)^{d+1-j}$.
Thus, we have
\eqb\label{eq:ineq_Ij2}
\splitb
    I_j^{(2)} &\le \const \eps_j^{-(d-k)}\delta_j^{d+1-j} \int_{\bx_1} d\bx_1 h_{r}(\bx_1)e^{-n\omega_d\rho_1^d}\rho_1^{d(k+1-j)}\\
    &= \const  \eps_j^{-(d-k)}\delta_j^{d+1-j}n^{-(2k + 1 -j)} \Lambda^{2k-j}e^{-\Lambda}.
 \splite
\eqe

\underline{Bounding $I_j^{(3)}$:}\  

In Lemma \ref{lem:vol_r1_r2} we define $\vint(r_1,r_2,\Delta)$ as the volume of intersection of balls with radii $r_1,r_2$, and whose centers are $\Delta$ apart. Then,
\[
	\vuni(\bx_1,\bx_2) = \omega_d(\rho_1^d + \rho_2^d) - \vint(\rho_1,\rho_2,|c_1-c_2|).
\]
Recall \eqref{eq:rho2_ineq}, then the conditions of Lemma \ref{lem:vol_r1_r2} hold, and we have
\[
	\vuni(\bx_1,\bx_2) \ge \frac{\omega_d}{2}(\rho_1^d + \rho_2^d) + \kint|c_1-c_2|(\rho_1^{d-1}+\rho_2^{d-1}).
\]
Since in $I_j^{(3)}$  we have $|c_1-c_2| > \delta_j\rho_1$, and  $\rho_1,\rho_2 \ge r$, we have
\[	
n\vuni(\bx_1,\bx_2) \ge \frac{\omega_d}{2} n(\rho_1^d + r^d) + 2\kint \delta_j n r^d = \Lambda(1/2 + C_2\delta_j) + (\omega_d/2)n\rho_1^d,
\]
where $C_2 =2\kint/\omega_d$. Therefore,
\[
I_j^{(3)} \le  e^{- \Lambda(1/2 +  C_2\delta_j) } 
\int_{\cA_j'} e^{-(\omega_d/2)n\rho^d_1}.
\]
Next, recall that in $\cA_j'$ we have $\hsep(\bx_1,\bx_2) = 1$, and therefore the points in $\bx\bs \bx_1$ must all lie outside a ball of radius $\rho_1$ around $c_1$.
On the other hand, since $B(\bx_1)\cap B(\bx_2)\ne \emptyset$, we have $|c_1-c_2| \le 2\rho_1$ and therefore the points in $\bx\bs\bx_1$ are inside a ball of radius $3\rho_1$ around $c_1$. Thus, we conclude that $\bx\bs\bx_1$  lie in an annulus with radii range $[\rho_1, 3\rho_1]$ around $c_1$, whose volume is $(3^d-1)\omega_d\rho_1^d $. This yields,
\eqb\label{eq:ineq_Ij3}
\splitb
I_j^{(3)} &\le  e^{- \Lambda(1/2 +  C_2\delta_j) } 
\int_{\bx_1} h_{r}(\bx_1)e^{-(\omega_d/2)n\rho^d_1} \\
&\times \int_{\bx\bs\bx_1} \indf{\bx\bs\bx_1 \subset A_{[\rho,3\rho]}}d\bx,\\
&=\const e^{-\Lambda (1/2+ C_2\delta_j)} \int_{\bx_1} h_{r}(\bx_1) e^{-(\omega_d/2)\rho^d_1}\rho_1^{d(k+1-j)}d\bx 
\\
&\le \const n^{-(2k+1-j)} \Lambda^{2k-j}e^{-\Lambda- C_2\delta_j\Lambda}.
\splite
\eqe

Putting  \eqref{eq:ineq_Ij1}-\eqref{eq:ineq_Ij3} back into \eqref{eq:Ij_sum}
we have
\eqb\label{eq:ineq_Ij}
{I_j} \le \const n\Lambda^{k-1} e^{-\Lambda} \param{\eps_j^{k+2-j} \Lambda^{k+1-j}
+
\param{\frac{\delta_j}{\eps_j}}^{d-k} (\delta_j\Lambda)^{k-j+1}
 +
\Lambda^{k+1-j} e^{-C_2\delta_j\Lambda}}.
\eqe

\underline{Bounding $I_0 - \meanx{F_{k,r}}^2$}:\ 

Back to \eqref{eq:var_ck}, we are left with bounding the difference ${I_0}-\mean{F_{k,r}}^2$.
Define,
\[
\Phi(\cX_1,\cX_2) := \ind\set{B(\cX_1)\cap B(\cX_2) = \emptyset},
\]
then we can write
\[
\splitb
{I_0} &=    \mean{ \sum_{\cX_1, \cX_2 \in \cC^k(\cP_n)}    g_{r}(\cX_1,\cP_n) g_{r}(\cX_2,\cP_n) } \\
&=   \mean{  \sum_{\cX_1, \cX_2 \in \cC^k(\cP_n)} g_{r}(\cX_1,\cP_n) g_{r}(\cX_2,\cP_n)  \Phi(\cX_1,\cX_2) }\\ 
&+    \mean{ \sum_{\cX_1, \cX_2  \in \cC^k(\cP_n)}   g_{r}(\cX_1,\cP_n) g_{r}(\cX_2,\cP_n) (1-\Phi(\cX_1,\cX_2)) } \\ 
 &=  {T_1} +  {T_2},
\splite
\]
where $\cX_1,\cX_2$ are disjoint subsets, and $T_1,T_2$ correspond to each of the terms. Next, we will show that ${T_1}-\mean{F_{k,r}}^2 \le 0$, and that will leave us with bounding ${T_2}$.

Using Palm theory (Theorem \ref{thm:palm}) we have
\[
\begin{split}
\meanx{F_{k,r}}^2 & =  \frac{n^{2k+2}}{((k+1)!)^2} \mean{g_{r}(\cX'_1 , \cX'_1 \cup \cP_n)g_{r}(\cX'_2 , \cX'_2 \cup \cP'_n)} \\ 
{T_1 } & =  \frac{n^{2k+2}}{((k+1)!)^2} \mean{g_{r}(\cX'_1 , \cX' \cup \cP_n)g_{r}(\cX'_2 , \cX' \cup \cP_n)  \Phi(\cX_1',\cX_2')   },
\end{split}
\]
where $\cX'_1, \cX'_2$ are independent sets of $k+1$ points uniformly distributed in $\T^d$, $\cX'= \cX'_1 \cup \cX'_2$, and $\cP_n'$ an independent copy of $\cP_n$. Thus, we have 
\[
\splitb
 {T_1}-\mean{F_{k,r}}^2  &=   \frac{n^{2k+2}}{((k+1)!)^2}  \Big(  \mean{g_{r}(\cX'_1 , \cX' \cup \cP_n)g_{r}(\cX'_2 , \cX' \cup \cP_n) \Phi(\cX'_1,\cX'_2) } \\
&  \qquad\qquad\quad\ \  - \mean{  g_{r}(\cX'_1 , \cX'_1 \cup \cP_n)g_{r}(\cX'_2 , \cX'_2 \cup \cP'_n)  \Phi(\cX'_1,\cX'_2)}    \\ 
&  \qquad\qquad\quad\ \ - \mean{   g_{r}(\cX'_1 , \cX'_1 \cup \cP_n)g_{r}(\cX'_2 , \cX'_2 \cup \cP'_n) (1- \Phi(\cX'_1,\cX'_2))   }   \Big) \\ 
& \leq    \frac{n^{2k+2}}{((k+1)!)^2}  \Big(  \mean{g_{r}(\cX'_1 , \cX'_1 \cup \cP_n)g_{r}(\cX'_2 , \cX'_2 \cup \cP_n)  \Phi(\cX'_1,\cX'_2)}    \\ 
&  \qquad\qquad\quad\ \  - \mean{  g_{r}(\cX'_1 , \cX'_1 \cup \cP_n)g_{r}(\cX'_2 , \cX'_2 \cup \cP'_n)  \Phi(\cX'_1,\cX'_2)  }  \Big) \\ 
& =  \frac{n^{2k+2}}{((k+1)!)^2} \mean{\Delta g},
\splite
\]
where
$$\Delta g:= \param{g_{r}(\cX'_1 , \cX'_1 \cup \cP_n)g_{r}(\cX'_2 , \cX'_2 \cup \cP_n) - g_{r}(\cX'_1 , \cX'_1 \cup \cP_n)g_{r}(\cX'_2 , \cX'_2 \cup \cP'_n)}\Phi(\cX'_1,\cX'_2).$$
To show that $\mean{\Delta g} = 0$  we consider the conditional expectation $\E_{\cX'_1,\cX'_2}\set{\cdot} := \cmean{\cdot}{\cX'_1,\cX'_2}$. Given $\cX_1',\cX_2'$, if $\Delta g \ne 0$ then necessarily $B(\cX'_1)\cap B(\cX'_2)=\emptyset$. Using the spatial independence  of the Poisson process $\cP_n$, together with the fact that the value of $g_{r}(\cX'_i, \cX'_i\cup \cP_n)$ only depends on the points of $\cP_n$ lying inside $B(\cX'_i)$,
we conclude that
\[
\splitb
&\E_{\cX'_1,\cX'_2}\set{g_{r}(\cX'_1,\cX'_1\cup\cP_n)g_{r}(\cX'_2,\cX'_2\cup\cP_n)\Phi(\cX'_1,\cX'_2)}  \\ &\qquad= \Phi(\cX'_1,\cX'_2)\E_{\cX'_1,\cX'_2}\set{g_{r}(\cX'_1,\cX'_1\cup\cP_n)}\E_{\cX'_1,\cX'_2}\set{g_{r}(\cX'_2,\cX'_2\cup \cP_n)} \\
&\qquad= \Phi(\cX'_1,\cX'_2)\E_{\cX'_1,\cX'_2}\set{g_{r}(\cX'_1,\cX'_1\cup\cP_n)}\E_{\cX'_1,\cX'_2}\set{g_{r}(\cX'_2,\cX'_2\cup \cP'_n)},
\splite
\]
since $\cP_n$ and $\cP_n'$ are independent and have the same distribution. Thus, we have that $\E_{\cX_1,\cX_2}\set{\Delta g} = 0$.
Consequently, $\mean{\Delta g} = \mean{\E_{\cX_1,\cX_2}\set{\Delta g} }=0$.

In order to bound ${T_2}$, we  write
\[
\splitb
    I_0^{(1)} &:= \int_{\cA'_0} \indf{|c_1-c_2|\le \eps_0\rho_1} e^{-n\vuni(\bx_1,\bx_2)} d\bx \\
    I_0^{(2)} &:= \int_{\cA'_0} \indf{|c_1-c_2|> \eps_0\rho_1} e^{-n\vuni(\bx_1,\bx_2)} d\bx.
\splite
\]
Note that
\[
  I_0^{(1)} \le \int_{\bx_1} h_{r}(\bx_1)  e^{-n\omega_d\rho_1^d}
\int_{\bx_2} h_{r}(\bx_2) \indf{\rho_2\le\rho_1,\  |c_1-c_2|\le \eps_0 \rho_1}d\bx_2 d\bx_1.
\]
Since $\bx_1,\bx_2$ are disjoint, we can 
 apply the BP formula from Lemma \ref{lem:bp_torus} to $\bx_2$, and have
\[
\splitb
  I_0^{(1)} &\le \int_{\bx_1} h_{r}(\bx_1) e^{-n\omega_d \rho_1^d} d\bx_1
\\
&\times  \int_{c_2}dc_2\int_{\Pi_2}d\Pi_2 \int_{\rho_2} d\rho_2 \int_{\bth_2} d{\bth_2} 
\rho_2^{dk-1} (\vsimp(\bth_2))^{d-k+1} \indf{\rho_2\le \rho_1,\ |c_1-c_2|\le \eps_0\rho_1} .
 \splite
\]
Using the facts that $\vsimp$ is bounded, $\rho_1(1-\eps_0^2)^{1/2} \le \rho_2 \le \rho_1$, and $c_2\in B_{\eps_0\rho_1}(c_1)$, we have
\[
\splitb
I_0^{(1)} &\le \const \int_{\bx_1} h_{r}(\bx_1) e^{-n\omega_d \rho_1^d}(\eps_0\rho_1)^d (\rho_1^{dk}-\rho_1^{dk}(1-\eps_0^2)^{dk/2}) d\bx \\
&= \const \eps_0^{d+2} \int_{\bx_1} h_{r}(\bx_1) e^{-n\omega_d \rho_1^d}\rho_1^{d(k+1)}  d\bx \\
&=\const \eps_0^{d+2}n^{-(2k+1)}  \Lambda^{2k}e^{-\Lambda}.
\splite
\]
Bounding $I_0^{(2)}$ can be done along the same lines as $I_j^{(3)}$ \eqref{eq:ineq_Ij2}, which will lead to
\[
	I_0^{(2)} \le \const n^{-(2k+1)}\Lambda^{2k}e^{-\Lambda -C_2\eps_0\Lambda}.
\]
Therefore, we conclude that
\eqb\label{eq:ineq_I0}
{I_0}-\mean{F_{k,r}}^2 \le \const n\Lambda^{k-1}e^{-\Lambda}\param{\eps_0^{d+2}\Lambda^{k+1} + \Lambda^{k+1}e^{-C_2\eps_0\Lambda}}.
\eqe

We are now ready to put all the pieces together in order to bound the variance.

\underline{Convergences of the variance:}\ 

We want to show that 
\[
	\frac{\var{F_{k,r}}}{\mean{F_{k,r}}} \to 1.
\]
First, notice that from the expectation part of Theorem \ref{prop:mean_var},  we have $\mean{F_{k,r}}
\approx n\Lambda^{k-1}e^{-\Lambda}$.
Using \eqref{eq:ineq_Ij}, we have
\[
\splitb
\frac{I_j}{\mean{F_{k,r}}} &\le
\const\param{\eps_j^{k+2-j} \Lambda^{k+1-j}
+
\param{\frac{\delta_j}{\eps_j}}^{d-k}  (\delta_j\Lambda)^{k-j+1}
 +
\Lambda^{k+1-j} e^{-C_2\delta_j\Lambda}}.
\splite
\]
Taking $\eps_j = \Lambda^{-\frac{k+3/2-j}{k+2-j}} = \frac{\Lambda^{\frac{1/2}{k+2-j}}}{\Lambda}$, and $\delta_j = \frac{k+2-j}{C_2} \frac{\log \Lambda}{\Lambda}$, 
we  have $\frac{{I_j}}{\mean{F_{k,r}}}\to 0$.
Next, we use \eqref{eq:ineq_I0} to have
\[
\frac{{I_0}-\mean{F_{k,r}}^2}{\mean{F_{k,r}}} \le \const \param{\eps_0^{d+2}\Lambda^{k+1} + \Lambda^{k+1}e^{-C_2\eps_0\Lambda}}.
\]
Taking $\eps_0 = \Lambda^{-(k+2)/(d+2)}$, makes the limit go to zero.
Using \eqref{eq:var_ck} , 
we can finally conclude that
\[
\frac{\var{F_{k,r}}}{\mean{F_{k,r}}}\approx 1.
\]
This completes the proof.

\end{proof}

\section{Conclusion}

In this paper we studied homological connectivity in the random \cech complex generated by a homogeneous Poisson process on the torus. We proved a sharp phase transition, as well as a detailed analysis for the obstructions to homological connectivity.
To conclude, we want to propose some  observations and open problems.
\begin{itemize}
\item In this paper we studied the homogeneous Poisson process only. It is interesting (and not straightforward) to study what happens for non-homogeneous processes, or even non-Poisson process (e.g.~\cite{yogeshwaran_topology_2015}).
\item While the results here are stated for the flat torus, they should apply to any compact smooth Riemannian manifold using the machinery developed in \cite{bobrowski_random_2019}. It would be interesting to see what happens in other cases as well. For example - manifolds with  boundary or non-compact spaces. For example, in \cite{penrose_random_2003} it is shown that for the $d$-dimensional normal distribution, the connectivity threshold is $r \approx \frac{(d-1)\logg n}{\sqrt{2\log n}}$, which is much larger than the threshold in \eqref{eq:pt_rgg}. More related results for graphs appear in \cite{gupta_criticality_2010,hsing_extremes_2005}.
\item As seen in Corollary \ref{cor:pt_ihk_1},  the event $\set{H_k(\cC_r)\cong H_k(\T^d)}$ behaves differently than $\cH_{k,r}$ and it would be interesting to try find the exact threshold here as well. Following Corollary \ref{cor:pt_ihk_1} we conjecture that the threshold is at $\Lambda = \log n + (k-2)\logg n$. The lower bound in that case was already proven in \cite{bobrowski_vanishing_2017}, so it remains to find a matching upper bound.
\item Another interesting geometric model is the random \emph{Vietoris-Rips} (aka Rips) complex, where one starts with a random geometric graph $G(n,r)$ and  constructs a clique complex. The results in \cite{iyer_thresholds_2018} suggest a completely different sequence of phase transitions for the Rips complex. We conjecture that there exist sharp phase transitions for homological connectivity in the Rips complex, with thresholds of the form $C_k \log n$ for $C_1 < C_2 \cdots < C_d$. The main difference between the \cech and the Rips complex, is that the entire apparatus of Morse theory applies to the \cech complex only, due to the Nerve Lemma. The Rips complex analysis therefore requires a different toolbox. In \cite{kahle_random_2011} discrete Morse theory was used to derive bounds for homological connectivity. However, it is not clear if this framework can be used to prove a sharp phase transition.
\item Note that  the results for the LM complex $Y_d(n,p)$ are a perfect generalization of the results for the ER graph $G(n,p)$ in the sense that taking $d=1$ in \cite{kahle_inside_2016,meshulam_homological_2009} retrieves the results in \cite{erdos_random_1959}. This is not the case for  the \cech complex $\cC_r$. We saw that the connectivity threshold is $\Lambda = 2^{-d} \log n$ while for $k\ge 1$ we have the thresholds $\Lambda = \thres$. It would be nice to see, though, if the Rips complex provides  a more natural generalization for $G(n,r)$.

\item In the introduction we discussed the connection between isolated faces and homological connectivity. We argued that in the general case, homological connectivity is tied to the vanishing of the \emph{critical} isolated faces. It would still be interesting to examine in detail the process that contains \emph{all} isolated faces, and its connection to \emph{hypergraph connectivity} -- the connectivity of the graph that represent neighboring faces (see \cite{iyer_thresholds_2018}).  We believe that some of the tools developed in this paper could be useful  in providing a more detailed analysis of this problem, and the interaction between homological and hypergraph connectivity.
\end{itemize}

\section*{Acknowledgements}
The author is grateful to the following people, for many helpful conversations and useful advice:
Robert Adler, Asaf Cohen, Yogeshwaran Dhandapani, Herbert Edelsbrunner, Matthew Kahle, Anthea Monod, Sunder Ram-Krishnan, Primoz Skraba, and Shmuel Weinberger. Special thanks to Goncalo Oliveira.

\appendix

\newpage
{\begin{center}\bf \Large APPENDIX\end{center}}

\section{Palm Theory}

The focus of Palm theory (cf.\cite{chiu_stochastic_2013}) is  studying the conditional distribution of a point process $\cP$ given that a fixed set of point $\cY$ is already in $\cP$. The Poisson process has the unique property that its Palm measure is the same as its original distribution. This yields the following useful results, see \cite{bobrowski_distance_2014,penrose_random_2003} for proofs.

\begin{thm}
\label{thm:palm}
Let $(X,\rho)$ be a metric space, $f:X\to\R$ be a probability density on $X$, and let $\cP_n$ be a Poisson process on $X$ with intensity $\lambda_n = n f$.
Let $h(\cX,\cP)$ be a measurable function defined for all finite subsets $\cX \subset \cP \subset X$  with $\abs{\cX} = k$. Then
\[
    \mean{\sum_{ \cX \subset \cP_n}
    h(\cX,\cP_n)} = \frac{n^k}{k!} \mean{h(\cX',\cX' \cup \cP_n)}
\]
where $\cX'$ is a set of $k$ $iid$ points in $X$ with density $f$, independent of $\cP_n$.
\end{thm}

The last formula is also known as the \emph{Campbell-Mecke} formula.
To calculate second moments, we will also need the following corollary.

\begin{cor}\label{cor:palm}
With the notation above, assuming $\abs{\cX_1} = \abs{\cX_2} = k$,
\[
\mean{\sum_{ \substack {
                    \cX_1 ,\cX_2\subset \cP_n  \\
                    \abs{\cX_1 \cap \cX_2} = j }}
    h(\cX_1,\cP_n)h(\cX_2,\cP_n)} = {\frac{n^{2k-j}}{j!((k-j)!)^2}} \mean{h(\cX_1',\cX' \cup \cP_n)h(\cX_2',\cX' \cup \cP_n)}
\]
where $\cX_1',\cX'_2$ are sets of $k$ points with $\abs{\cX_1'\cap\cX_2'} = j$, such that  $\cX' := \cX'_1 \cup \cX'_2$ is a set of $2k-j$ $iid$ points in $X$ with density $f$, independent of $\cP_n$.
\end{cor}

\section{Spherical caps}

Spherical volumes are a key ingredient in our calculations and proofs. In this section we provide some useful estimates. Let $B_r$ be a concentric $d$-dimensional ball defined by
\[
    B_r := \set{x\in \R^d : \abs{x} \le r}.
\]
We define the spherical cap $B_{r,\Delta}$ to be
\[
    B_{r,\Delta} := \set{x = (x^{(1)},\ldots, x^{(d)})\in B_r : x^{(d)} \ge \Delta}.
\]
The following  result is proved in  \cite{li_concise_2011}.

\begin{lem} \label{lem:spher_cap}
The volume of a $B_{1,\Delta}$ is given by
\eqb\label{eq:vol_cap}
	\vcap(\Delta) :=  \vol(B_{1,\Delta}) = 
	\omega_{d-1}\int_\Delta^1 (1-\rho^2)^{\frac{d-1}{2}} d\rho,
\eqe
where $\omega_d$ is the volume of the unit ball in $\R^d$. By scaling we have,
\[
    \vol(B_{r,\Delta}) = r^d \vcap(\Delta/r).
\]
\end{lem}

Following \eqref{eq:vol_cap}, the derivative of $\vcap(\Delta)$ is
\eqb\label{eq:vol_deriv}
\frac{d}{d\Delta}\vcap(\Delta) = -\omega_{d-1}  (1-\Delta^2)^{(d-1)/2}.
\eqe
Therefore, we conclude that 
\eqb \label{eq:vol_taylor}
\vcap(\Delta) = \frac{\omega_d}{2} - \omega_{d-1}\Delta + o(\Delta).
\eqe

In some cases we will also need a higher-order polynomial expansion of $\vcap$ around zero, in this case we can show that the second and forth derivatives vanish, and therefore,
\eqb\label{eq:vol_taylor_2}
\vcap(\Delta) = C_0 - C_1\Delta + C_3\Delta^3 + o(\Delta^4),
\eqe
where $C_0 = \frac{\omega_d}{2}$, $C_1 = \omega_{d-1}$, and $C_3 = \frac{(d-1)\omega_{d-1}}{6}$.

For the proofs in this paper we will need some finer statements about the volume.

\begin{lem}\label{lem:vdiff}
Let $x_1,x_2\in \R^d$ be such that $\abs{x_1-x_2}= \eps>0$, and let $\alpha\in(0,1)$. Define
\[
V_{\diff}(\eps,\alpha):= \vol(B_{1-\alpha\eps}(x_2)\bs B_1(x_1) ).
\]
Then,
\[
\lim_{\eps\to0}\frac{V_{\diff}(\eps,\alpha)}{\eps} \in (0,\infty).
\]
\end{lem}

\begin{proof}

\begin{figure}
    \centering
        \begin{subfigure}{0.45\textwidth}
    \centering
    \includegraphics[scale=0.3]{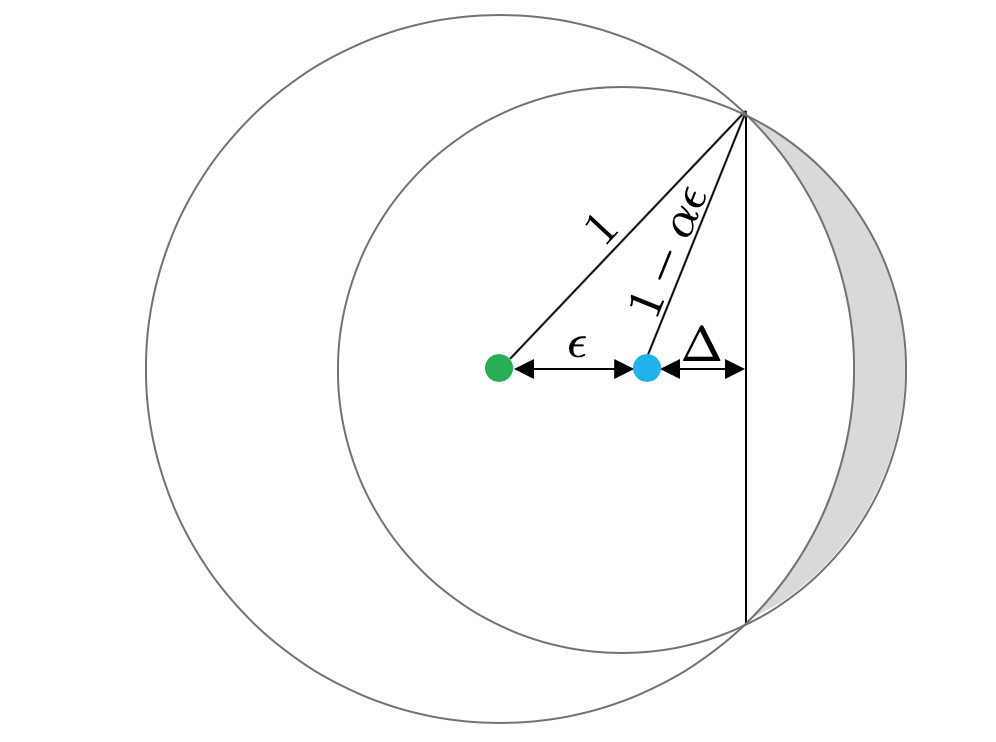}
    \caption{}
    \end{subfigure}
            \begin{subfigure}{0.45\textwidth}
    \centering
    \includegraphics[scale=0.3]{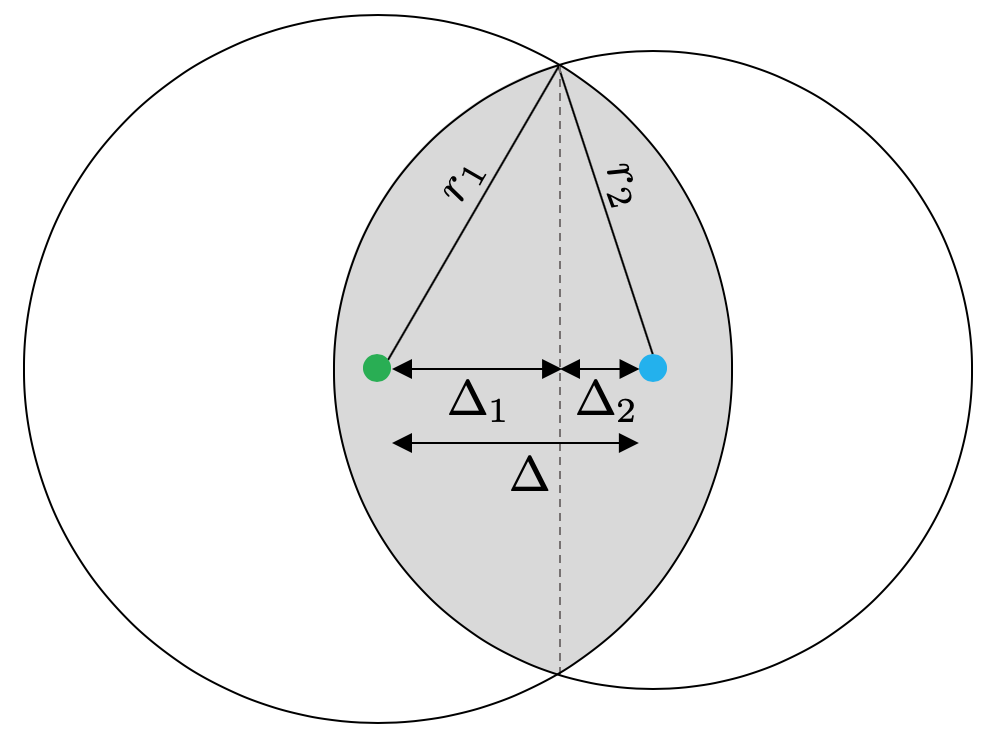}
    \caption{}
    \end{subfigure}
    \caption{(a) The balls $B_1(x_1)$, $B_{1-\alpha\eps}(x_2)$, and their difference (shaded area), whose volume is $\vdiff(\eps,\alpha)$. (b) The intersection of two balls with radii $r_1,r_2$, such that the distance between the centers is $\Delta$.}
    \label{fig:balls}
\end{figure}
Suppose that  $\eps \le 2\alpha / (1+\alpha^2)$, then $V_{\diff}(\eps,\alpha)$ can be depicted as in Figure \ref{fig:balls}(a), where the larger ball is of radius $1$ and the smaller is of radius $1-\alpha\eps$. 
The volume $V_{\diff}(\eps,\alpha)$ is the shaded area which is the difference  between two spherical caps $V^{(1)}-V^{(2)}$,
where $V^{(1)} = \vol(B_{1-\alpha\eps, \Delta})$, and $V^{(2)} = \vol(B_{1, \Delta+\eps})$. Notice that
\[
	(1-\alpha\eps)^2 - \Delta^2 = 1 - (\Delta+\eps)^2,
\]
and solving  for $\Delta$ yields $\Delta = \alpha - \frac{\eps}{2}(1+\alpha^2)$.
Therefore,
\[
\splitb
	V^{(1)} &=  (1-\alpha\eps)^d V_{\scap}\param{\frac{\Delta}{1-\alpha\eps}} \\
		&= (1-d\alpha\eps + o(\eps)) V_{\scap}\param{\Delta(1+\alpha\eps+o(\eps))} \\
		&= (1-d\alpha\eps + o(\eps)) V_{\scap}\param{\alpha-\frac{1-\alpha^2}{2}\eps +o(\eps)} \\
		&= (1-d\alpha\eps + o(\eps))\param{V_{\scap}(\alpha) -\frac{1-\alpha^2}{2}\eps V_{\scap}'(\alpha) +o(\eps)}\\
		&= V_{\scap}(\alpha) - \eps\param{\frac{1-\alpha^2}{2} V_{\scap}'(\alpha) + d\alpha V_{\scap}(\alpha)} + o(\eps),
\splite
\]
where $\vcap'$ is the derivative of $\vcap$.
In addition,
\[
	V^{(2)} = \vcap(\Delta+\eps) = \vcap\param{\alpha+\frac{1-\alpha^2}{2}\eps} = \vcap(\alpha) + \frac{1-\alpha^2}{2}\eps \vcap'(\alpha) + o(\eps).
\]
Thus,
\[
	\vdiff(\eps,\alpha) = V^{(1)}-V^{(2)} = -\eps\param{(1-\alpha^2)\vcap'(\alpha) +d\alpha \vcap(\alpha)} + o(\eps).
\]
This proves that the limit $\lim_{\eps\to 0}\frac{ \vdiff(\eps,\alpha)}{ \eps}$ exists and is finite. It remains to show that it is strictly positive.
Using \eqref{eq:vol_cap} we have that
\[
	\vcap(\alpha) = \omega_{d-1}\int_{\alpha}^1  (1-\rho^2)^{\frac{d-1}{2}}d\rho \le \frac{\omega_{d-1}}{\alpha} \int_{\alpha}^1  \rho(1-\rho^2)^{\frac{d-1}{2}}d\rho = \frac{\omega_{d-1}}{\alpha(d+1)} (1-\alpha^2)^{\frac{d+1}{2}}.
\]
In addition, from \eqref{eq:vol_deriv} we have
\[
	\vcap'(\alpha) = -\omega_{d-1}(1-\alpha^2)^{\frac{d-1}{2}}.
\]
Therefore,
\[
	\vdiff(\eps,\alpha)  \ge \eps \frac{\omega_{d-1}}{d+1} (1-\alpha^2)^{\frac{d+1}{2}} + o(\eps).
\]
Since $\alpha\in(0,1)$ we can conclude that the limit is positive.

\end{proof}

The last statement we need is a uniform bound on the volume of intersection of two balls. 

\begin{lem}\label{lem:vol_r1_r2}
Let $\vint(r_1,r_2,\Delta)$ be the volume of the intersection of balls with radii $r_1,r_2$ and whose centers are at distance $\Delta$ apart. Assuming that $\Delta^2 \ge \abs{r_1^2-r_2^2}$, then there exists $\kint>0$ such that
\[
	\vint(r_1,r_2,\Delta) \le \frac{\omega_d}{2}(r_1^d+ r_2^d) - \kint\Delta(r_1^{d-1}+r_2^{d-1}).
\]
\end{lem}

\begin{proof}
When $\Delta^2 \ge \abs{r_1^2-r_2^2}$, 
the intersection $B_{r_1}(c_1) \cap B_{r_2}(c_2)$ consists of two spherical caps as in Figure \ref{fig:balls}(b).
The $i$-th cap is at distance $\Delta_i$ from $c_i$ and with radius $r_i$, where $\Delta_1+\Delta_2 = \Delta$. Denote its volume by $V^{(i)}$. Then
\eqb\label{eq:V_i}
	V^{(i)} = r_i^d\vcap(\Delta_i / r_i),
\eqe
where $\Delta_i/r_i \le 1$. Next, define
\[
F(\Delta) = \begin{cases}
\frac{\vcap(\Delta) - \frac{\omega_d}{2}}{\Delta} & \Delta \in (0,1],\\
-\omega_{d-1} & \Delta = 0.
\end{cases}
\]
Using \eqref{eq:vol_deriv} we have that $F(\Delta)$ is continuous and negative in $[0,1]$, and thus has a negative maximum. Therefore, there exists $C_1>0$ such that for all $\Delta\in [0,1]$
\[
\vcap(\Delta) \le \frac{\omega_{d}}{2} - C_1\Delta.
\]
Thus, using \eqref{eq:V_i} we have
\eqb\label{eq:v_r1_r2}
	\vint(r_1,r_2,\Delta) = V^{(1)}+V^{(2)} \le\frac{\omega_d}{2}(r_1^d + r_2^d) -  C_1(\Delta_1 r_1^{d-1} + \Delta_2r_2^{d-1})
\eqe

Next, notice that $r_1^2-\Delta_1^2  = r_2^2-\Delta_2^2$,\ or\ $r_1^2-r_2^2 = \Delta_1^2 - \Delta_2^2$.
Therefore, $r_1 \ge r_2$ if and only if $\Delta_1\ge \Delta_2$, which implies
\[
(r_1^{d-1} - r_2^{d-1})(\Delta_1-\Delta_2) \ge 0,
\]
which in turn implies that
\[
r_1^{d-1}\Delta_1 + r_2^{d-1}\Delta_2 \ge \frac{1}{2}(r_1^{d-1}+r_2^{d-1})(\Delta_1+\Delta_2) = \frac{1}{2}(r_1^{d-1} + r_2^{d-1})\Delta.
\]
Taking $\kint= C_1/2$ and putting this together with \eqref{eq:v_r1_r2} yields
\[
\vint(r_1,r_2,\Delta) \le \frac{\omega_d}{2}(r_1^d + r_2^d) -  \kint(r_1^{d-1}+r_2^{d-1})\Delta,
\]
completing the proof.

\end{proof}

\section{Blaschke-Petkantschin-type formulae}\label{sec:bp}

Many of our calculations will involve integrating over sets of points that form critical $k$-faces. The following lemma introduces a change of variables, extending the idea of polar coordinates, that will be highly useful in this context.
Let $\bx=(x_1,\ldots,x_{k+1}) \in (\T^d)^{k+1}$ be the coordinates we want to integrate over. The change of variables we will use follows the mapping $\bx\to (c,\rho,\Pi,\bth)$ discussed in Section \ref{sec:crit_faces}. 
As stated in Section \ref{sec:flat_torus}, we will implicitly assume that $\bx$ is contained in a ball of radius $\rmax$, and that $E(\bx)\cap B_{\rmax}(\bx) \ne \emptyset$. 
This guarantees that the definitions in \eqref{eq:crit_face_1} are all valid. The change of variables we use will include the following steps.
\begin{enumerate}
\item Choose ax center $c$,
\item Choose a radius $\rho$,
\item Choose the $k$-dimensional linear space $\Pi\subset \R^d_c$,
\item Choose points on the  unit sphere $\S^{k-1}\subset \Pi$,  denoted $\bth=(\theta_1,\ldots,\theta_{k+1}) \in (\S^{k-1})^{k+1}$.
\end{enumerate}
These variables are related via
\[
	\bx = c + \rho \bth(\Pi),
\]
where by $\bth(\Pi)$ we mean taking points on the $(k-1)$-sphere in $\Pi$.
In all our calculations, we will consider functions $f(\bx)$ that are both shift and affine invariant. 
This implies that,
\eqb\label{eq:func_inv}
	f(\bx) = f(c+\rho \bth(\Pi)) = f(\rho \bth(\Pi_0)) := f(\rho \bth),
\eqe
where $\Pi_0$ is the canonical embedding of $\R^{k}$ in $\R^d$ as $\R^k \times \set{0}^{d-k}$.
With these observation, we present the following lemma known as the \emph{Blaschke-Petkantschin} (BP) formula, that will be highly useful throughout the proofs in this paper. In \cite{bobrowski_random_2019} we presented this formula for any compact Riemannian manifold. Here, we will present its version for the flat torus. Throughout, we will always assume that if the diameter of $\bx$ is greater than $2\rmax$ then $f(\bx) = 0$.

\begin{lem}\label{lem:bp_torus}
Let $f:(\T^d)^{k+1}\to \R$ be a measurable bounded function satisfying \eqref{eq:func_inv}.
Then,
\eqb\label{eq:bp_torus_sep}
	\int_{(\T^d)^{k+1}} f(\bx) d\bx = \Abp\int_0^\infty \int_{(\S^{k-1})^{k+1}} \rho^{dk-1} f(\rho \bth)( \vsimp(\bth))^{d-k+1}d\bth d\rho ,
\eqe
where $\vsimp(\bth)$ is the volume of the $k$-simplex spanned by $\bth$, $\Abp = (k!)^{d-k+1}\Gamma_{d,k}$,  $\Gamma_{d,k} = \binom{d}{k}\frac{\Omega_d}{\Omega_k\Omega_{d-k}}$ is the volume of the Grassmannian $\Gr(d,k)$, and $\Omega_k = \prod_{i=1}^k\omega_i$. 

\end{lem}

\begin{proof}
This is just a special case of Lemma 5.1 in \cite{bobrowski_random_2019} proved for compact Riemannian manifolds.

\end{proof}

We will also need another special case of the BP formula, for integrating over point configurations on the $(k-1)$-sphere. 

\begin{lem}\label{lem:bp_sphere}
Let $f:(\S^{k-1})^k\to \R$ be a bounded measurable function. Then,
\eqb\label{eq:BP_sphere}
\int_{(\S^{k-1})^k} f(\bth)d\bth = \Abp^{s} \int_0^1 \int_{(\S^{k-2})^k}\rho^{k^2-2k} (1-\rho^2)^{-1/2}  f(\rho \bvphi)\vsimp(\bvphi)d\bvphi d\rho,
\eqe
where $\Abp^{\text{s}} = k! \omega_{k} $. 

\end{lem}

\begin{proof}
This is also a special case of Lemma 5.1 in \cite{bobrowski_random_2019}, but can be more easily derived from Theorem 3 in \cite{edelsbrunner_random_2018}.
In particular, this formula can be retrieved from equation (2.1) in \cite{edelsbrunner_random_2018} by taking the change of variable $t\to\rho^2$.

\end{proof}

Finally, in some cases we will need a ``partial" version of the BP formula. By partial we mean that for $\bx \in (\T^d)^{k+1}$ we keep $m+1$ variables unchanged, and use a BP-style change of variables for the remaining $(k-m)$ variables. 

Let $\bx = (x_1,\ldots, x_{k+1}) \in (\T^d)^{k+1}$, and define $\bx_0 = (x_1,\ldots,x_{m+1})$, $\hat\bx_0 =\bx\bs\bx_0 = (x_{m+2},\ldots,x_{k+1})$. The variables in $\bx_0$ will remain unchanged. Fix $\bx_0$ and
denote $c_0 = c(\bx_0), \rho_0 = \rho(\bx_0)$. 
Notice that if $c(\bx)$ and $\rho(\bx)$ are well-defined (see Section \ref{sec:flat_torus}), then $\bx_0$ lies on a smaller sphere than $\bx$, and therefore 
$c(\bx_0)$ and $\rho(\bx_0)$ are well-defined as well. 
Next, similarly to our discussion in Section \ref{sec:crit_faces}, we will consider  $\bx$ in the coordinate system $\R^d_{c_0}$.
Notice that $\bx_0$ spans a $m$-dimensional  plane $\Pi(\bx_0)\subset \R^d_{c_0}$. In that case, the center $c(\bx)$ must lie on the $(d-m)$-dimensional plane orthogonal to $\Pi(\bx_0)$ denoted $\Pi^{^\perp}(\bx_0)$. Once the center $c(\bx)$ is set, in order to determine the $k$-dimensional plane $\Pi(\bx)$, we choose a $\Pi_0^{^\perp}\in \Gr(k-m-1,d-m-1)$, so that $\Pi(\bx) = \Pi(\bx_0 \cup {c}) \oplus \Pi_0^{^\perp} \cong \R^k$. Notice that once we chose $\bx_0$ and $c(\bx)$ then $\rho = \rho(\bx)$ is determined by $\rho = \sqrt{\rho^2_0 + |c-c_0|^2}$. Therefore, in order to determine the location of the points $\hat\bx_0$ all that is remained is to choose their spherical coordinates $\hat\bth_0 \in (\S^{k-1})^{k-m}$. Overall, we obtained a change of variables $\bx \to (\bx_0, c, \Pi_0^{^\perp}, \hat\bth_0)$ which leads to the following statement.

\begin{lem}\label{lem:bp_partial}
Let $f:(\T^d)^{k+1}\to \R$, be a bounded measurable function. Suppose that $0\le m \le k-1$, then
\[
\int_{\bx} f(\bx)d\bx = \param{\frac{k!}{m!}}^{d-k+1}\int_{\bx_0}d\bx_0
\int_c dc\int_{\Pi_0^{^\perp}} d\Pi_0^{^\perp}\int_{\hat\bth_0} d\hat\bth_0 \frac{\rho^{(d-1)(k-m)}}{|c-c_0|^{(d-k)}}
\param{\frac{\vsimp(\bth)}{\vsimp(\bth_0)}}^{d-k+1}f(\bx),
\]
where $\bx_0 = (x_1,\ldots, x_{m+1})$, $c\in \Pi^{^\perp}(\bx_0) \cong \R^{d-m}$,  $\Pi_0^{^\perp} \in \Gr(k-m-1,d-m-1)$, and $\hat\bth_0 \in (\S^{k-1})^{k-m}$.
In addition, $\bth = (\theta_1,\ldots, \theta_{k+1}) =  \theta(\bx) \in (\S^{k-1})^{k+1}$, $\bth_0 = (\theta_1,\ldots, \theta_{m+1})$, $c_0 = c(\bx_0)$, and $\rho = \rho(\bx) =  \sqrt{\rho^2_0 + |c-c_0|^2}$.
\end{lem}

\begin{proof}
We will prove first the case $1\le m \le k-1$.
We first go back to raw form of the BP formula that appeared in Lemma 5.1 in \cite{bobrowski_random_2019}, stating that
\eqb\label{eq:BP_torus}
    \int_{\bx\in (\T^d)^{k+1}}f(\bx)d\bx = (k!)^{d-k+1} \int_c dc \int_{\Pi}d\Pi \int_\rho d\rho \int_{\bth}d\bth \rho^{dk-1} (\vsimp(\bth))^{d-k+1}f(\bx),
\eqe
where $c\in \T^d$, $\Pi\in \Gr(k,d)$,  $\rho \in [0,\infty)$,  $\bth \in (\S^{k-1})^{k+1}$, and $\bx$ on the right hand side is short for $c+\rho\bth(\Pi)$. 

Next, we also go back to the raw form of the spherical BP formula appearing in Theorem 3 of \cite{edelsbrunner_random_2018}. This states that if $g:(\S^{k-1})^{m+1}\to \R$, then
\[
\int_{(\S^{k-1})^{m+1}}d\bth_0 g(\bth_0) = (m!)^{k-m}\int_{\Pi'}d\Pi'\int_{p}dp\int_{\bvphi_0}d\bvphi_0  (1-|p|^2)^{\frac{m(k-1)}{2}-1}(\vsimp(\bvphi_0))^{k-m}g(\bth_0),
\]
where $\Pi'\in \Gr(m,k)$, $p\in B_1\cap (\Pi')^{^\perp}$ is a point in the $(k-m)$-unit ball in $(\Pi')^{^\perp}$, and $\bvphi_0 \in (\S^{m-1})^{m+1}$ are spherical coordinates in $\Pi'$. This way, $\bth_0 = (p+(1-|p|^2)^{1/2}\bvphi_0))$.
Using this in \eqref{eq:BP_torus} we have
\[
\splitb
\int_{\bx}f(\bx)d\bx &= (m!)^{k-m}(k!)^{d-k+1}\int_c dc   \int_{\Pi}d\Pi\int_\rho d\rho\int_{\hat\bth_0}d\hat\bth_0\int_{\Pi'}d\Pi'\int_p dp \int_{\bvphi_0}d{\bvphi_0}  \\
& \times\rho^{dk-1} (1-|p|^2)^{\frac{m(k-1)}{2}-1}  (\vsimp(\bvphi_0))^{k-m} (\vsimp(\bth))^{d-k+1}f(\bx),
\splite
\]
where $\bx = c+\rho\bth = (c_0 + \rho_0\bvphi_0(\Pi'), c+\rho\hat\bth_0(\Pi))$, and $c_0 = c+\rho p$, $\rho_0 = \rho(1-|p|^2)^{1/2}$.
Notice that integrating over  $(\Pi,\Pi')\in \Gr(k,d)\times\Gr(m,k)$ provides the $k$-plane $\Pi(\bx)$ and the $m$-plane $\Pi(\bx_0)\subset \Pi(\bx)$. We can switch the order of integration by taking $(\Pi_0, \Pi_0^{^\perp})  \in\Gr(m,d)\times\Gr(k-m,d-m)$, so that $\Pi' = \Pi_0$ and $\Pi = \Pi_0 \oplus \Pi_0^{^\perp}$. This yields,
\[
\splitb
\int_{\bx}f(\bx)d\bx &= (m!)^{k-m}(k!)^{d-k+1}\int_c dc\int_{\Pi_0}d\Pi_0\int_{\Pi_0^{^\perp}} d\Pi_0^{^\perp} \int_p dp  \int_\rho d\rho \int_{\hat\bth_0} d\hat\bth_0 \int_{\bvphi_0}d{\bvphi_0}\\
& \rho^{dk-1} (1-|p|^2)^{\frac{m(k-1)}{2}-1}  (\vsimp(\bvphi_0))^{k-m} (\vsimp(\bth))^{d-k+1}f(\bx),
\splite
\]
where $p\in B_1\cap  \Pi_0^{^\perp}$.
Next, notice that fixing $p$ and taking  $\rho\to \rho_0 (1-|p|^2)^{-1/2}$, then $d\rho = d\rho_0 (1-|p|^2)^{-1/2}$. In addition, fixing $\rho$ and $p$, recall that $c_0 = c+\rho p$, implying that we can use the change of variables $c\to c_0-\rho p$ with $dc = dc_0$.
With these changes, we have
\[
\splitb
\int_{\bx}f(\bx)d\bx 
 &= (m!)^{k-m}(k!)^{d-k+1}\int_{\Pi_0}d\Pi_0\int_{\Pi_0^{^\perp}} d\Pi_0^{^\perp}  \int_p dp \int_{\rho_0} d\rho_0 \int_{c_0} dc_0\int_{\hat\bth_0} d\hat\bth_0 \int_{\bvphi_0}d{\bvphi_0}\\
& \rho_0^{dk-1} (1-|p|^2)^{\frac{m(k-1)-dk}{2}-1}  (\vsimp(\bvphi_0))^{k-m} (\vsimp(\bth))^{d-k+1} f(\bx),
\splite
\]
where $\rho_0 \in [0,\infty)$, and $c_0\in \T^d$. Reordering the ingetral we have,
\eqb\label{eq:bp_two_sets}
\splitb
\int_{\bx}f(\bx)d\bx &= (k!/m!)^{d-k+1}(m!)^{d-m+1}\\
&\times\int_{c_0} dc_0\int_{\Pi_0}d\Pi_0 \int_{\rho_0} d\rho_0 \int_{\bvphi_0}d{\bvphi_0}
\rho_0^{dm-1}(\vsimp(\bvphi_0))^{d-m+1}\\
&
\times\int_{\Pi_0^{^\perp}} d\Pi_0^{^\perp}\int_p dp \int_{\hat\bth_0} d\hat\bth_0  
\rho_0^{d(k-m)}(\rho/\rho_0)^{dk-mk+m+2}  
\param{\frac{\vsimp(\bth)}{\vsimp(\bvphi_0)}}^{d-k+1}f(\bx) \\
&= (k!/m!)^{d-k+1}\int_{\bx_0}d\bx_0
\int_{\Pi_0^{^\perp}} d\Pi_0^{^\perp}\int_p dp \int_{\hat\bth_0} d\hat\bth_0 \\
&\times \rho_0^{d(k-m)} (\rho/\rho_0)^{dk-mk+m+2}  \param{\frac{\vsimp(\bth)}{\vsimp(\bvphi_0)}}^{d-k+1}f(\bx) \\
&= (k!/m!)^{d-k+1}\int_{\bx_0}d\bx_0
\int_{\Pi_0^{^\perp}} d\Pi_0^{^\perp}\int_p dp \int_{\hat\bth_0} d\hat\bth_0 \\
&\times
\rho_0^{d(k-m)}\param{{\rho}/{\rho_0}}^{d(k-m)+2}   \param{\frac{\vsimp(\bth)}{\vsimp(\bth_0)}}^{d-k+1}f(\bx),
\splite
\eqe
where in the last equality we used the fact that $\bvphi_0$ are coordinates in $\S^{m-1}$, and  $\bth_0 = p+  (\rho_0/\rho)\bvphi_0$ are on a sphere of radius $(\rho_0/\rho)$ inside $\S^{k-1}$. 
Therefore $\vsimp(\bth_0) = (\rho_0/\rho)^m\vsimp(\bvphi_0)$.

Next, recall that $\Pi_0^{^\perp}$ is a $(k-m)$-plane in $\Pi^{^\perp}(\bx_0) \cong \R^{d-m}$, and $p \in B_1 \cap \Pi_0^{^\perp}$. 
Since $p$ is in a $(k-m)$-plane, we can use polar coordinates $p\to  u\alpha$  where $u \in [0,1]$ and $\alpha\in \S^{k-m-1}$. This will result in 
\[
dp = u^{k-m-1} dud\alpha.
\]
Notice that
$
c_0-c = \rho p = \rho_0 (1-u^2)^{-1/2} u \alpha$. We will therefore take another change of variables $u(1-u^2)^{-1/2} \to v$, so $u = v(1+v^2)^{-1/2}$ and $du= (1+v^2)^{-3/2} dv$. Thus, we have
\[
dp = \param{\frac{v}{(1+v^2)^{1/2}}}^{k-m-1}(1+v^2)^{-3/2}  dv d\alpha,
\] 
where $v\in [0,\infty)$.
Since $(\rho/\rho_0)^2 = (1-u^2)^{-1} = 1+v^2$, we can write this as
\[
dp = v^{k-m-1}\param{\frac{\rho_0}{\rho}}^{k-m+2}  dv d\alpha.
\]
Since $\rho p = \rho_0 v \alpha$ we take $\rho_0 v \to w$, so that $\rho p = w\alpha$, $v = \rho_0^{-1}w$, and $dv = \rho_0^{-1}dw$. Then
\[
dp = w^{k-m-1}\param{\frac{\rho_0}{\rho}}^{k-m+2} \rho_0^{-(k-m)} dw d\alpha,
\]
where $w\in [0,\infty)$.
Putting this back into \eqref{eq:bp_two_sets}, we have
\[
\splitb 
\int_{\bx} f(\bx)d\bx &= (k!/m!)^{d-k+1}\int_{\bx_0}d\bx_0
\int_{\Pi_0^{^\perp}} d\Pi_0^{^\perp}\int_w dw\int_\alpha d\alpha \int_{\hat\bth_0} d\hat\bth_0 \\
&\times
w^{k-m-1}
\rho^{(d-1)(k-m)} (\vsimp(\bth)/\vsimp(\bth_0))^{d-k+1}f(\bx).
\splite
\]
Note that integrating over $(w,\S^{k-m-1}) \in \R^+\times \S^{k-m-1}$ is equivalent to integrating over $(w', L) \in \R\times \Gr(1,k-m)$. Next, integrating over $(\Pi_0^{^\perp}, L) \in \Gr(k-m, d-m) \times \Gr(1,k-m)$ we can switch the order to 
$(L, \Pi_0^{^\perp}) \in \Gr(1,d-m) \times \Gr(k-m-1,d-m-1)$.
Once we make this switch, we revert from $(w', L) \in \R\times \Gr(1,d-m)$ to $(w,\alpha) \in \R^+ \times \S^{d-m-1}$. This leads to,
\[
\splitb 
\int_{\bx} f(\bx)d\bx &= (k!/m!)^{d-k+1}\int_{\bx_0}d\bx_0
\int_w dw\int_\alpha d\alpha \int_{\Pi_0^{^\perp}} d\Pi_0^{^\perp}
\int_{\hat\bth_0} d\hat\bth_0 \\
&\times
w^{k-m-1}
\rho^{(d-1)(k-m)} (\vsimp(\bth)/\vsimp(\bth_0))^{d-k+1}f(\bx).
\splite
\]
The variables $w\in [0,\infty)$ and $\alpha \in \S^{d-m-1}$ are polar coordinates for a point in $z\in \Pi^{^\perp}(\bx_0)\cong\R^{d-m}$, where $dz = w^{d-m-1}dw d\alpha$. Thus, we can write
\[
\splitb 
\int_{\bx} f(\bx)d\bx &= (k!/m!)^{d-k+1}\int_{\bx_0}d\bx_0
\int_z dz
\int_{\Pi_0^{^\perp}} d\Pi_0^{^\perp}
\int_{\hat\bth_0} d\hat\bth_0 \\
&\times
|z|^{-(d-k)}
\rho^{(d-1)(k-m)} (\vsimp(\bth)/\vsimp(\bth_0))^{d-k+1}f(\bx).
\splite
\]
Finally, taking $z\to (c-c_0)$, completes the proof.

For $m=0$, we start with \eqref{eq:BP_torus}, and change the order of integration.
\[
\splitb
    \int_{\bx\in (\T^d)^{k+1}}f(\bx)d\bx &= (k!)^{d-k+1} \int_c dc\int_{\Pi}d\Pi \int_\rho d\rho \int_{\bth}d\bth \rho^{dk-1} (\vsimp(\bth))^{d-k+1}f(\bx)\\
    &=(k!)^{d-k+1} \int_c dc\int_{\Pi} d\Pi \int_{\theta_1} d\theta_1  \int_\rho d\rho \int_{\hat\bth_0}d\hat\bth_0 \rho^{dk-1} (\vsimp(\bth))^{d-k+1}f(\bx).
    \splite
\]
Similarly to what we did above, we can change the order of 
$(\Pi, \theta_1) \in \Gr(k,d) \times \S^{k-1}$ into $(\theta_1,\Pi_0^{^\perp}) \in \S^{d-1}\times \Gr(k-1,d-1)$, so that $\Pi = \Pi(\theta_1) \oplus \Pi_0^{^\perp}$.
Then we have
\[
    \int_{\bx\in (\T^d)^{k+1}}f(\bx)d\bx 
    =(k!)^{d-k+1} \int_c dc\int_\rho \rho^{d-1} d\rho\int_{\theta_1} d\theta_1\int_{\Pi_0^{^\perp}} d\Pi_0^{^\perp} \int_{\hat\bth_0}d\hat\bth_0 \rho^{d(k-1)} (\vsimp(\bth))^{d-k+1}f(\bx).
\]
Denoting $z = \rho\theta_1$, then $dz=\rho^{d-1}d\rho d\theta_1$, and thus
\[
 \int_{\bx\in (\T^d)^{k+1}}f(\bx)d\bx 
    =(k!)^{d-k+1} \int_c dc\int_z dz \int_{\Pi_0^{^\perp}} d\Pi_0^{^\perp}   \int_{\hat\bth_0}d\hat\bth_0 \rho^{d(k-1)} (\vsimp(\bth))^{d-k+1}f(\bx).
\]
Using the  fact that $x_1 = c+z$ we can switch $(c,z) \in \T^d\times \R^d $ into $(x_1, c) \in \T^d\times \Pi^{^\perp}(x_1)$, where $\Pi^{^\perp}(x_1) = \R^d_{x_1} \cong \R^d$. Therefore,
\[
 \int_{\bx\in (\T^d)^{k+1}}f(\bx)d\bx 
    =(k!)^{d-k+1} \int_{x_1} dx_1  \int_c dc \int_{\Pi_0^{^\perp}} d\Pi_0^{^\perp} \int_{\hat\bth_0}d\hat\bth_0 \rho^{d(k-1)} (\vsimp(\bth))^{d-k+1}f(\bx).
\]
Recalling that in this case $c_0 = x_1$ we have that $|c-c_0| = \rho$. In addition, we have $\vsimp(\bth_0)=1$, and we reached the stated result for $m=0$, concluding the proof.

\end{proof}

To conclude this section, since $\vsimp$ plays a key role in the BP-formula, we will provide the following useful estimates.

\begin{lem}\label{lem:simp_vol}
Let $\cX = \set{x_1,\ldots x_{k+1}}$ be a $k$-simplex, and $\hat \cX_i  = \cX \bs \set{x_i}$ be a $(k-1)$-face. Denote by $h_i$ the distance between $x_{i}$ and the affine plane $\Pi(\hat\cX_i)$. Then,
\[
\vsimp(\cX) = \frac{1}{k}h_i \vsimp(\hat\cX_i).
\]
\end{lem}

For example, for triangles we have  the well known formula $\mathrm{Area} = \mathrm{Length(base)}\times\mathrm{height}/2$, and for a tetrahedron we have  $\mathrm{Volume} = \mathrm{Area(base)}\times \mathrm{height}/3$.

\begin{proof}
Without loss of generality, suppose that $\cX \subset \R^k$, in which case we have (see \cite{stein_note_1966}),
\[
	\vsimp(\cX) = \frac{1}{k!} \abs{\det(x_2-x_1, x_3-x_1,\ldots, x_{k+1}-x_1)}.
\]
Suppose further that $i=k+1$.
Without loss of generality, since the volume is shift and rotation invariant, we can assume that $\hcX_{k+1}$ lies on $\R^{k-1}\times \set{0} \subset \R^k$, and $x_{k+1}$ is on the $k$-th axis $\set{0}^{k-1}\times \R$.
In this case we have $x_i^{(k)} - x_1^{(k)} = 0$ for all $i=1,\ldots, k$, and $x_{k+1}^{(k)} - x_1^{(k)} = x_{k+1}^{(k)} = h_{k+1}$.
Denote by $\hat x_i = (x_i^{(1)},\ldots, x_i^{(k-1)})$ (the projection of $x_i$ on $\R^{k-1}$). Then we have
\[
	\vsimp(\cX) = \frac{1}{k!} h_{k+1} \abs{\det(\hat x_2 - \hat x_1, \ldots, \hat x_k - \hat x_1)}	= \frac{(k-1)!}{k!} h_{k+1} \vsimp(\hcX_{k+1}).
\]
This completes the proof.

\end{proof}

Finally, we use the last Lemma to bound the ratio of volumes that shows up in Lemma \ref{lem:bp_partial}.

\begin{cor}\label{cor:ratio_vol}
For $k\ge 2$, let $\bth = (\theta_1, \ldots, \theta_{k+1})\in (\S^{k-1})^{k+1}$, and $\bth_m = (\theta_1,\ldots,\theta_{m})$, for $2\le m \le k$.
Then 
\[
	\vsimp(\bth) \le \vsimp(\bth_m).
\]
\end{cor}

\begin{proof}
For every $j$ define $h_j$ to be the distance between $x_{j+1}$ and $\Pi(\bth_j)$. Using Lemma \ref{lem:simp_vol} we have
\[
	\vsimp(\bth) = \frac{h_k}{k} \vsimp(\bth_k) = \frac{h_k}{k} \frac{h_{k-1}}{k-1} \vsimp(\bth_{k-1}) = \cdots = \frac{h_k}{k} \frac{h_{k-1}}{k-1}\cdots \frac{h_m}{m} \vsimp(\bth_m).
\]
Since $\bth\subset(\S^{k-1})^{k+1}$, we have that $h_j \le 2$ for all $j$.
Since $m\ge 2$ we have $\frac{h_k}{k} \frac{h_{k-1}}{k-1}\cdots \frac{h_m}{m}  \le 1$, that completes the proof.

\end{proof}
\newpage
\bibliographystyle{plain}
\bibliography{zotero}

\end{document}